\documentclass{article}

     \PassOptionsToPackage{numbers, compress, sort}{natbib}

\usepackage[preprint]{neurips_2025}

\usepackage[utf8]{inputenc} 
\usepackage[T1]{fontenc}    
\usepackage{hyperref}       
\usepackage{url}            
\usepackage{booktabs}       
\usepackage{amsfonts}       
\usepackage{nicefrac}       
\usepackage{microtype}      
\usepackage{xcolor}         

\usepackage{amsmath}
\usepackage{amssymb}
\usepackage{mathtools}
\usepackage{amsthm}
\usepackage{algorithm}
\usepackage{algorithmic}
\usepackage[capitalize,noabbrev]{cleveref}
\usepackage{wrapfig}
\crefname{equation}{}{}
\crefname{figure}{Figure}{Figures}
\creflabelformat{equation}{\textup{(#2#1#3)}}
\crefname{assumption}{Assumption}{Assumptions}
\crefname{condition}{Condition}{Conditions}

\theoremstyle{plain}
\newtheorem{theorem}{Theorem}[section]
\newtheorem{proposition}[theorem]{Proposition}
\newtheorem{lemma}[theorem]{Lemma}

\theoremstyle{definition}

\newtheorem{assumption}[theorem]{Assumption}
\theoremstyle{remark}
\newtheorem{remark}[theorem]{Remark}

\title{First-ish Order Methods: \\ Hessian-aware Scalings of Gradient Descent}

\author{%
  Oscar Smee
    \\
  School of Mathematics and Physics\\
  University of Queensland\\
  Brisbane, Australia\\
  \texttt{o.smee@uq.edu.au} \\
   \And
    Fred Roosta\\ 
    School of Mathematics and Physics\\
    ARC Training Centre for Information Resilience \\
    University of Queensland \\ 
    Brisbane, Australia \\
    \texttt{fred.roosta@uq.edu.au} \\
  \AND
  Stephen J. Wright \\
  Computer Sciences Department\\
  University of Wisconsin-Madison\\
  Madison, Wisconsin\\
  \texttt{swright@cs.wisc.edu}\\
}

\usepackage{enumitem}
\setlist[enumerate]{leftmargin=1.3em, itemindent=0em, itemsep=0em, topsep=0em, label = {\bfseries \arabic*.}} 
\setlist[itemize]{leftmargin=1.3em,itemindent=0em, itemsep=0em, topsep=0em}

\newcommand{\SPC}{\texttt{SPC}}
\newcommand{\LPC}{\texttt{LPC}}
\newcommand{\NC}{\texttt{NC}}

\usepackage{pifont}
\usepackage{xspace}

\renewcommand\th{\textsuperscript{th}\xspace}

\usepackage{accents}

\usepackage{stackengine}
\stackMath
\newcommand\tsup[2][2]{%
	\def\useanchorwidth{T}%
	\ifnum#1>1%
	\stackon[-.5pt]{\tsup[\numexpr#1-1\relax]{#2}}{\scriptscriptstyle\sim}%
	\else%
	\stackon[.5pt]{#2}{\scriptscriptstyle\sim}%
	\fi%
}
\makeatletter
\newcommand{\longdash}[1][2em]{%
	\makebox[#1]{$\m@th\smash-\mkern-7mu\cleaders\hbox{$\mkern-2mu\smash-\mkern-2mu$}\hfill\mkern-7mu\smash-$}}
\makeatother
\newcommand{\omitskip}{\kern-\arraycolsep}

\newcommand{\real}{\mathbb{R}}

\newcommand*\tageq{\refstepcounter{equation}\tag{\theequation}}
\newcommand {\Lg} { {L}_{\bgg} }
\newcommand {\Lt} {\tilde{L}}
\newcommand {\Lmax} {L^{\text{max}}}
\newcommand {\eg} {\varepsilon_g}

\newcommand{\sB}{\mathcal{B}}

\newcommand{\sD}{\mathcal{D}}

\newcommand{\sF}{\mathcal{F}}
\newcommand{\sG}{\mathcal{G}}

\newcommand{\sK}{\mathcal{K}}

\newcommand{\sI}{\mathcal{I}}

\renewcommand {\AA}  { {\mathbf{A}} }

\newcommand {\BB}  { {\mathbf{B}} }

\newcommand {\HH}  { {\mathbf{H}} }

\newcommand{\eye}{\mathbf{I}}

\renewcommand {\aa}  { {\bf a} }

\newcommand {\bgg}  { {\bf g} }

\newcommand {\yy}  { {\bf y} }
\newcommand {\hh}  { {\bf h} }

\newcommand {\rr}  { {\bf r} }

\newcommand {\pp}  { {\bf p} }

\newcommand {\vv}  { {\bf v} }

\newcommand {\xx}  { {\bf x} }

\newcommand {\zero}  { {\bf 0} }
\newcommand {\one}  { {\bf 1} }

\newcommand {\alphak}  { {{\alpha}_{k}} }

\newcommand {\HHk}  { {\HH_{k}} }

\newcommand {\bggk}  { {{\bgg}_{k}} }

\newcommand {\xxo}  { {{\xx}_{0}} }

\newcommand {\xxk}  { {{\xx}_{k}} }

\newcommand {\ppk}  { {{\pp}_{k}} }

\newcommand {\range}  { {\textnormal{Range}} }
\newcommand {\Null}  { {\textnormal{Null}} }

\newcommand{\cSPC}{c^{\text{SPC}}}

\newcommand{\cLPC}{c^\text{LPC}}
\newcommand{\cNC}{c^\text{NC}}
\newcommand{\pMR}{\pp^{\text{MR}}}
\newcommand{\pCG}{\pp^{\text{CG}}}
\newcommand{\pGM}{\pp^{\text{GM}}}

\newcommand{\sCG}{s^{\text{CG}}}
\newcommand{\sMR}{s^{\text{MR}}}
\newcommand{\sGM}{s^{\text{GM}}}
\newcommand{\sLPC}{s^\text{LPC}}
\newcommand{\sSPC}{s^{\text{SPC}}}
\newcommand{\sNC}{s^\text{NC}}

\newcommand{\st}{\tilde{s}}

\newcommand{\sNCt}{\tilde{s}^\text{NC}}

\newcommand{\ct}{\tilde{c}}
\newcommand{\cLPCt}{\ct^\text{LPC}}
\newcommand{\cNCt}{\ct^\text{NC}}
\newcommand{\cSPCt}{\ct^{\text{SPC}}}

\newcommand{\HHt}{\tilde{\HH}}

\newcommand{\gHg}{\dotprod{\bgg, \HH \bgg}}

\newcommand{\ppt}{\tilde{\pp}}

\makeatletter
\newcommand*{\transpose}{%
	{\mathpalette\@transpose{}}%
}
\newcommand*{\@transpose}[2]{%
	\raisebox{\depth}{$\m@th#1\intercal$}%
}
\makeatother

\renewcommand{\vec}[1]{\ensuremath{\mathbf{#1}}}
\newcommand{\grad}{\ensuremath {\vec \nabla}}

\newcommand{\defeq}{\triangleq}

\newcommand{\tmax}{\max}
\newcommand{\tmin}{\min}

\newcommand*\dotprod[1]{\left\langle #1\right\rangle}
\newcommand*\vnorm[1]{\left\| #1\right\|}
\newcommand*\abs[1]{\left| #1\right|}

\renewcommand {\Pr} { {\mathbb P} }

\newcommand*\bigO[1]{\mathcal O\left( #1\right)}

\begin{document}
\maketitle

\begin{abstract}
Gradient descent is the primary workhorse for optimizing large-scale problems in machine learning. However, its performance is highly sensitive to the choice of the learning rate.
A key limitation of gradient descent is its lack of natural scaling, which often necessitates expensive line searches or heuristic tuning to determine an appropriate step size.
In this paper, we address this limitation by incorporating Hessian information to scale the gradient direction. By accounting for the curvature of the function along the gradient, our adaptive, Hessian-aware scaling method ensures a local unit step size guarantee, even in nonconvex settings. 
Near a local minimum that satisfies the second-order sufficient conditions, our approach achieves linear convergence with a unit step size.
We show that our method converges globally under a significantly weaker version of the standard Lipschitz gradient smoothness assumption. 
Even when Hessian information is inexact, the local unit step size guarantee and global convergence properties remain valid under mild conditions.
Finally, we validate our theoretical results empirically on a range of convex and nonconvex machine learning tasks.

\end{abstract}

\section{Introduction}

Consider the optimization problem 
\begin{align}
\label{eqn:optimisation problem}
    \min_{\xx \in \real^d} f(\xx), 
\end{align}
where $f:\real^{d} \to \real$ is twice continuously differentiable, bounded below, and possibly nonconvex. 
Arguably, the simplest and most widely used workhorse for solving such problems in large-scale machine learning is gradient descent (GD) and its stochastic variants \cite{lanFirstorderStochasticOptimization2020}. 
Recall that an iteration of GD takes the form 
\begin{align*}
    \xx_{k+1} = \xx_k - \alpha_k \bgg_k,
\end{align*}
where $\bggk \defeq \bgg(\xxk) = \nabla f(\xxk)$ is the gradient of $f$ at $\xxk$ and $\alpha_k$ is a step size. Gradient descent is favored for its simplicity, low per-iteration cost, and convergence properties in nonconvex settings. With an arbitrary initialization, $\xxo$, the only hyperparameter is the step size, which may vary across iterations. Choosing this parameter has been the focus of significant research. Often GD is  analyzed using the Lipschitz smoothness of the gradient to select a constant step size that guarantees a monotonic decrease in function value. However, this strategy is typically impractical, as it relies on knowing the Lipschitz constant and results in overly conservative steps \citep{malitskyAdaptiveGradientDescent2020}. In modern machine learning, learning rates and schedules are usually tuned through trial and error. While this can deliver strong empirical performance, it incurs significant computational cost, lacks theoretical justification, and must be repeated whenever the model or dataset changes.

Backtracking line search is a principled and often more robust alternative \citep{nocedalNumericalOptimization2006} to constant step sizes and trial-and-error tuning. It comes with theoretical convergence guarantees, does not require problem-specific constants, and offers some adaptivity to local geometry \citep{fox2024glocal}, leading to a recent resurgence of interest in deep learning applications \citep{vaswaniPainlessStochasticGradient2021,galliDontBeMonotone2023}. A key challenge, however, lies in the choice of initial step size in the backtracking procedure. If it is too large, many backtracking steps are needed; if too small, the method may fail to explore larger step sizes, as illustrated in \cref{fig:logistic regression line search example}. To mitigate this, various initialization strategies have emerged, including heuristics \citep[Chapter 3]{nocedalNumericalOptimization2006}, resetting schemes \citep{vaswaniPainlessStochasticGradient2021}, and Polyak-based methods \citep{galliDontBeMonotone2023}, though the latter requires a knowledge of a lower bound on the objective function as well as an explicit limit on the step size.

\begin{figure}
    \centering
    \begin{minipage}{0.55\linewidth}
        \includegraphics[width=\linewidth]{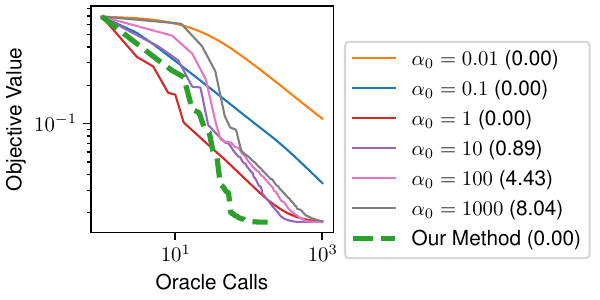}
    \end{minipage}%
    \hspace{4mm}
    \begin{minipage}{0.4\linewidth}
        \caption{Logistic regression on the Mushrooms dataset \citep{mushrooms}. We compare our scaling approach to standard GD with backtracking line search, using different initialization of  the backtracking procedure ($\alpha^0$), and assess performance based on oracle calls (\cref{apx:numerical-results-oracle-calls}). The average number of backtracking steps per iteration is reported in brackets.}
        \label{fig:logistic regression line search example}
    \end{minipage}
\end{figure}


In the case of line search applied to second-order methods, the situation is markedly different. These methods incorporate local curvature information to construct properly scaled search directions, for which a unit step size yields a sufficient decrease in the objective function\footnote{This brings to mind the following quote about the significance of a unit step size: \textit{“Any optimization algorithm for which the unit step length works has some wisdom. It is too much of a fluke if the unit step length [accidentally] works.”} - J. Nocedal (\href{https://www.stat.berkeley.edu/~mmahoney/talks/second_order_ml_sum17.pdf}{Source}.)}. This property makes second-order methods notably less sensitive to hyperparameter choices \citep{xu2020SecondOrderNonconvexEmpiricalStudy} and allows the unit step to serve as a natural initial guess in backtracking line search \citep{boydConvexOptimization2004,nocedalNumericalOptimization2006}. Indeed, practical experience suggests that backtracking is typically only required for a handful of iterations when using Newton's method. Unfortunately, second-order methods are often impractical for large-scale problems.

Together, these facts lead us to the motivating question for our paper: \emph{Can we employ local curvature to properly, and efficiently, scale the gradient itself?}
To that end, we propose our ``first-ish order'' method, which employs updates of the form:
\begin{align}
    \label{eq:method}
    \xx_{k+1} = \xx_k + \alphak \pp_k, \quad \pp_k = -s_k \bgg_k,
\end{align}
where $s_k > 0$ is an adaptive, \textit{Hessian-aware} scaling and $\alphak > 0$ is a step size chosen by line search, to ensure global convergence. 
Remarkably, we show that certain Hessian-aware scalings can endow GD with a local unit step size guarantee in a similar manner to second-order methods. Our approach is extremely simple and only requires access to local, directional curvature information, which can be calculated with a single Hessian-vector product.
The key to our results is the fact that the scaled gradient direction, $\ppk$, satisfies the following \textit{second-order descent} condition\footnote{Note that \cref{eqn:second-order descent} has also been called ``2.5\th order descent'' as it is a stronger condition than the typical reduction in the local quadratic approximation to the objective.}:
 \begin{align}
    \label{eqn:second-order descent}
    \dotprod{\bggk, \ppk} + \dotprod{\ppk, \HHk \ppk} \leq 0,
\end{align}
where $\HHk \defeq \HH(\xxk) = \nabla^{2} f(\xxk)$ is the Hessian of $f$ at $\xxk$. When curvature along $\pp$ is positive, \cref{eqn:second-order descent} strengthens the standard first-order descent condition $\dotprod{\bgg_k, \pp_k} < 0$. The condition \cref{eqn:second-order descent} is satisfied by search directions in second-order methods such as Newton’s method and its inexact variants (e.g., Newton-MR, Newton-CG) when non-positive curvature is not encountered \citep{liuNewtonMRAlgorithmComplexity2023,nocedalNumericalOptimization2006}. It avoids crude curvature bounds and enables second-order analysis without requiring problem-specific constants. Our method handles negative curvature by direct verification along the gradient direction, allowing for use in nonconvex settings without intrusive modification of the Hessian.


\paragraph{Contributions.} Our contributions are outlined as follows:
\begin{enumerate}
    \item In \cref{sec:global}, we show that when the gradient is small, our scaled gradient direction satisfies the Armijo condition with unit step size, even for nonconvex $f$. We also prove global convergence under generalized smoothness assumptions that extend beyond standard Lipschitz continuity of the gradient and Hessian. An extension to inexact Hessians is discussed in \cref{apx:inexact Hessian}.
    \item In \cref{sec:local}, we establish local convergence results for our method. Specifically, we show that under certain conditions near a minimum, the unit step size is consistently accepted by the line search condition, leading to local linear convergence of the function value. Furthermore, we show our method can potentially attain an improved rate over standard gradient descent. For a particular scaling, we also prove local linear convergence of the gradient norm. 
    \item Finally, in \cref{sec:numerical results}, we validate our results numerically on large scale, nonconvex objectives arising from problems in data science.
\end{enumerate}

\section{Hessian-aware Scaling} \label{sec:approach}
We now introduce and motivate our Hessian-aware scalings and  highlight their key properties.

\subsection{Our Approach}
\label{sec:hess_aware}
To determine the scaling, $s_k$, in \cref{eq:method} required for the update direction, we consider solving the Newton system subproblem with the search direction fixed to the negative gradient. We also account for varying curvature along the gradient direction, by considering the cases of strongly positive, limited positive, and negative curvature separately.

\paragraph{Strong Positive Curvature ($\SPC$).}

For strongly convex problems, the Newton direction minimizes the local quadratic approximation: $\min_{\pp} \dotprod{\pp,\bgg} + \dotprod{\pp, \HH \pp}/2$. Since our step direction is fixed to the negative gradient, we restrict $\pp = -s\bgg$, and, assuming $\gHg > 0$, optimize over $s$ to obtain
\begin{align}
    \label{eqn:CG scaling defn}
    \sCG = \frac{ \vnorm{\bgg}^2}{\langle \bgg, \HH \bgg \rangle}.
\end{align}
The notation $\sCG$ reflects a connection to the conjugate gradient (CG) method. Indeed, the CG scaled direction, $\pp =-\sCG \bgg$ is equivalent to the first inner iteration of Newton-CG, an inexact Newton method method based on applying CG to the Newton system \cite{nocedalNumericalOptimization2006}.
Alternatively, we can also view the Newton direction as the solution to the least squares problem $\min_{\pp} \vnorm{\HH \pp + \bgg}^{2}$. Applying the restriction, $\pp = -s\bgg$, and optimizing over $s$ yields an alternative scaling
\begin{align}
    \label{eqn:MR scaling defn}
    \sMR = \frac{\dotprod{\bgg, \HH \bgg} }{\vnorm{\HH \bgg}^2}.
\end{align}
The notation $\sMR$ highlights a connection to the minimum residual (MINRES) method \cite{paige1975solution}. In fact, $\pp = -\sMR\bgg$ is the search direction produced on the first inner iteration of Newton-MR, an inexact Newton method which employs MINRES as a subproblem solver \cite{roostaNewtonMRInexactNewton2022, liuNewtonMRAlgorithmComplexity2023,smeeInexactNewtontypeMethods2024,limComplexityGuaranteesNonconvex2024}. The link between the CG and MR scalings and inexact second-order methods is an impetus for our study of scaled gradient methods. Indeed, our curvature validation was inspired by the treatment of limited positive curvature in \cite{smeeInexactNewtontypeMethods2024, limComplexityGuaranteesNonconvex2024}.

From these two fundamental scalings, additional factors like the geometric mean of CG and MR scalings can be derived:
\begin{align}
    \label{eqn:GM scaling defn}
    \sGM = \sqrt{\sMR\sCG} = \frac{\vnorm{\bgg}}{\vnorm{\HH \bgg}}.
\end{align}
The CG, GM, and MR scalings each reflect an aspect of the inverse Hessian restricted to one dimension. For example, the CG scaling in \cref{eqn:CG scaling defn} is an inverse Rayleigh quotient of the Hessian, $\HH$, with respect to $\bgg$, while the MR scaling in \cref{eqn:MR scaling defn} corresponds to a Rayleigh quotient of the Hessian \textit{pseudo-inverse}, $\HH^\dagger$, with respect to $\HH\bgg$ (using $\HH \HH^\dagger \HH = \HH$).
In the univariate case the update in \cref{eq:method} collapses to the exact Newton step for all scalings. Thus, these scalings are most effective under $\SPC$ along the gradient direction, i.e., $\gHg > \sigma \vnorm{\bgg}^2$ for some $\sigma > 0$. The $\SPC$ condition can be interpreted as a strong convexity condition \textit{restricted to the gradient direction}.

\paragraph{Negative Curvature ($\NC$).}
In nonconvex settings, the indefiniteness of the Hessian can make the gradient a $\NC$ direction, i.e., $\gHg < 0$. Since $-\bgg$ is already a descent direction, both the first and second directional derivatives 
along $-\bgg$ are negative, regardless of scaling choice. Previous theoretical results \cite{gouldExploitingNegativeCurvature2000,curtisExploitingNegativeCurvature2019,liuNewtonMRAlgorithmComplexity2023} and practical experience suggest that substantial progress can be made with large steps along negative curvature directions. 

\paragraph{Limited Positive Curvature ($\LPC$).}
The $\LPC$ case, where $0 \leq \gHg \leq \sigma \vnorm{\bgg}^2$, represents a middle ground between $\SPC$ and $\NC$, characterized by small but non-negative curvature along $\bgg$. Here, $\sigma$ acts like a gradient Lipschitz constant, constraining the second-order term in the Taylor expansion along $-\bgg$. Since second-order information is unreliable in this regime, we revert to gradient descent with a step size independent of second-order information. Thanks to the curvature bound $\sigma$, scalings satisfying $s \leq 1/\sigma$ maintain second-order descent properties akin to those in the $\SPC$ and $\NC$ cases. 

The above discussions are summarized in \cref{alg:scaling selection}. At first glance, \cref{alg:scaling selection} comes with a number of hyperparameters. However most of these hyperparameters come with natural settings and do not require tuning. We detail these settings in \cref{apx:scaled gradient parameters discussion}.
\begin{algorithm}[ht]
    \begin{algorithmic}[1]
        \STATE \textbf{Inputs}: Gradient $\bgg$, Hessian $\HH$, and $\SPC$ scaling tolerance $\sigma > 0$.
        \vspace{1mm} \STATE Set the range for $\LPC$ scaling as $\sLPC_{\tmin} \in (0, 1/\sigma)$.
        \vspace{1mm} \STATE Set the range for $\NC$ scaling as $ 0 < \sNC_{\tmin} \leq \sNC_{\tmax} < \infty$.
        \IF{ $\dotprod{ \bgg, \HH \bgg}  > \sigma \vnorm{\bgg}^2$ }
            \vspace{1mm} \STATE  choose $\sSPC \in \{ \sMR, \sCG, \sGM \}$, set $\pp = -\sSPC \bgg$, set $\texttt{FLAG} = \SPC$.
        \vspace{1mm} \ELSIF{ $0 \leq \dotprod{ \bgg, \HH \bgg} \leq \sigma \vnorm{\bgg}^2$ }
            \vspace{1mm} \STATE choose $\sLPC \in [\sLPC_{\tmin}, 1/\sigma]$, set $\pp = -\sLPC \bgg$, set $\texttt{FLAG} = \LPC$.
        \vspace{1mm} \ELSE 
            \vspace{1mm} \STATE choose $\sNC \in [\sNC_{\tmin}, \sNC_{\tmax}]$, set $\pp = - \sNC \bgg$, set $\texttt{FLAG} = \NC$.
        \vspace{1mm} \ENDIF  
        \STATE \textbf{Return} $\pp$, $\texttt{FLAG}$.
    \end{algorithmic}
    \caption{Hessian-aware Scaling Selection}
    \label{alg:scaling selection}
\end{algorithm}

\begin{remark}
Computationally, \cref{alg:scaling selection} requires one additional Hessian-vector product, $\HH \bgg$, which can be computed efficiently without forming the full Hessian using automatic differentiation \cite{pearlmutterFastExactMultiplication1994, baydinAutomaticDifferentiationMachine2018}. This involves an extra forward and backward pass through the computational graph and some additional memory overhead. While this adds some cost to each iteration, our experiments (\cref{sec:numerical results}) show that our scaling method largely eliminates the need for backtracking in line search, reducing the total number of function evaluations (forward passes) compared to unscaled line search.
\end{remark}

\paragraph{Basic Properties} We now collect some basic properties of the scalings produced by \cref{alg:scaling selection}. We relegate all proofs to \cref{apx:proof-basic-properties}. 


\begin{proposition}[Scaling Upper Bounds] \label{prop:scaling upper bounds}
    In the $\SPC$ case, $0< \sMR \leq \sGM \leq  \sCG \leq 1/\sigma$, while for the $\LPC$ case, $0 < \sLPC \leq 1/\sigma$. Finally, in the $\NC$ case, $0 < \sNC \leq \sNC_{\max}$.
\end{proposition}
The proof of this result demonstrates the importance of validating the $\SPC$ condition to obtain an upper bound in the $\SPC$ and $\LPC$ cases. 


\begin{proposition}[Second-order Descent]\label{prop:second-order descent}
    Suppose $\bgg \neq 0$. The direction $\pp$ returned by \cref{alg:scaling selection} satisfies both the first-order descent condition $\dotprod{\bgg, \pp} < 0$ and the second-order descent condition \cref{eqn:second-order descent}.
\end{proposition}

\begin{remark}
Our selection of scalings satisfying second-order descent is not exhaustive. For instance, if $f$ has an $\Lg$-Lipschitz gradient, $s = 1/\Lg$ ensures \cref{eqn:second-order descent} by conservatively controlling curvature. However, this choice is not adaptive to local curvature, and $\Lg$ may be unknown.
\end{remark}

A well-known property of the classical Newton step is its invariance under affine transformations \citep[Chapter 9.5]{boydConvexOptimization2004}. While directions parallel to gradient cannot achieve full affine invariance for general objectives, \cref{prop:scalar invariance} shows that our Hessian-aware scalings lead to invariance to \textit{scalar} transformations.
\begin{proposition}[Scalar Invariance] 
\label{prop:scalar invariance} 
Consider \cref{eq:method} with $s_k = \sSPC_k$ for all $k$, applied to $f(\xx)$ and its scalar reparameterization $f(\yy)$ where $\yy =  \xx/c$ for any $c \neq 0$. Then $\yy_{k} = \xxk/c$ for all $k$.
\end{proposition}
As shown in \cref{prop:scalar invariance}, our method is fully invariant to scalar coordinate transformations in the strongly convex setting, where $\SPC$ scalings are used at each step. This contrasts with standard gradient descent, which is highly sensitive to such changes.


\subsection{Related Works}
\label{sec:related_work}
\paragraph{Quadratic Problems.} 
The scalings \cref{eqn:CG scaling defn,eqn:MR scaling defn} are well-studied for strongly convex quadratic functions; see \cite{gonzagaSteepestDescentAlgorithm2016,macdonaldFamilyRelaxedGradient2024} and references therein. In fact, the scalings correspond to exact line search methods along the gradient direction for $f$ and $\vnorm{\bgg}^2$, known as steepest descent and minimal gradient methods, respectively\footnote{Beyond quadratics, there is no correspondence to exact line search, so we avoid labeling \cref{eqn:CG scaling defn,eqn:MR scaling defn} as such.}. The GM scaling \cref{eqn:GM scaling defn} has also been studied for such problems by \cite{daiNewGradientMethod2006}, showing that for $\alpha > 0$, $\sGM$ estimates an inverse gradient Lipschitz constant.


\vspace{-1mm}
\paragraph{Barzilai-Borwein Methods.} 
The Barzilai-Borwein (BB) method \cite{barzilaiTwopointStepSize1988, fletcherBarzilaiBorweinMethod2005, daiFamilySpectralGradient2019} is also based on scalar minimization of a second-order approximation. For quadratics, the long and short BB step sizes are equivalent to our CG and MR scalings, shifted by one iteration, due to their equivalence to steepest descent and minimal gradient methods. While BB achieves R-linear convergence for strongly convex quadratics, its convergence for non-quadratic objectives cannot be guaranteed \citep{burdakov2019stabilized}.


\vspace{-1mm}
\paragraph{Smoothness and Adaptivity.} 
The global convergence analysis in this paper does not rely on the conventional Lipschitz gradient smoothness assumption, which has come under increasing scrutiny in recent works. Many machine learning objectives, such as feedforward or recurrent neural networks, fail to satisfy this condition \cite{patelGlobalConvergenceStability2022}. Recent studies suggest that this assumption may not hold even along the optimization trajectory \cite{cohenGradientDescentNeural2022}. Instead, \citet[Remark 5]{ahnUnderstandingUnstableConvergence2022} argue that {\em Lipschitz Hessian smoothness} is more appropriate; our work adopts a weaker form of this assumption (\cref{ass:Hessian gradient smoothness condition}). In line with these ideas, \citet{patelGradientDescentAbsence2023} study gradient descent with diminishing step sizes under local Lipschitz smoothness, a weaker assumption than the standard one.

Like our method, a number of works have incorporated local information into step size selection, e.g., the Polyak step size \citep{polyakIntroductionOptimization1987} and its variants \citep{loizouStochasticPolyakStepsize2021,orvieto2022dynamics,oikonomou2024stochastic}. \citet{malitskyAdaptiveGradientDescent2020,malitskyAdaptiveProximalGradient2024} adaptively set step sizes using local Lipschitz estimates and control sequences, ensuring global convergence under convexity and locally Lipschitz gradients. \citet{mishkinDirectionalSmoothnessGradient2024} explores adaptive step sizes via directional smoothness, closely related to our use of gradient curvature. \citet{grimmerProvablyFasterGradient2024} show that larger step sizes can improve rates, while \citet{altschulerAccelerationStepsizeHedging2023} demonstrate that combining long and short steps enhances convergence in convex settings. \citet{berahasNonUniformSmoothnessGradient2023} investigates local first-order smoothness oracles, while methods like D-adaptation \cite{defazioLearningRateFreeLearningDAdaptation2023,mishchenkoProdigyExpeditiouslyAdaptive2024} and Do(W)G \cite{ivgiDoGSGDsBest2023,khaledDoWGUnleashedEfficient2024} achieve parameter-free convergence by adaptively estimating problem constants.

\vspace{-1mm}
\paragraph{Quadratic Model Scaling } 
Utilizing a quadratic model to rescale search directions is not new. The KFAC method \citep{martensOptimizingNeuralNetworks2020} rescales directions using a quadratic model based on the Fisher information matrix. Similarly,  \citet{rouletSteppingEdgeCurvature2024} adaptively set the step size using a quadratic model of the objective along a search direction, akin to our CG scaling. \citet{rouletSteppingEdgeCurvature2024} show that curvature-aware scaling can align optimization dynamics with the edge of stability \citep{cohenGradientDescentNeural2022}, though they do not exploit negative curvature or provide theoretical guarantees. \citet{casteraSecondOrderStepSizeTuning2022} use a quadratic model (similar to our CG scaling) to verify curvature, enabling BB step sizes in nonconvex and stochastic settings, and also find that large step sizes are viable with negative curvature. \citet{degournay2022adaptivescalinglearningrate} rescale step directions using directional second-order information, applying a moving average of directional curvatures, but focus on practical implementation without convergence theory.

\section{Convergence Analyses} \label{sec:convergence-analysis}
In \cref{sec:global}, we study the global behavior of \cref{eq:method} with \cref{alg:scaling selection} and the classical Armijo line search; local convergence is addressed in \cref{sec:local}. Proofs are deferred to \cref{apx:proof-global-convergence}. We note that an extension to the case where the Hessian is only known inexactly is provided in \cref{apx:inexact Hessian}.

\subsection{Unit Step Size Guarantee and Global Convergence}
\label{sec:global}

To globalize the iteration in \cref{eq:method}, we select a step size via the Armijo condition, which requires $\alpha$ to satisfy a sufficient reduction in the function value
\begin{align}
    \label{eqn:armijo condition} 
    f(\xx + \alpha \pp) \leq f(\xx) + \rho \alpha \dotprod{\pp, \bgg}, 
\end{align}
for some constant $\rho \in (0,1/2)$.
The resulting method is depicted in \cref{alg:scaled gradient}, where backtracking (\cref{alg:back tracking line search} in \cref{sec:appendix:linesearch_algs}) and forward tracking (\cref{alg:forward tracking line search} in \cref{sec:appendix:linesearch_algs}) strategies are used to find $\alphak$ satisfying \cref{eqn:armijo condition}. 
\begin{algorithm}[ht]
    \begin{algorithmic}[1]
        \STATE \textbf{Input}: Line search parameter $\rho < 1/2$, termination tolerance $\eg > 0$, initialization $\xxo$, $k=0$.
        \vspace{1mm}\WHILE{$\vnorm{\bgg_k} \geq \eg$}
        \vspace{1mm}\STATE $[\ppk, \texttt{FLAG}]  \leftarrow$ Call \cref{alg:scaling selection} with $\HHk$, $\bggk$ and $\sigma$
        \vspace{1mm}\STATE For $\texttt{FLAG} = \SPC/\LPC$, use \cref{alg:back tracking line search} with $\alpha^0=1$ to find $\alpha_k \in (0,1]$ satisfying \cref{eqn:armijo condition}. For $\texttt{FLAG} = \NC$, use \cref{alg:forward tracking line search} with $\alpha^0=1$ to find $\alpha_k \in (0,\infty)$ satisfying \cref{eqn:armijo condition}. 
        \vspace{1mm}\STATE $\xx_{k+1} = \xx_k +\alpha_k \pp_k$
        \vspace{1mm}\STATE $k \gets k+1$
        \ENDWHILE
    \end{algorithmic}
    \caption{Scaled Gradient Descent With Line Search}
    \label{alg:scaled gradient}
\end{algorithm}

The primary assumption our analysis is built on is a weakened version of the typical Lipschitz Hessian smoothness condition, which requires smoothness to hold only along the negative-gradient direction.

\begin{assumption}[Hessian Directional Smoothness] \label{ass:Hessian smoothness condition}
    There exists, $0 \leq L_2 < \infty$ such that, $\forall \xx \in \real^d$ and $\forall t\geq 0$ we have $
        \vnorm{\HH\left(\xx - t\bgg(\xx)\right) - \HH(\xx) } \leq t L_2 \vnorm{\bgg(\xx)}$.
\end{assumption}
\begin{remark}
    While \cref{ass:Hessian smoothness condition} is difficult to verify directly, it is implied by $L_H$-Hessian Lipschitz continuity. Indeed, $L_2 \leq L_H$ since \cref{ass:Hessian smoothness condition} applies along a single direction at each $\xx$.
\end{remark}
Our main result states that, under only \cref{ass:Hessian smoothness condition}, the unit step size, $\alpha =1 $, locally ensures sufficient decrease in the function value for the scaled gradient direction, $\pp$, produced by \cref{alg:scaling selection}.

\begin{theorem}[Sufficient Condition for Acceptance of Unit Step Size] \label{thm:unit step size acceptance}
    Consider \cref{ass:Hessian smoothness condition} and let $\pp$ be the direction selected by \cref{alg:scaling selection}. The Armijo condition \cref{eqn:armijo condition} is satisfied with $\alpha=1$ if 
    \begin{align}
        \label{eqn:gradient line search termination}
        \vnorm{\bgg}  \leq \min\left\{\frac{6\sigma^2(1/2 - \rho)}{L_2}, \frac{6(1-\rho)}{L_2 (\sNC_{\tmax})^2} \right\}.
    \end{align}
\end{theorem}
\cref{thm:unit step size acceptance} shows that our gradient scaling is ‘natural,’ as the unit step eventually provides sufficient descent—an attribute typical of Newton-type methods \citep{boydConvexOptimization2004,nocedalNumericalOptimization2006}. The bound in \cref{eqn:gradient line search termination} is a worst-case result and may differ from typical behavior. In practice, as seen in \cref{sec:numerical results}, line search usually accepts $\alpha = 1$ for most iterations. Moreover, \cref{apx:curvature condition extension} shows that our $\SPC$ steps locally satisfy a \textit{curvature condition}, ensuring that the unit step length produces sufficiently large update directions.

Under the standard Lipschitz gradient assumption (with no need for \cref{ass:Hessian smoothness condition}), global convergence of our method follows straightforwardly from the application of line search to the gradient, as the scaling factor, $s$, is bounded below (a consequence of \cref{lemma:scaling lower bounds}). However, a more careful analysis reveals that, through our explicit control of the curvature term via \cref{eqn:second-order descent}, we can also guarantee convergence under \cref{ass:Hessian smoothness condition} and a weaker form of gradient Lipschitz smoothness.
\begin{assumption}[Hessian-gradient Directional Smoothness] 
    \label{ass:Hessian gradient smoothness condition} There exists $0 \leq L_1 < \infty$ such that
    for all $\xx \in \real^d$, if $\dotprod{\bgg(\xx), \HH(\xx) \bgg(\xx)} > 0$, then $ \vnorm{\HH(\xx) \bgg(\xx)} \leq L_1 \vnorm{\bgg(\xx)}$.
\end{assumption}

\begin{remark}
For twice continuously differentiable functions, \cref{ass:Hessian gradient smoothness condition} relaxes the standard $\Lg$-Lipschitz gradient smoothness assumption by requiring regularity only along $\bgg$. Indeed, $L_1 \leq \Lg$. Additionally, the concept of {\em moral smoothness}, introduced by \citet[Assumption 2]{roostaNewtonMRInexactNewton2022}, is strictly weaker than Lipschitz gradient and Hessian assumptions on any sublevel set of the gradient norm, yet it still implies \cref{ass:Hessian gradient smoothness condition}; see  \citet[Lemma 2]{roostaNewtonMRInexactNewton2022}.  
\end{remark}
\begin{proposition}[Global Convergence] \label{prop:global convergence}
    Consider \cref{ass:Hessian gradient smoothness condition,ass:Hessian smoothness condition} and suppose $f$ is lower bounded. For any $0 < \eg < 1$, after at most $K \in \bigO{\eg^{-2}}$  iterations of \cref{alg:scaled gradient}, we have $\vnorm{\bgg_k} \leq \varepsilon_{\bgg}$ for some $0\leq k\leq K$.
\end{proposition}
The detailed complexity bound in \cref{prop:global convergence}, including all underlying constants, is provided in \cref{apx:proof-global-convergence}.
Unsurprisingly, since the search direction is entirely determined by the gradient, the rate in \cref{prop:global convergence} matches that of gradient descent under the Lipschitz gradient smoothness condition with a line search or fixed step size \cite{cartisEvaluationComplexityAlgorithms2022,nesterovIntroductoryLecturesConvex2004}. Furthermore, since our method coincides with the Newton's method when $d=1$, it is not surprising that the convergence rate also matches that of Newton's method in the nonconvex setting with the stronger Lipschitz gradient and Hessian smoothness conditions \citep{cartisEvaluationComplexityAlgorithms2022}.
The novelty of \cref{prop:global convergence} lies in achieving this complexity under the alterative smoothness conditions, i.e.,  \cref{ass:Hessian gradient smoothness condition,ass:Hessian smoothness condition}.

\subsection{Local Convergence}
\label{sec:local}
Next we consider the local convergence properties of our method. Improved local convergence rates are often considered to be the primary appeal of second order method \citep{nocedalNumericalOptimization2006,roosta2019subsampledNewton}.
Let $\xx^{\star}$ be a local minima satisfying the second-order sufficient condition $\bgg(\xx^{\star}) = 0$ and $\HH(\xx^{\star}) \succ 0$. 
By the continuity of the Hessian, there exists a ball of radius $r$ around $\xx^{\star}$, denoted by $\sB_r^\star$, such that  
\begin{align} 
\label{eqn:second-order sufficient conditions}
    0 \hspace{-0.5mm} < \hspace{-0.5mm} \mu \hspace{-0.5mm}\defeq \hspace{-0.5mm}\min_{\xx \in \sB_r^\star} \lambda_{\tmin} (\HH) \hspace{-0.5mm} \leq \hspace{-0.5mm} \max_{\xx \in \sB_r^\star} \lambda_{\tmax}(\HH) \hspace{-0.5mm} \defeq \hspace{-0.5mm} M \hspace{-0.5mm}< \hspace{-0.5mm} \infty. 
\end{align}
Our next result demonstrates that, near $\xx^\star$, the unit step size is acceptable to line search for all iterations, leading to a linear decrease in the sub-optimality of the objective value.
\begin{theorem} \label{thm:local convergence}
    Consider \cref{ass:Hessian smoothness condition} and suppose $\xx^{\star}$ is a local minimum satisfying the second order sufficient conditions. If $\xx_0$ is sufficiently close to $\xx^{\star}$, then for all iteration of \cref{alg:scaled gradient}, the unit step size $\alphak=1$ satisfies the Armijo condition \cref {eqn:armijo condition} and  we have 
    \begin{align*}
        f(\xx_{k+1}) - f(\xx^{\star}) \leq (1-\tau)(f(\xx_k) - f(\xx^{\star})),
    \end{align*}
    where $\tau \defeq 2 \rho \mu \max\{1/M, \sLPC_{\tmin}\} \in (0, 1]$ if $\sigma \geq \mu$, and $\tau \defeq 2\rho \mu/M \in (0, 1]$ otherwise.
\end{theorem}

\begin{remark}
    Suppose $\xx\in \sB_r^\star$ where $\sB_r^\star$ is as in \cref{eqn:second-order sufficient conditions}. Since $\dotprod{ \bgg, \HH\bgg} \leq  M \vnorm{\bgg}^2$, the $\SPC$ case only arises on $\sB_r^\star$ if $\sigma$ is chosen such that $\sigma \leq M$; otherwise \cref{alg:scaling selection} always returns the $\LPC$ scaling. In this case, since $\sLPC_{\tmin} < 1/\sigma < 1/M$, and $\rho<1/2$ we always have $0 < \tau \leq 1$ in \cref{thm:local convergence}.
\end{remark}
\cref{thm:local convergence} suggests that the rate of convergence for scaled gradient descent has a similar dependence on the local condition number $\kappa \defeq M/\mu$ as that of the typical gradient descent. However, as the proof of \cref{thm:local convergence} reveals, this worst-case analysis overlooks the potential for scaled gradient methods to exploit larger scalings by adapting to local geometry. For instance, if $\sigma < \mu$ then $\sSPC \leq 1/\mu$. Therefore, in ``best-case'' iterations where large scalings (close to $1/\mu$) pass the line search with unit step size, the linear rate can be as small as $1 - 2\rho$ (cf. \cref{eqn:local function sub-optimality recursion}). This rate eliminates dependence on the local condition number, resembling the problem-independent convergence rates of many Newton-type methods \cite{roosta2019subsampledNewton, roostaNewtonMRInexactNewton2022}. This reflected in the practical performance of our algorithm. In particular, in \cref{sec:numerical results}, we demonstrate numerically that scaled gradient methods often produce large scalings that pass the line search with unit step size, leading to rapid convergence. Conversely, if $\sigma \gg \mu$, both $\LPC$ and $\SPC$ scalings are upper bounded by $1/\sigma \ll 1/\mu$, limiting the local rate.

Our next result states that the MR scaling \cref{eqn:MR scaling defn} can give rise to linear convergence in the gradient norm. 
Intuitively, this result arises because the MR scaling minimizes the norm of the residual of the Newton system associated with the scaled gradient, which can also be viewed as the norm of the linearized gradient.

\begin{theorem} \label{thm:local convergence MR second-order sufficient}
    Consider \cref{ass:Hessian smoothness condition} and suppose $\xx^{\star}$ is a local minimum satisfying the second order sufficient conditions. If $\xx_0$ is sufficiently close to $\xx^{\star}$, then the iterations of the form $\xx_{k+1} = \xx_k - \sMR_{k} \bgg_k$ converge linearly in the gradient norm.
\end{theorem}
Our proof of this result (\cref{apx:convergence in gradient norm}), extends beyond second-order sufficient conditions to a more general setting.
An immediate implication of \cref{thm:local convergence MR second-order sufficient} is that the gradient norm serves as a ``secondary objective'' for the MR scaling. This is notable given recent work on gradient norm regularization, which biases gradient descent toward regions with small gradient norms, often linked to ``flat'' regions and better generalization in machine learning \cite{barrettImplicitGradientRegularization2022,smithOriginImplicitRegularization2021,zhaoPenalizingGradientNorm2022,karakidaUnderstandingGradientRegularization2023,hochreiterFlatMinima1997,keskarLargeBatchTrainingDeep2017}. Unlike explicit regularization, which requires additional hyperparameter tuning, MR scaling implicitly achieves this bias without additional considerations.

\section{Numerical Results} \label{sec:numerical results}

We now proceed to validate our results by comparing our scaled gradient method against multiple variants of GD including fixed step size, Nesterov acceleration \citep{nesterovIntroductoryLecturesConvex2004}, heavy ball momentum \citep{wrightOptimizationDataAnalysis2022}, and Adam \citep{kingma2014adam}. In addition, we compare against hyperparameter free methods such as line search GD, with a number of step size resetting techniques inspired by \citep{vaswaniPainlessStochasticGradient2021}, as well as a deterministic variant of Polyak non-monotone (PoNo) Line Search \citep{galliDontBeMonotone2023} (see \cref{apx:competitor algorithms} for details). For visual clarity we report the best performing resetting scheme as ``line search'' and defer PoNo line search comparison to \cref{apx:additional-numerical-results}.

For our scaling method, we consider the CG, MR and GM scalings. However, in the quadratic case, it is well known that steepest descent and minimal gradient (our CG and MR resp.) step sizes are prone to a ``zig-zag effect'' \citep{nocedalNumericalOptimization2006,daiAlternateMinimizationGradient2003}. \citet{daiAlternateMinimizationGradient2003} theoretically and empirically demonstrate that this zig-zag effect can be mitigated by alternating between steepest descent and minimal gradient steps. For this reason, we also consider alternating scalings on each $\SPC$ iteration\footnote{If an $\LPC/\NC$ arises, we proceed based on the last $\SPC$ step.}. Alternating scalings are denoted by a concatenated string of the two scalings, in order, e.g. ``MRCG''.

In the main body, we report the best performing scaling, relegating performance comparison between scalings to \cref{apx:additional-numerical-results}.
Where available, we apply theoretically convergent hyperparameter values for our competitor methods, otherwise we manually tune hyperparameter settings.
We note that tuning incurs a significant—and often overlooked—cost. For example, tuning Adam on our logistic regression CIFAR10 task took roughly 20 times longer than a CGMR run. 

To fairly compare methods with different per-iteration costs, we plot objective values against the number of \textit{oracle calls}—equivalent function evaluations (see \cref{apx:numerical-results-oracle-calls} for details). For completeness, wall-clock time results are provided in \cref{apx:additional-numerical-results}. The setup for each problem and optimizer is detailed there as well. All methods are implemented in PyTorch \citep{pytorch} and run on various cluster GPUs (e.g., H100, A100). A link to our code is available in \cref{apx:additional-numerical-results}.

\paragraph{Multi-class Logistic Regression.}

\begin{figure*}[ht]
    \centering
    \includegraphics[width=0.8\textwidth]{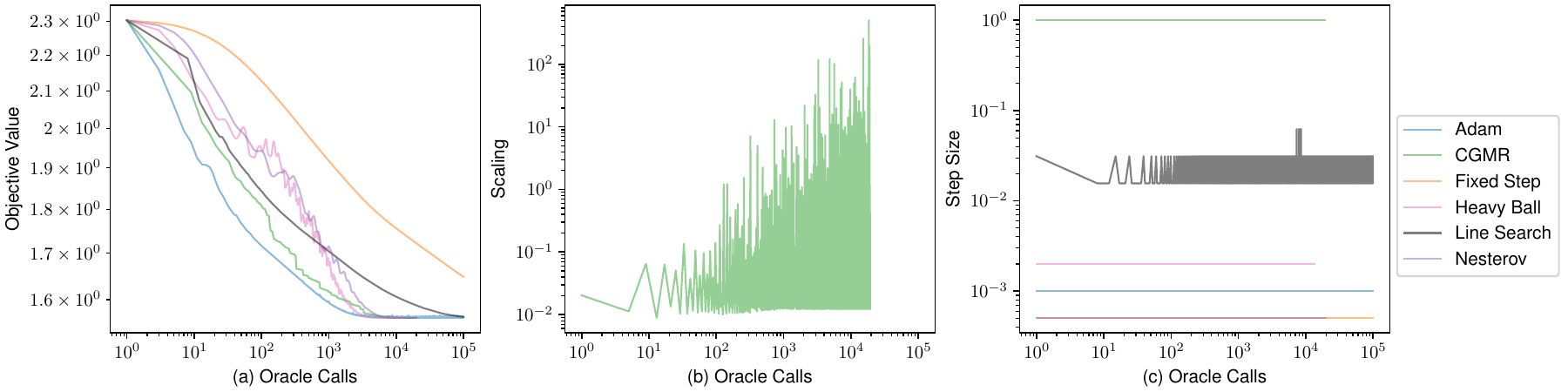}
    \caption{Multi-class logistic regression on CIFAR10. (a) Objective value. (b) Scaling utilized by the CGMR method. (c) Step size. }
    \label{fig:logistic_regression_cifar10}
\end{figure*}

In \cref{fig:logistic_regression_cifar10}, we present multi-class logistic regression with $\ell_2$ regularization on the CIFAR10 dataset \citep{Krizhevsky2009CIFAR10}. This $\mu$-strongly convex problem provides a well-behaved, yet challenging, baseline for comparison. 
We observe that ``CGMR'' is competitive with Adam and accelerated methods, while significantly outperforming line search (including PoNo) and fixed step size approaches. This is achieved without requiring any tuning or prior knowledge of problem constants. Notably, the scalings produced by our method oscillate between large (close to $1/\mu$) and small values throughout the optimization trajectory\footnote{As shown in \cref{apx:numerical-results-logistic}, this ``large scaling'' effect is most pronounced for alternating scaling.}. 
Despite the magnitude of these scaling values, the line search accepts the unit step length at \textit{all iterations}. In light of \cref{thm:local convergence}, we believe these large scaling values may contribute to the rapid convergence of our method.

\paragraph{Multilayer Perceptron (MLP).} 
In \cref{fig:mlp_fashionMNIST}, we examine a nonconvex, small MLP on the FashionMNIST dataset \citep{xiao2017fashionmnist}. The results show that ``CGMR'' scaled gradient is again competitive with Adam and significantly outperforms all other methods, all without requiring hyperparameter tuning. 
Notably, the unit step size is accepted by the line search at almost every $\SPC$ iteration. When $\NC$ is detected it is exploited using forward tracking search. Meanwhile, $\LPC$ does not occur at any point during the optimization trajectory. The scaling chosen by CGMR oscillates between large and small values, mirroring the behavior observed in the convex logistic regression experiment.

\begin{figure*}[ht]
    \centering
    \includegraphics[width=0.8\textwidth]{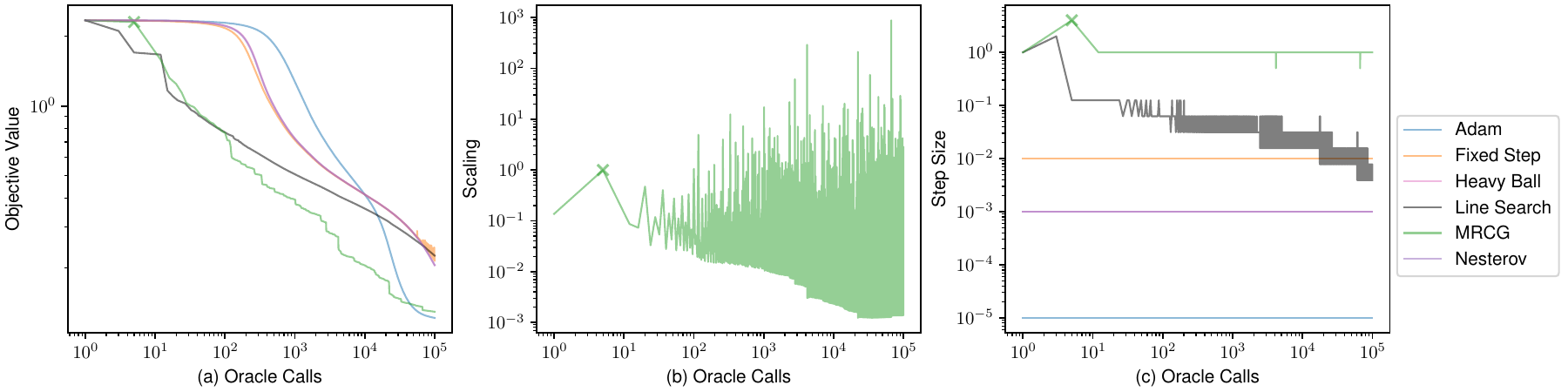}
    \caption{MLP on the FashionMNIST. (a) Objective value (b) Scaling utilized by the CGMR method (c) Step size. Crosses indicate iterations where negative curvature is detected.}
    \label{fig:mlp_fashionMNIST}
\end{figure*}

\paragraph{ResNet.} 

In \cref{fig:resnet_Imagenette}, we examine an over-parameterized, nonconvex ResNet18 architecture \citep{Kaiming2016ResNet} on the Imagenette dataset \citep{Howard2019Imagenette}. The results show that the unit step size is accepted at each iteration for our scaled gradient method, with no $\NC$ or $\LPC$ directions detected during the iterations.
Our method significantly outperforms the only other hyperparameter-free methods (line search and PoNo) but is surpassed by hyperparameter-tuned approaches. In contrast to scaled gradient, the tuned methods exhibit notable instability in the early iterations, followed by stabilization and convergence. This ``unstable convergence'' effect has been observed to be beneficial in many large-scale, nonconvex models \citep{cohenGradientDescentNeural2022,ahnUnderstandingUnstableConvergence2022}.
Given this, it appears that for scaled gradient to be competitive on large-scale problems, some degree of non-monotonicity must be introduced into the iterations. \citet{rouletSteppingEdgeCurvature2024} demonstrate that by using CG-style scaling with large step sizes (e.g., $\alpha \geq 2$), unstable convergence can be induced in a principled manner. A similar approach could be applied to our methods to design \emph{principled} step size schedules, such as annealing from $\alpha \geq 2$ (unstable) to $\alpha = 1$ (stable), to leverage the benefits of the unstable regime. We leave this exploration for future work. Finally, we note that the MR scaling uniformly decreases the gradient with the unit step size across our examples, consistent with \cref{thm:local convergence MR second-order sufficient}; see \cref{apx:additional-numerical-results} .

\begin{figure}[ht]
    \centering
    \includegraphics[width=0.8\linewidth]{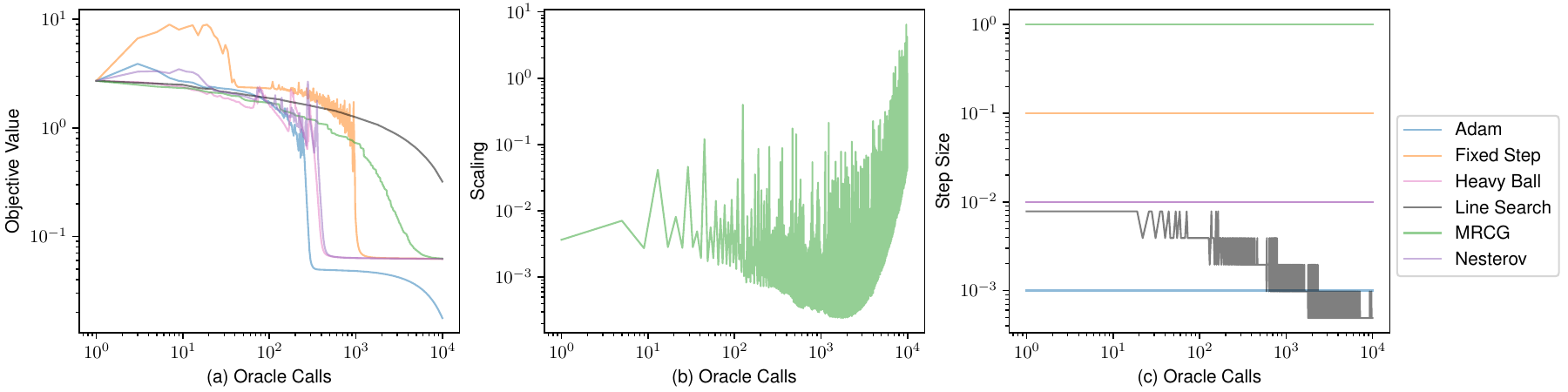}
    \caption{ResNet18 on Imagenette. (a) Objective value. (b) Scaling utilized by the MRCG method. (c) Step size. Crosses indicate iterations where negative curvature is detected.}
    \label{fig:resnet_Imagenette}
\end{figure}

\section{Conclusions} \label{sec:conclusion}

In this work we developed a framework based on (weakened) Lipschitz Hessian smoothness for analyzing Hessian-aware scaling of  gradient directions. Under this framework, we show that Hessian curvature information can be used to enhance vanilla GD with a unit step size guarantee. We show this guarantee holds for all iterations near minima satisfying certain regularity conditions. Furthermore, under alternative smoothness conditions, we prove global convergence. 
Numerically, we observe that the unit step size guarantee holds across most of the optimization trajectory. We also demonstrate that our method outperforms other hyperparameter free algorithms.

\noindent
\textbf{Limitations and Future Directions.} 
A limitation of our theory is the local nature of our unit step size guarantee. In particular, \cref{thm:unit step size acceptance} suggests that the guarantee may only hold in a region with a prohibitive dependence on certain problem constants. Fortunately, our numerical results demonstrate that the unit step size guarantee is widely applicable in practice. A natural future direction of work is obtaining theory that more closely tracks the practical performance of our algorithm. This can perhaps be achieved by considering analysis under an alternative globalization framework or with an alternative potential function. Another limitation  is performance compared with hand tuned alternatives. Extensions like Hessian aware step size schedules, scaling for adaptive gradient methods, and stochastic variations are promising candidates for addressing this limitation.

\section*{Acknowledgments}

This research was partially supported by the Australian Research Council through an Industrial Transformation Training Centre for Information Resilience (IC200100022).

\bibliography{bibliography}
\bibliographystyle{apalike}

\newpage 


\newpage
\appendix

\section{Additional Material for \cref{sec:approach}} \label{apx:proof-basic-properties}

\subsection{Discussion of \cref{alg:scaling selection} Hyperparameters} \label{apx:scaled gradient parameters discussion}

We now consider how to set the hyperparameters of \cref{alg:scaling selection}, beginning with $\sigma$. Recall that in \cref{prop:scaling upper bounds}, the constant $1/\sigma$ bounds the step size in both the $\SPC$ and $\LPC$ cases. 
Furthermore, in \cref{thm:local convergence} we show that better local rates can be attained when $\sigma$ is smaller than the local strong convexity constant. Together, these results suggest that $\sigma$ should generally be chosen small ($\sigma \ll 1$). This has the additional effect of maximizing the number of $\SPC$ steps, which, due to the second order scaling, tend to the be most effective in practice. Despite the dependence on $\sigma$ suggested by our local unit step size guarantee (\cref{thm:local convergence}), our numerical experience suggests that our algorithm widely utilizes the unit step when $\sigma$ is small or even zero.

Taking $\sigma$ to be set small, $\LPC$ steps are encountered rarely in practice. Therefore $\sLPC_{\tmin} = \sLPC_{\tmax} =\sLPC=1/\sigma$ is a reasonable choice, even if backtracking is necessary from this large initial step size. Practical experience also suggests that $\NC$ directions are relatively uncommon (see \cref{sec:numerical results}). Moreover, significant progress can be expected when exploiting negative curvature directions with large step sizes \cite{gouldExploitingNegativeCurvature2000,curtisExploitingNegativeCurvature2019,liuNewtonMRAlgorithmComplexity2023}. For these reasons, instead of tuning the $\NC$ scaling, we suggest setting $\sNC_{\tmin} = \sNC_{\tmax}=1$ and exploiting the $\NC$ directions with a forward/back-tracking line search. In the case where the problem is $\mu$-strongly convex, we can take $\sigma=0$ and disregard $\sLPC$ and $\sNC$, since curvature along the gradient direction will always be bounded below by $\mu$. This also implies that, in the strongly convex case, $\mu$ can take the place of $\sigma$ in our analysis.

\subsection{Proofs}
\paragraph{Proof of \cref{prop:scaling upper bounds}}
\begin{proof}
    In the $\LPC$ and $\NC$ cases, the results follow from the definition of $\sLPC$ and $\sNC$ in \cref{alg:scaling selection}. The $\SPC$ condition excludes the case where $\bgg = 0$, furthermore, it is clear that $\dotprod{\bgg, \HH \bgg} > 0$ implies $\HH\bgg \neq 0$, and hence all of the scalings are well defined. The Cauchy-Schwarz inequality implies $\gHg \leq \vnorm{\bgg} \vnorm{\HH\bgg}$ and $\vnorm{\HH\bgg} \geq \dotprod{\bgg, \HH \bgg}/\vnorm{\bgg}$. Applying these two inequalities subsequently to $\sMR$ we have 
    \begin{align*}
        0 < \sMR = \frac{\gHg}{\vnorm{\HH \bgg}^2} \leq \frac{\vnorm{\bgg}}{ \vnorm{\HH \bgg }} \leq \frac{\vnorm{\bgg}^2}{\gHg},
    \end{align*}
    which captures the three left inequalities in this case. The right-most inequality follows from $\dotprod{\bgg, \HH \bgg } > \sigma \vnorm{\bgg}^2$ and the definition of \cref{eqn:CG scaling defn}. 
\end{proof}

\paragraph{Proof of \cref{prop:second-order descent}}

\begin{proof}
    The first-order descent trivially holds. 
    In the $\NC$ case, the result immediately follows from $\dotprod{\pp, \HH \pp} < 0$ and the first-order descent property. In the $\LPC$ case, 
    \begin{align*}
        \dotprod{\bgg, \pp} + \dotprod{\pp, \HH \pp} &= -\sLPC \vnorm{\bgg}^2 + (\sLPC)^2\dotprod{\bgg, \HH \bgg} \\
        &\leq -\sLPC \vnorm{\bgg}^2 +  (\sLPC)^2 \sigma \vnorm{\bgg}^2 \\
        &\leq -\sLPC \vnorm{\bgg}^2 + \sLPC \vnorm{\bgg}^2 = 0,
    \end{align*}
    where the second to last line follows from the $\LPC$ condition and the last line follows from $\sLPC \leq 1/\sigma$ in \cref{prop:scaling upper bounds}.  

    For the $\SPC$ case, we consider each scaling successively. For the CG scaling, we have 
    \begin{align*}
        \dotprod{ \pCG, \bgg} + \dotprod{\pCG, \HH \pCG} &= -\sCG \vnorm{\bgg}^2 + (\sCG)^2 \dotprod{\bgg, \HH \bgg} \\
        &= -\frac{\vnorm{ \bgg}^4}{\dotprod{\bgg, \HH \bgg}} + \frac{\vnorm{ \bgg}^4}{(\dotprod{\bgg, \HH \bgg})^2} \dotprod{\bgg, \HH \bgg} \\
        &= 0.
    \end{align*}
    For the MR scaling, we have
    \begin{align*}
        \dotprod{\pMR, \bgg} + \dotprod{\pMR, \HH \pMR} 
        &= -\sMR \vnorm{\bgg}^2 + (\sMR)^2 \dotprod{\bgg, \HH \bgg} \\
        &= -\frac{\langle \bgg, \HH \bgg \rangle }{\| \HH \bgg \|^2} \| \bgg\|^2 + \frac{(\dotprod{\bgg, \HH \bgg} )^2}{ \| \HH \bgg \|^4}\dotprod{\bgg, \HH \bgg} \\
        &\leq -\frac{\dotprod{\bgg, \HH \bgg} }{\| \HH \bgg \|^2} \vnorm{\bgg}^2 + \frac{ \vnorm{\bgg}^2}{ \| \HH \bgg \|^2} \dotprod{\bgg, \HH \bgg} \\
        &= 0,
    \end{align*}
    where the third line follows from the Cauchy-Schwarz inequality. Finally, for the geometric mean scaling, we have
    \begin{align*}
        \dotprod{\pGM, \bgg} + \dotprod{\pGM, \HH \pGM} &= -\sGM \vnorm{\bgg}^2 + (\sGM)^2\dotprod{\bgg, \HH \bgg} \\
        &= -\frac{\vnorm{\bgg}^3}{\vnorm{\HH \bgg}} + \frac{\vnorm{\bgg}^2}{\vnorm{\HH \bgg}^2} \dotprod{\bgg, \HH \bgg} \\
        &\leq  -\frac{\vnorm{\bgg}^3}{\vnorm{\HH \bgg }} + \frac{\vnorm{\bgg}^2}{\vnorm{\HH \bgg}^2} \vnorm{\bgg} \vnorm{\HH \bgg} \\
        &=  -\frac{\vnorm{\bgg}^3}{\vnorm{\HH \bgg }} + \frac{\vnorm{\bgg}^3}{\vnorm{\HH \bgg}} = 0,
    \end{align*}
    where again the third line follows from the Cauchy-Schwarz inequality.

\end{proof}

\paragraph{Proof of \cref{prop:scalar invariance}}
\begin{proof}
     Consider a scalar reparameterization of the original variables $\xx$ as $\yy = \xx/c$, for some $c \neq 0$. Our objective over the new coordinates is $\bar{f}(\yy) = f(c \yy) $ so that
     \begin{align*}
         \grad_\yy \bar{f}(\yy) = c \grad f(\xx), \quad \grad^2_\yy \bar{f}(\yy) = c^2 \grad^2 f(\xx).
     \end{align*}
    Therefore, $\SPC$ scalings computed in the transformed coordinates satisfy
    \begin{align*}
        s(\yy) = \frac{s(\xx)}{c^2}.
    \end{align*}
    For example, for the MR scaling we have 
    \begin{align*}
        s^\text{MR}(\yy) = \frac{\dotprod{ \grad_\yy \bar{f}(\yy), \grad_\yy^2 \bar{f}(\yy) \grad_\yy \bar{f}(\yy)}}{\vnorm{\grad_\yy^2 \bar{f}(\yy) \grad_\yy \bar{f}(\yy)}^2} 
        = \frac{\dotprod{ c \grad f(\xx), (c^2 \grad^2 f(\xx)) c \grad f(\xx)}}{\vnorm{c^2 \grad^2 f(\xx) c \grad f(\xx)}^2} = \frac{1}{c^2} \sMR(\xx).
    \end{align*}
    Now this implies that, for the $\SPC$ scaled gradient directions, the update computed with respect to the new coordinates preserves the reparameterization. Indeed, given $\yy_k = \xx_k/c$, the scaled gradient descent update at $\yy_k$ with some fixed step size $\alpha > 0$ is 
    \begin{align*}
        \yy_{k+1} &= \yy_k - \alpha s(\yy_k)\grad_\yy \bar{f}(\yy_k) = \frac{1}{c}  \left(\xx_k - \alpha s(\xx_k) \grad f(\xx_k)\right)  = \frac{\xx_{k+1}}{c},
    \end{align*}
    where $\xx_{k+1} - \xx_k - \alpha s(\xx_k) \grad f(\xx_k)$ is the update in the original coordinates. That is, the same relationship between $\xx$ and $\yy$ holds at the following iteration.
\end{proof}

\section{Additional Material for \cref{sec:convergence-analysis}}  \label{apx:proof-global-convergence}

\subsection{Line Search Algorithms}
\label{sec:appendix:linesearch_algs}

\begin{algorithm}[htbp]
    \begin{algorithmic}[1]
        \STATE \textbf{input}: Initial step size $\alpha^0$, Scaling parameter $0 < \theta < 1$.
        \vspace{1mm}
        \STATE $\alpha \gets \alpha^0$.
        \vspace{1mm}
        \WHILE{\cref{eqn:armijo condition} is not satisfied}
        \vspace{1mm}
            \STATE $\alpha \gets \theta\alpha$.
            \vspace{1mm}
        \ENDWHILE
        \vspace{1mm}
        \STATE \textbf{return} $\alpha$.
        \vspace{1mm}
    \end{algorithmic}
    \caption{Backward Tracking Line Search.}
    \label{alg:back tracking line search}
\end{algorithm}

\begin{algorithm}[htbp]
    \begin{algorithmic}[1]
        \STATE \textbf{input}: Initial step size $\alpha^0$, Scaling parameter $0 < \theta < 1$.
        \vspace{1mm}
        \STATE $\alpha \gets \alpha^0$.
        \vspace{1mm}
        \IF {\cref{eqn:armijo condition} is not satisfied} 
		\vspace{1mm}
		\STATE \text{Call \cref{alg:back tracking line search}}
		\vspace{1mm}
		\ELSE
		\vspace{1mm}
		\WHILE {\cref{eqn:armijo condition} is satisfied}
		\vspace{1mm}
		\STATE $ \alpha = \alpha/\theta .$
		\vspace{1mm}
		\ENDWHILE
        \vspace{1mm}
		\STATE \textbf{return} $\theta\alpha$.
		\vspace{1mm}
        \ENDIF
    \end{algorithmic}
    \caption{Forward/Backward Tracking Line Search}
    \label{alg:forward tracking line search}
\end{algorithm}

\subsection{Derivation and Proof of \cref{thm:unit step size acceptance}} \label{apx:proof-unit-step-size}

Throughout the following we routinely make use of the facts, due to collinearity of $\bgg$ and $\pp$, that $\dotprod{\pp, \bgg} = - \vnorm{\pp} \vnorm{\bgg}$, $\vnorm{\pp} = s \vnorm{\bgg}$, and the fact that $\dotprod{\bgg, \HH \bgg}$ and $\dotprod{\pp, \HH \pp}$ have the same sign. Letting $\pp = -s \bgg$, $s \geq 0$, \cref{ass:Hessian smoothness condition} and twice continuous differentiability of $f$ imply
\begin{align*}
    \vnorm{\bgg(\xx + \pp) - \bgg (\xx) - \HH(\xx)\pp}  &\leq \frac{L_2}{2}\vnorm{\pp}^2, \tageq\label{eqn:gradient upper bound}\\
    \abs{ f(\xx + \pp) - f(\xx) - \dotprod{\bgg, \pp} - \frac{1}{2} \dotprod{\pp, \HH\pp}} &\leq \frac{L_2}{6} \vnorm{\pp}^3. \tageq\label{eqn:function value upper bound} 
\end{align*}
These bounds can be derived similarly to the general Lipschitz Hessian upper bounds in, e.g., \citet[Lemma 1]{nesterovCubicRegularizationNewton2006}. Our analysis is based on applying the second order descent condition \cref{eqn:second-order descent} to the cubic upper bound in \cref{eqn:function value upper bound}. We consider the non-negative and negative curvature cases separately.

\paragraph{Non-negative Curvature Case.} Suppose the curvature along the gradient direction is non-negative, i.e., $\gHg \geq 0$. In this case, \cref{alg:scaling selection} selects either the $\SPC$ or $\LPC$ scalings.  By \cref{prop:scaling upper bounds,prop:second-order descent}, we have $s \leq 1/\sigma$ and \cref{eqn:second-order descent}. Considering \cref{eqn:function value upper bound} and applying $\alpha \leq 1 $ and \cref{eqn:second-order descent}, we have
\begin{align*}
    f(\xx + \alpha \pp) &\leq f(\xx) + \alpha\dotprod{\bgg, \pp}  + \frac{\alpha^2}{2} \dotprod{\pp , \HH \pp} + \frac{L_2 \alpha^3 }{6} \vnorm{\pp}^3 \\ 
    &\leq f(\xx) + \frac{\alpha}{2}\dotprod{\bgg, \pp}  + \frac{\alpha}{2} \left(\dotprod{\bgg, \pp} + \dotprod{\pp , \HH \pp}\right) + \frac{L_2 \alpha^3 }{6} \vnorm{\pp}^3 \\
    &\leq f(\xx) + \frac{\alpha}{2}\dotprod{\bgg, \pp} + \frac{L_2 \alpha^3 }{6} \vnorm{\pp}^3. 
\end{align*}
Subtracting $f(\xx) + \alpha \rho \dotprod{\pp, \bgg}$ from both sides yields 
\begin{align*}
    f(\xx + \alpha \pp) - f(\xx) - \alpha \rho \dotprod{\pp, \bgg} &\leq \alpha\left(\frac{1}{2} - \rho \right)\dotprod{\pp, \bgg} + \frac{L_2 \alpha^3}{6} \vnorm{\pp}^3 \\
    &= -\alpha \left(\frac{1}{2} - \rho \right)\vnorm{\bgg}\vnorm{\pp} + \frac{L_2 \alpha^3 }{6} s \vnorm{\bgg} \vnorm{\pp}^2 \\
    &= \left( \rho - \frac{1}{2} + \frac{L_2 \alpha^2 s \vnorm{\pp} }{6}  \right) \alpha \vnorm{\pp}\vnorm{\bgg} \\
    &\leq \left( \rho - \frac{1}{2} + \frac{L_2 \alpha^2 \vnorm{\pp} }{6\sigma }  \right) \alpha \vnorm{\pp}\vnorm{\bgg},
\end{align*}
where the last inequality follows from $s \leq 1/\sigma$. This implies \cref{eqn:armijo condition} holds if  
\begin{align}
    \label{eqn:positive curvature line search termination}
    \alpha  \leq \min\left\{1,\sqrt{\frac{6\sigma(1/2 - \rho)}{L_2 \vnorm{\pp}}}\right\}.
\end{align}

\paragraph{Negative Curvature Case.} We now consider the case where $\dotprod{\bgg, \HH \bgg} < 0$, in which case \cref{alg:scaling selection} selects the $\NC$ scaling. As a result, dropping the negative term in \cref{eqn:function value upper bound} gives
\begin{align*}
    f(\xx + \alpha \pp) &\leq f(\xx) + \alpha\dotprod{\pp, \bgg} + \frac{\alpha^2}{2} \dotprod{\pp, \HH \pp} + \frac{L_2 \alpha^3}{6} \vnorm{\pp}^3 \\
    &\leq f(\xx) + \alpha\dotprod{\pp, \bgg} + \frac{L_2 \alpha^3}{6} \vnorm{\pp}^3.
\end{align*}
Similarly to the previous case, subtracting $f(\xx) + \alpha \rho \dotprod{\pp, \bgg}$ yields 
\begin{align*}
    f(\xx + \alpha \pp ) - f(\xx) - \alpha  \rho \dotprod{\pp, \bgg} &\leq \alpha (1 - \rho)\dotprod{\pp, \bgg} + \frac{L_2 \alpha^3}{6} \vnorm{\pp}^3 \\
    &\leq -\alpha (1 - \rho)\vnorm{\bgg}\vnorm{\pp} + \frac{L_2 \alpha^3 s}{6} \vnorm{\bgg}\vnorm{\pp}^2 \\
    &\leq \alpha \vnorm{\pp} \vnorm{\bgg} \left(- (1 - \rho) + \frac{L_2 \alpha^2 (\sNC_{\tmax})}{6} \vnorm{\pp}\right),
\end{align*}
which implies \cref{eqn:armijo condition} if  
\begin{align}    \label{eqn:NC case line search termination}
    \alpha  \leq \sqrt{\frac{6(1-\rho)}{L_2\sNC_{\tmax}\vnorm{\pp}}}, 
\end{align}

This analysis demonstrates the importance of separating the non-negative and negative cases. In particular, in the case of the former, the condition $\alpha \leq 1$ is necessary, while in the latter, the second-order term is omitted entirely, allowing for unrestricted step sizes. This distinction explains why large step sizes (e.g., from forward tracking line search) are feasible in the negative curvature case. 

The bounds in \cref{eqn:positive curvature line search termination,eqn:NC case line search termination} ensure that the line search condition is satisfied by small enough step sizes. While these bounds are typically utilized for proving global convergence, we instead establish a \textit{local unit step size} result for our Hessian-aware scaled gradient descent. Specifically, these bounds imply that, for $\vnorm{\pp}$ sufficiently small, the line search condition \cref{eqn:armijo condition} is satisfied with $\alpha = 1$. Through \cref{prop:scaling upper bounds}, this can be linked to a condition on the gradient, which forms the proof of \cref{thm:unit step size acceptance}.

\begin{proof}[Proof of \cref{thm:unit step size acceptance}]
Suppose $\gHg \geq 0$. Since $s \leq 1/\sigma$, we have $\vnorm{\pp} \leq \vnorm{\bgg}/\sigma$ so that if \cref{eqn:gradient line search termination} holds then
\begin{align*}
    \vnorm{\pp} \leq \frac{\vnorm{\bgg}}{\sigma} \leq \frac{6\sigma(1/2 - \rho)}{L_2},
\end{align*}
which implies \cref{eqn:positive curvature line search termination} holds with $\alpha=1$. When $\gHg < 0$, we have $s \leq \sNC_{\tmax}$, and hence if \cref{eqn:gradient line search termination} then
\begin{align*}
    \vnorm{\pp} \leq \sNC_{\tmax}\vnorm{\bgg} \leq\frac{6(1-\rho)}{L_2\sNC_{\tmax}},
\end{align*}
which implies \cref{eqn:NC case line search termination} with $\alpha=1$.
\end{proof}

\subsection{Proof of \cref{prop:global convergence}}
Let us first restate \cref{prop:global convergence} with all the details and underlying constants.
\begin{proposition}[Restatement of \cref{prop:global convergence}]
    \label{prop:global convergence_detailed}
    Consider \cref{ass:Hessian gradient smoothness condition,ass:Hessian smoothness condition} and suppose $- \infty < f^{\star} = \min_{\xx \in \real^{d}} f(\xx)$. For any $0 < \eg < 1$, after at most  
    \begin{align*}
        K = \left\lceil\frac{f(\xx_0) - f^{\star}}{\min\left\{ \cSPC, \cLPC, \cNC \right\} \eg^2}\right\rceil,
    \end{align*}
    iterations of \cref{alg:scaled gradient}, we have $\vnorm{\bgg_k} \leq \varepsilon_{\bgg}$, where $\cSPC$, $\cLPC$ and $\cNC$ are defined as follows:
    \begin{align*}
        \cSPC &\defeq \sigma\rho\min\left\{ \left(\frac{6\sigma(\tfrac12 - \rho)}{L_2} \right)^2,  \frac{1}{L_1^2} \right\}\\
        \cLPC &\defeq\rho\min\left\{\sLPC_{\tmin}, \sqrt{\frac{6\sigma(\tfrac12 - \rho)\sLPC_{\tmin}}{L_2}} \right\}, \\
        \cNC &\defeq \rho\sqrt{\frac{6(1-\rho) \sNC_{\tmin}}{L_2 \sNC_{\tmax}}}.
    \end{align*}
\end{proposition}
Before providing the proof, we note that, as discussed earlier, $L_1$ and $L_2$ are lower bounds for the global Lipschitz constants of the problem (if such constants exist), so the factor based on $L_1$ and $L_2$ could are smaller than those based on their global counterparts. The constant factor in \cref{prop:global convergence_detailed} depends quadratically on both $L_{1}$ and $L_{2}$, but the dependence on $L_{1}$ can be improved to linear if the MR scaling is not used; see \cref{remark:scaling lower bound}. This would align with the dependence on the Lipschitz gradient constant for standard gradient descent in the nonconvex setting.

To prove \cref{prop:global convergence_detailed}, we first consider a lemma which analyzes the worst case \emph{per-iteration} descent for the \texttt{NC}, \texttt{LPC}, and \texttt{SPC} separately. 

\begin{lemma}[Per-iteration Descent]\label{lemma:per iteration decrease}
    Consider \cref{ass:Hessian smoothness condition}, and let $\pp$ and $\alpha > 0$ be generated by \cref{alg:scaled gradient}. If $\gHg < 0$ ($\NC$), then
        \begin{align*}
            f(\xx + \alpha \pp) \leq f(\xx) - \rho\sqrt{\frac{6(1-\rho) \sNC_{\tmin}}{L_2 \sNC_{\tmax}}} \vnorm{\bgg}^{3/2}.
        \end{align*}
    If $0 \leq \gHg \leq \sigma \vnorm{\bgg}^2$ ($\LPC$), then
    \begin{align*}
        f(\xx + \alpha \pp) \leq f(\xx) - \rho\min\left\{\sLPC_{\tmin} \vnorm{\bgg}^2, \sqrt{\frac{6\sigma(1/2 - \rho)\sLPC_{\tmin}}{L_2}} \vnorm{\bgg}^{3/2} \right\}.
    \end{align*}
    Suppose $\gHg > \sigma \vnorm{\bgg}^2$ ($\SPC$). If
    \begin{align}
        \label{eqn:SPC step norm lower bound}
        \vnorm{\pp} \geq \frac{6 \sigma (1/2 - \rho)}{L_2},
    \end{align}
    then
    \begin{align*}
        f(\xx + \alpha \pp) \leq f(\xx) -\rho \sigma \left(\frac{6\sigma(1/2 - \rho)}{L_2}\right)^2.
    \end{align*}
    Otherwise, we have $\alpha =1$ and 
    \begin{align*}
        f(\xx + \pp) \leq  f(\xx) -\rho \vnorm{\pp}\vnorm{\bgg}.
    \end{align*}
\end{lemma}
\begin{proof}
    ($\NC$) From \cref{eqn:NC case line search termination}, it follows that the largest step size satisfying \cref{eqn:armijo condition} must satisfy
    \begin{align*}
        \alpha \geq \sqrt{\frac{6(1-\rho)}{L_2 \sNC_{\tmax} \vnorm{\pp}}}.
    \end{align*}
    Hence, 
    \begin{align*}
        f(\xx + \alpha \pp) - f(\xx) &\leq  \rho \alpha \dotprod{\pp, \bgg} = -\rho \alpha \vnorm{\pp}\vnorm{\bgg}\leq - \rho \sqrt{\frac{6(1-\rho) \vnorm{\pp}}{L_2 \sNC_{\tmax}}} \vnorm{\bgg}.
    \end{align*}
    Applying $\vnorm{\pp} \geq \sNC_{\tmin} \vnorm{\bgg}$ gives the result. 
    
    ($\LPC$) Similarly, \cref{eqn:positive curvature line search termination} implies that the largest step size $\alpha \in (0,1]$ satisfying \cref{eqn:armijo condition} must satisfy 
    \begin{align*}
        \alpha \geq \min\left\{1, \sqrt{\frac{6\sigma(1/2 - \rho)}{L_2 \vnorm{\pp}}} \right\}.
    \end{align*}
    Now, from \cref{eqn:armijo condition} we obtain 
    \begin{align*}
        f(\xx + \alpha \pp) - f(\xx) &\leq \rho \alpha \dotprod{\pp, \bgg} = -\rho \alpha \vnorm{\pp} \vnorm{\bgg} \\
        &\leq -\rho\min\left\{\vnorm{\pp} \vnorm{\bgg}, \sqrt{\frac{6\sigma(1/2 - \rho)\vnorm{\pp}}{L_2}}\vnorm{\bgg} \right\}.
    \end{align*}
    Since $\sLPC \geq \sLPC_{\tmin}$, $\vnorm{\pp} \geq \sLPC_{\tmin} \vnorm{\bgg}$, which yields the result.
    
    ($\SPC$) Similar to the $\LPC$ case, \cref{eqn:positive curvature line search termination} implies that the largest step size $\alpha \in (0,1]$ satisfying \cref{eqn:armijo condition} must satisfy 
    \begin{align*}
        \min\left\{1, \sqrt{\frac{6\sigma(1/2 - \rho)}{L_2 \vnorm{\pp}}} \right\} \leq \alpha \leq 1.
    \end{align*}
    Suppose  \cref{eqn:SPC step norm lower bound} holds, which implies  
    \begin{align*}
        \alpha \geq \sqrt{\frac{6\sigma(1/2 - \rho)}{L_2 \vnorm{\pp}}}.
    \end{align*}
    In light of \cref{eqn:armijo condition,eqn:second-order descent} and the fact that  $\dotprod{\pp, \HH \pp} > \sigma \vnorm{\pp}^2$, we have
    \begin{align*}
        f(\xx + \alpha \pp) - f(\xx) &\leq \rho \alpha \dotprod{\pp, \bgg} \leq - \rho \alpha \dotprod{\pp, \HH \pp} 
        \leq - \rho \alpha \sigma \vnorm{\pp}^2 \leq -\rho \sigma \left(\frac{6\sigma(1/2 - \rho)}{L_2}\right)^2,
    \end{align*}
    where the final inequality follows from \cref{eqn:SPC step norm lower bound}. If \cref{eqn:SPC step norm lower bound} does not hold, then $\alpha=1$ and hence 
    \begin{align*}
        f(\xx + \pp) - f(\xx) &\leq \rho \dotprod{\pp, \bgg} = - \rho \vnorm{\bgg}\vnorm{\pp}.
    \end{align*}
\end{proof}

\begin{remark}
    Many of the cases considered considered in \cref{lemma:per iteration decrease} suggest descent with an improved dependence on the gradient norm, when compared with standard gradient descent. For instance, in the $\NC$ case the dependence is $\vnorm{\bgg}^{3/2}$, which matches the optimal rate for second order methods in the nonconvex setting. On the other hand, in the $\SPC$ case, with large $\vnorm{\pp}$, the descent is \textit{independent} of the gradient norm. An unfortunate consequence of worst case analysis is that these improvements are not reflected in the final rate. 
\end{remark}

By inspecting the proof of \cref{lemma:per iteration decrease}, we see that the per iteration decrease is independent of the scaling magnitude in all cases except $\SPC$ steps with $\alpha=1$. As discussed in the main body, we handle this case with a regularity condition on the Hessian-gradient product, \cref{ass:Hessian gradient smoothness condition}. In particular, by combining \cref{eqn:MR scaling defn}, $\sigma \vnorm{\bgg}^2 < \dotprod{\bgg, \HH \bgg }$, and $ \vnorm{\HH \bgg} \leq L_1 \vnorm{\bgg}$ from \cref{ass:Hessian gradient smoothness condition}, we can obtain a lower bound on the $\SPC$ scalings. 

\begin{lemma}[Scaling Lower Bounds] \label{lemma:scaling lower bounds}
    Consider \cref{ass:Hessian gradient smoothness condition} and let $\sigma \leq L_1$. We have
    \begin{align}
    \label{eqn:MR scaling lower bound}
        \frac{\sigma}{L_1^2} &\leq \sMR \leq \sGM \leq \sCG.
    \end{align}
\end{lemma}

\begin{remark} \label{remark:scaling lower bound}
    Note that since $\dotprod{ \bgg, \HH\bgg} \leq \vnorm{\HH \bgg} \vnorm{\bgg} \leq L_1 \vnorm{\bgg}^2$, the $\SPC$ case is detected only when $\sigma$ is chosen such that $\sigma \leq L_1$; otherwise \cref{alg:scaling selection} always returns either of $\LPC$ or $\NC$ scalings. Additionally, by applying $\vnorm{\HH\bgg} \leq L_1 \vnorm{\bgg}$ directly to $\sGM$ we can sharpen the lower bound for the CG and GM scalings to $1/L_1 \leq \sGM \leq \sCG$.
\end{remark}

The proof of \cref{prop:global convergence_detailed} proceeds by combining the worst case analysis of \cref{lemma:per iteration decrease} with the lower bounds from \cref{lemma:scaling lower bounds}.

\begin{proof}[Proof of \cref{prop:global convergence_detailed}]
    Suppose that $\vnorm{\bgg_k} \leq \eg$ fails to hold for $k=0,\ldots,K-1$. We divide the iterations $\sK = \{0, \ldots,K-1\}$ into a disjoint union $\sK = \sK_{\text{NC}} \cup \sK_{\text{LPC}} \cup \sK_{\text{SPC}} $ where 
    \begin{align*} 
        \sK_{\text{NC}} &= \left\{ k \in \sK \mid \dotprod{\bgg_k, \HH_k \bgg_k} < 0 \right\} \\
        \sK_{\text{LPC}} &= \left\{ k \in \sK \mid 0 \leq \dotprod{\bgg_k, \HH_k \bgg_k} \leq \sigma \vnorm{\bgg_k}^2 \right\} \\
        \sK_{\text{SPC}} &= \left\{ k \in \sK \mid \dotprod{\bgg_k, \HH_k \bgg_k} > \sigma \vnorm{\bgg_k}^2 \right\}.
    \end{align*}
    For $k \in \sK_{\text{NC}}$, \cref{lemma:per iteration decrease}, $\vnorm{\bgg_k} > \eg$ and $\eg < 1$ give
    \begin{align*}
         f(\xx_k) - f(\xx_{k+1}) >  \rho\sqrt{\frac{6(1-\rho) \sNC_{\tmin}}{L_2 \sNC_{\tmax}}} \eg^{3/2} \geq \cNC \eg^2,
    \end{align*}
    Similarly, for $k \in \sK_{\text{LPC}}$, \cref{lemma:per iteration decrease}, $\vnorm{\bgg_k} > \eg$, and $\eg < 1$ imply
    \begin{align*}
         f(\xx_k) -f(\xx_{k+1}) &> \rho\min\left\{\sLPC_{\tmin} \vnorm{\bgg_k}^2, \sqrt{\frac{6\sigma(1/2 - \rho)\sLPC_{\tmin}}{L_2}} \vnorm{\bgg_k}^{3/2} \right\} \\
        &> \rho\min\left\{\sLPC_{\tmin} \eg^2, \sqrt{\frac{6\sigma(1/2 - \rho)\sLPC_{\tmin}}{L_2}}\eg^{3/2} \right\}  \\
        &\geq \cLPC \eg^2,
    \end{align*}
    For $k \in \sK_{\text{SPC}}$, if $\cref{eqn:SPC step norm lower bound}$ does not hold, then \cref{lemma:scaling lower bounds,lemma:per iteration decrease} give
    \begin{align*}
        f(\xx_k) - f(\xx_{k+1}) &\geq \rho \vnorm{\pp_k} \vnorm{\bgg_k} 
        \geq \frac{\sigma \rho}{L_1^2} \vnorm{\bgg_k}^2 
        > \frac{\sigma \rho}{L_1^2} \eg^2.
    \end{align*}
    We can combine this with the case where $\cref{eqn:SPC step norm lower bound}$ is satisfied to obtain 
    \begin{align*}
        f(\xx_k) - f(\xx_{k+1}) &> \rho\min\left\{ \sigma \left(\frac{6\sigma^2(1/2 - \rho)}{L_2} \right)^2,  \frac{\sigma}{L_1^2} \eg^2\right\} \\
        &\geq  \sigma\rho\min\left\{ \left(\frac{6\sigma(1/2 - \rho)}{L_2} \right)^2,  \frac{1}{L_1^2} \right\}\eg^2 \\
        &= \cSPC \eg^2,
    \end{align*}
    where in the second line we applied $\eg<1$. Finally, we apply a telescoping sum over each case
    \begin{align*}
        f(\xx_0) - f(\xx_K) &= \sum_{i=0}^{K-1} f(\xx_k) - f(\xx_{k+1}) \\
        &= \sum_{k\in\sK_{\text{NC}}} f(\xx_k) - f(\xx_{k+1}) + \sum_{k\in\sK_{\text{LPC}}} f(\xx_k) - f(\xx_{k+1}) + \sum_{k\in\sK_{\text{SPC}}} f(\xx_k) - f(\xx_{k+1}) \\
        &> |\sK_{\text{NC}} | \cNC \eg^2 + |\sK_{\text{LPC}}|\cLPC \eg^2 + |\sK_{\text{SPC}}| \cSPC \eg^2 \\
        &\geq (|\sK_{\text{NC}} | + |\sK_{\text{LPC}} | + |\sK_{\text{SPC}}|)\min\left\{ \cNC, \cLPC, \cSPC  \right\}\eg^2 \\
        &= K\min\left\{ \cNC, \cLPC, \cSPC \right\}\eg^2 \\
        &\geq f(\xx_0) - f^{\star},
    \end{align*}
    which implies $f^{\star} > f(\xx_K)$, leading to a contradiction.
\end{proof}

\subsection{Curvature Condition Extension} \label{apx:curvature condition extension}

Our result in \cref{thm:local convergence} offers a guarantee on when $\alpha=1$ will satisfy the Amijo condition \cref{eqn:armijo condition}. Recall that Amijo condition will be satisfied by any sufficiently short step length (if $\alpha=1$ any small scaling $s$). For this reason the Amijo condition is often paired with a \textit{curvature condition} \citep{nocedalNumericalOptimization2006}, which ensures that the step length is \textit{large enough}, in the sense that the directional derivative at the new point has increased sufficiently from its initial value. 
That is, for a descent direction $\pp$ and $\eta \in (0,1)$ we require $\alpha$ to satisfy
\begin{align*}
    \dotprod{\bgg(\xx + \alpha \pp), \pp} \geq \eta \dotprod{\bgg(\xx), \pp}. \tageq\label{eqn:curvature condition}
\end{align*}
When $\rho < \eta$ the pair \cref{eqn:armijo condition} and \cref{eqn:curvature condition} are known as the Wolfe conditions. It should be noted that, while the Amijo condition \cref{eqn:armijo condition} can be guaranteed to hold by employing backtracking line search, satisfaction of the Wolfe conditions requires more sophisticated search techniques \citep{nocedalNumericalOptimization2006}. Fortunately, the next proposition reveals that, under certain smoothness conditions, $\SPC$ scalings automatically satisfy \cref{eqn:curvature condition} locally with $\alpha=1$.
\begin{proposition} \label{prop:curvature condition}
    Consider \cref{ass:Hessian smoothness condition}. If
    \begin{align*}
        \vnorm{\bgg} \leq \frac{2\sigma^2\eta}{L_2}, \tageq\label{eqn:curvature condition gradient upper bound - CG case}
    \end{align*}
    then \cref{eqn:curvature condition} is satisfied by $\pCG$ with $\alpha=1$ and $\eta \in (0,1)$. Now consider, in addition, \cref{ass:Hessian gradient smoothness condition}. If 
    \begin{align*}
        \vnorm{\bgg} \leq \frac{2\sigma^2}{L_2}\left(\eta - \left( 1 - \frac{\sigma^2}{L_1^2} \right) \right), \tageq\label{eqn:curvature condition gradient upper bound - MRGM case}
    \end{align*}
    then $\pGM$ and $\pMR$ satisfy \cref{eqn:curvature condition} with $\alpha=1$ and $\eta > 1 - \sigma^2/L_1^2$.

\end{proposition} 
\begin{proof}[Proof of \cref{prop:curvature condition}] Note that in this proof we make explicit reference to the iteration counter to make tracking updates simpler. In particular, we write $\xx_{k+1} = \xx_k + \pp_k$. A Taylor expansion for the gradient yields 
\begin{align*}
    \bgg_{k+1} &= \bgg_k + \int_{0}^1 \HH(\xx_k + t\pp_k) \pp_k \, dt \\
    &= \bgg_k + \int_0^1 (\HH(\xx_k + t \pp_k) - \HH_k + \HH_k)\pp_k \, dt.
\end{align*}
Taking the inner product with $\pp_k$ and applying \cref{ass:Hessian smoothness condition}
\begin{align*}
    \dotprod{\bgg_{k+1}, \pp_k} &= \dotprod{\bgg_k, \pp_k} + \dotprod{\pp_k, \HH_k \pp_k} + \int_0^1 \dotprod{\pp_k, (\HH(\xx_k + t \pp_k) - \HH_k)\pp_k} \, dt \\
    &\geq \dotprod{\bgg_k, \pp_k} + \dotprod{\pp_k, \HH_k \pp_k} - L_2 \int_0^1 t \vnorm{\pp_k}^3 \, dt \\ 
    &= \dotprod{\bgg_k, \pp_k} + \dotprod{\pp_k, \HH_k \pp_k} - \frac{L_2}{2} \vnorm{\pp_k}^3.
\end{align*}
We can write $\vnorm{\pp_k}^3 = - s_k \dotprod{\bgg_k, \pp_k} \vnorm{\pp_k}$ so that 
\begin{align*}
    \dotprod{\bgg_{k+1}, \pp_k} &\geq \dotprod{\bgg_k, \pp_k} + \dotprod{\pp_k, \HH_k \pp_k} + \frac{L_2}{2} s_k \dotprod{\bgg_k, \pp_k}\vnorm{\pp_k}.
\end{align*}
Now considering the CG scaling specifically. Combining the condition on the gradient norm \cref{eqn:curvature condition gradient upper bound - CG case} with \cref{prop:scaling upper bounds} we have
\begin{align*}
    \sCG_k \vnorm{\pCG_k} \leq \frac{\vnorm{\bgg_k}}{\sigma^2} \leq \frac{2\eta}{L_2},
\end{align*}
so that 
\begin{align*}
    \dotprod{\bgg_{k+1}, \pp_k} &\geq \dotprod{\bgg_k, \pp_k} + \dotprod{\pp_k, \HH_k \pp_k} + \eta\dotprod{\bgg_k, \pp_k}.
\end{align*}
Finally, for the CG scaling, $\dotprod{\bgg_k, \pCG_k} + \dotprod{\pCG_k, \HH_k\pCG_k} =0 $ (see the proof of \cref{prop:second-order descent} in \cref{apx:proof-basic-properties}) so 
\begin{align*}
    \dotprod{\bgg_{k+1}, \pCG_k} \geq \eta\dotprod{\bgg_k, \pCG_k},
\end{align*}
which is \cref{eqn:curvature condition}.

For the GM and MR scalings the second order descent condition \cref{eqn:second-order descent} doesn't hold with equality, so we need to lower bound the second order expansion. Starting from the positive curvature test $\dotprod{\pp_k, \HH \pp_k} \geq \sigma\vnorm{\pp_k}^2$ we have
\begin{align*}
    \dotprod{\bgg_{k+1}, \pp_k} &\geq \dotprod{\bgg_k, \pp_k} + \sigma \vnorm{\pp_k}^2 + \frac{L_2}{2} s_k \dotprod{\bgg_k, \pp_k}\vnorm{\pp_k}. \\ 
    &= \left(1 - s_k \sigma + \frac{L_2}{2} s_k \vnorm{\pp_k} \right)\dotprod{\bgg_k, \pp_k}.
\end{align*}
Next we apply a lower bound to the scaling, which requires \cref{ass:Hessian gradient smoothness condition}. In particular, from \cref{lemma:scaling lower bounds} we have $\sGM \geq \sMR \geq \sigma/L_1^2$ and so for $\pp_k = -s_k \bgg_k$ with $s_k \in \{\sGM, \sMR\}$ we have
\begin{align*}
    \dotprod{\bgg_{k+1}, \pp_k} \geq \left(1 - \sigma^2/L_1^2 + \frac{L_2}{2} s_k \vnorm{\pp_k} \right)\dotprod{\bgg_k, \pp_k}.
\end{align*}
Finally, applying \cref{eqn:curvature condition gradient upper bound - MRGM case} and \cref{prop:scaling upper bounds} we obtain
\begin{align*}
    s_k\vnorm{\pp_k} \leq \frac{\vnorm{\bgg_k}}{\sigma^2} \leq \frac{2}{L_2}\left(\eta - \left( 1 - \frac{\sigma^2}{L_1^2} \right) \right),
\end{align*}
and hence $\dotprod{\bgg_{k+1}, \pp_k} \geq \eta \dotprod{\bgg_k, \pp_k}$.
\end{proof}
We remark that an improved dependence on $\sigma$ and $L_1$ can be obtained in \cref{eqn:curvature condition gradient upper bound - MRGM case} for the GM scaling by applying the tighter bound $\sGM_k \geq 1/L_1$ (See \cref{remark:scaling lower bound}).

\subsection{Proof of \cref{thm:local convergence}}
\label{sec:appendix:proofs}

 We begin with a lemma, which bounds on the objective function sub-optimality and the gradient norm in a region local to a second order sufficient minima.     
\begin{lemma}
    \label{lem: mu and M local}
    Suppose that $\xx^\star$ satisfies the second order sufficient conditions. Furthermore, suppose that $r$ is chosen such that \cref{eqn:second-order sufficient conditions} holds for any $\xx \in \sB^*_r$, then
    \begin{align*}
        \frac{\mu}{2} \vnorm{\xx - \xx^{\star}}^2 \leq f(\xx) - f(\xx^{\star}) &\leq \frac{M}{2} \vnorm{\xx - \xx^{\star}}^2, \tageq\label{eqn:function local bound} \\
        f(\xx) - f(\xx^{\star}) &\leq \frac{1}{2\mu} \vnorm{\bgg}^2, \tageq\label{eqn:local PL inequality}\\
        \mu \vnorm{\xx - \xx^{\star}} \leq \vnorm{\bgg(\xx)} &\leq M\vnorm{\xx - \xx^{\star}}. \tageq\label{eqn:gradient local bound}
    \end{align*}
\end{lemma}
\begin{proof}
    For any $\xx, \yy \in \real^d$, the mean value theorem for twice differentiable functions implies
    \begin{align*}
        f(\yy) = f(\xx) + \dotprod{\bgg(\xx), \yy - \xx} +\frac{1}{2}\dotprod{\yy - \xx, \HH(\xx + t(\yy - \xx)(\yy - \xx)},
    \end{align*}
    for some $t \in (0,1)$. If $\xx, \yy \in \sB^*_r$, then $\xx + t(\yy - \xx) \in \sB^*_r$, and so 
    \begin{align*}
        \frac{\mu}{2}\vnorm{\yy -\xx}^2 \leq f(\yy) - f(\xx) - \dotprod{\bgg(\xx), \yy - \xx} \leq \frac{M}{2}\vnorm{\yy -\xx}^2. 
    \end{align*}
    By setting $\xx = \xx^{\star}$ and $\yy = \xx \in \sB^*_r$ we obtain \cref{eqn:function local bound}. Now, by rearranging we obtain
    \begin{align*}
        f(\yy) \geq f(\xx) + \dotprod{\bgg(\xx), \yy - x} + \frac{\mu}{2} \vnorm{\yy -\xx}^2. 
    \end{align*}
    Minimizing both sides over $\sB_r^\star$ gives \cref{eqn:local PL inequality} since
    \begin{align*}
        f(\xx^{\star}) = \min_{\yy \in \sB_r^\star} f(\yy) &\geq \min_{\yy \in \sB_r^\star} f(\xx) + \dotprod{\bgg(\xx), \yy - x} + \frac{\mu}{2} \vnorm{\yy -\xx}^2 \\
        &\geq \min_{\yy \in \real^d}f(\xx) + \dotprod{\bgg(\xx), \yy - \xx} + \frac{\mu}{2} \vnorm{\yy -\xx}^2 \\
        &= f(\xx) - \frac{1}{2\mu} \vnorm{\bgg(\xx)}^2.
    \end{align*}
    Finally, recall that for twice continuously differentiable functions we have
    \begin{align*}
        \bgg(\xx) = \int_0^1 \HH(\xx^{\star} + t(\xx - \xx^{\star}))(\xx - \xx^{\star}) \, dt.
    \end{align*}
    Hence, for $\xx \in \sB_r^\star$, we get
    \begin{align*}
        \vnorm{\bgg(\xx)}^2 =\dotprod{\bgg(\xx), \bgg(\xx)} &= \dotprod{\int_0^1 \HH(\xx^{\star} + t(\xx - \xx^{\star}))(\xx - \xx^{\star}) \, dt, \int_0^1 \HH(\xx^{\star} + s(\xx - \xx^{\star}))(\xx - \xx^{\star}) \, ds} \\
        &= \int_0^1 \int_0^1 \dotprod{ \HH(\xx^{\star} + t(\xx - \xx^{\star}))(\xx - \xx^{\star}) , \HH(\xx^{\star} + s(\xx - \xx^{\star}))(\xx - \xx^{\star}) }\, ds \, dt \\
        &= \int_0^1 \int_0^1 \dotprod{\xx - \xx^{\star} , \HH(\xx^{\star} + t(\xx - \xx^{\star}))\HH(\xx^{\star} + s(\xx - \xx^{\star}))(\xx - \xx^{\star}) }\, ds \, dt .
    \end{align*}
    For positive definite matrices $\AA$ and $\BB$, we have $\lambda_{\tmin}(\AA)\lambda_{\tmin}(\BB) \leq \lambda_{\tmin}(\AA\BB)$ and $\lambda_{\tmax}(\AA\BB) \leq \lambda_{\tmax}(\AA)\lambda_{\tmax}(\BB)$. Now since $t, s \in [0, 1]$ we have
    \begin{align*}
        \mu^2\eye \preceq \HH(\xx^{\star} + t(\xx - \xx^{\star}))\HH(\xx^{\star} + s(\xx - \xx^{\star})) \preceq M^2\eye,
    \end{align*}
    which gives the bounds in \cref{eqn:gradient local bound}.
\end{proof}

With \cref{lem: mu and M local} in hand we prove \cref{thm:local convergence}.

\begin{proof}[Proof of \cref{thm:local convergence}]
    Since $\bgg(\xx^{\star}) = 0$ and the gradient is continuous, we can choose $r' \leq r$ small enough that that \cref{eqn:gradient line search termination} and \cref{eqn:second-order sufficient conditions} hold on $\xx \in \sB_{r'}^\star$.
    Now suppose that $\xx_k \in \sB_{r'}^\star$, \cref{thm:unit step size acceptance,eqn:armijo condition,eqn:local PL inequality} imply
    \begin{align*}
        f(\xx_{k+1}) &\leq f(\xx_k) + \rho \dotprod{\pp_k, \bgg_k} = f(\xx_k) - s_k \rho \vnorm{\bgg_k}^2 \leq f(\xx_k) - 2 s_k  \mu \rho\left( f(\xx_k) - f(\xx^{\star}) \right),
    \end{align*}
    which implies
    \begin{align}
    \label{eqn:local function sub-optimality recursion}
    f(\xx_{k+1}) - f(\xx^{\star}) &\leq \left( 1 -  2\mu\rho s_k \right)(f(\xx_k) - f(\xx^{\star})).
    \end{align}
    Since $\HH_k \succ 0$, we only need to consider $\LPC$ and $\SPC$ scalings. If $\bgg_k$ is a $\SPC$ direction, then following a similar line of reasoning as that in  \cref{lemma:scaling lower bounds}, with $M$ taking the place of $L_1$, we get $s_{k} = \sSPC \geq 1/M$. Note that for the MR scaling, this bound follows from $\mu \eye \preceq \HH_k \preceq M\eye$ and \cref{lemma:tighter MR lower bound}. Hence, \cref{eqn:local function sub-optimality recursion} gives
    \begin{align}
        f(\xx_{k+1}) - f(\xx^{\star}) &\leq \left( 1 -  \frac{2\mu\rho}{M}\right)(f(\xx_k) - f(\xx^{\star})). \label{eqn:local function sub-optimality recursion SPC CASE}
    \end{align}
    If $\sigma < \mu$ then clearly $\bgg_k$ cannot be an $\LPC$ direction and so \cref{eqn:local function sub-optimality recursion SPC CASE} yields the desired recursion in function sub-optimality. On the other hand, if $\sigma \geq \mu$ and $\bgg_k$ is an $\LPC$ direction, then we simply apply $\sLPC_{\tmin} \leq \sLPC$ to \cref{eqn:local function sub-optimality recursion} to obtain
    \begin{align}
        f(\xx_{k+1}) - f(\xx^{\star}) &\leq \left( 1 -  2\mu\sLPC_{\tmin}\rho \right)(f(\xx_k) - f(\xx^{\star})). \label{eqn:local function sub-optimality recursion LPC CASE}
    \end{align}
    The desired recursion follows from combining \cref{eqn:local function sub-optimality recursion SPC CASE} and \cref{eqn:local function sub-optimality recursion LPC CASE}.
    
    It remains to show that the iterates remain in $\sB_{r'}^\star$. To this end, define the set
    \begin{align*}
        \sF = \left\{ x \mid f(\xx) - f(\xx^{\star}) \leq \frac{\mu \sigma^{2} r^2}{2\left(\sigma + M\right)^2} \right\}.
    \end{align*}
    Now suppose that an iterate $\xx_k$ is sufficiently close to $\xx^{\star}$ such that $\xx_k \in B_{r'}^\star \cap \sF$. Such  $\xx_k$ exists, as implied by the upper bound in \cref{eqn:function local bound}. We now show that all subsequent iterates remain in $B_{r'}^\star \cap \sF$. Indeed, from the earlier part of the proof, we get $f(\xx_{k+1}) - f(\xx^{\star}) \leq (1-\tau)(f(\xx_k) - f(\xx^{\star}))$, which implies $ \xx_{k+1} \in \sF$. On the other hand, from $\xx_{k+1} = \xx_k - s_k \bgg_k$ (since $\alpha=1$ is accepted by the line search), $s \leq 1/\sigma$ (\cref{prop:scaling upper bounds}), the upper bound in \cref{eqn:gradient local bound}, and the lower bound in $\cref{eqn:function local bound}$, it follows that 
    \begin{align*}
        \vnorm{\xx_{k+1} - \xx^{\star}} &\leq \vnorm{\xx_k - \xx^{\star}} + s_k\vnorm{\bgg_k} \leq \left(1 + \frac{M}{\sigma} \right)\vnorm{\xx_k - \xx^{\star}} \\
        &\leq \left(1 + \frac{M}{\sigma}\right) \sqrt{\frac{2}{\mu} (f(\xx_k) - f(\xx^{\star}))} \leq r,
    \end{align*}
    which gives $\xx_{k+1} \in \sB_{r'}^\star$, and hence $\xx_{k+1} \in \sB_{r'}^\star \cap \sF$.
\end{proof}

\subsection{Convergence in Gradient Norm and Proof of \cref{thm:local convergence MR second-order sufficient}} 
\label{apx:convergence in gradient norm}

We prove \cref{thm:local convergence MR second-order sufficient} by considering a more general result, \cref{prop:gradient norm recursion}, to which \cref{thm:local convergence MR second-order sufficient} is a corollary. Our analysis begins by utilizing \cref{ass:Hessian smoothness condition} and a bound on the residual to obtain a recursion in the gradient norm. 

\begin{lemma} \label{lemma:gradient norm recursion}
    Considering \cref{ass:Hessian smoothness condition} and supposing $\vnorm{\HH\bgg} \neq 0$ and $\vnorm{\bgg} \neq 0$ we have the following recursion for the gradient with the step $\pp=-\sMR\bgg$
    \begin{align*}
        \vnorm{\bgg(\xx + \pp)} &\leq \frac{L_2 }{2}\vnorm{\pp}^2 + \sqrt{1 - \cos^2{\theta}}\vnorm{\bgg}, \tageq\label{eqn:MR case upper bound}
    \end{align*}
    where $\theta \defeq \measuredangle(\bgg, \HH \bgg)$ is the angle between $\bgg$ and $\HH \bgg$.
\end{lemma}
\begin{proof}
    Adding and subtracting appropriate values and applying the triangle inequality we have
    \begin{align*}
    \vnorm{\bgg(\xx + \pp)} &= \vnorm{\bgg(\xx + \pp) + \bgg + \HH \pp - \bgg - \HH \pp } \\
    &\leq \vnorm{\bgg(\xx + \pp) - \bgg - \HH \pp} + \vnorm{\bgg + \HH \pp} \\
    &= \vnorm{\bgg(\xx + \pp) - \bgg - \HH \pp} + \vnorm{\rr}.
    \end{align*}
    Now, \cref{ass:Hessian smoothness condition,eqn:gradient upper bound} gives
    \begin{align*}
        \vnorm{\bgg(\xx + \pp)} \leq \frac{L_2}{2}\vnorm{\pp}^2 + \vnorm{\rr}.
    \end{align*}
    We can handle the residual term using the properties of MR scaling. Indeed, applying the definition of $\sMR$ for $\HH\bgg \neq 0$ we have
    \begin{align*}
        \vnorm{\rr}^2 = \vnorm{\bgg - \sMR \HH \bgg}^2 &= \vnorm{\bgg}^2 - 2\sMR \dotprod{\bgg, \HH \bgg} + (\sMR)^2\vnorm{\HH \bgg}^2 \\
        &= \vnorm{\bgg}^2 - 2\frac{\dotprod{\bgg, \HH \bgg}^2 }{\vnorm{\HH \bgg}^2 } + \frac{\dotprod{\bgg, \HH \bgg}^2 }{ \vnorm{\HH \bgg}^2 } \\
        &= \vnorm{\bgg}^2\left( 1 - \frac{\dotprod{\bgg, \HH \bgg}^2 }{\vnorm{\HH \bgg}^2 \vnorm{\bgg}^2 } \right) \\
        &= \vnorm{\bgg}^2\left( 1 - \cos^2{\theta} \right),
    \end{align*}
    Putting this all together yields
    \begin{align*}
        \vnorm{\bgg(\xx + \pp)} &\leq \frac{L_2}{2}\vnorm{\pp}^2 + \sqrt{1 - \cos^2{\theta}} \vnorm{\bgg}.
    \end{align*}
    which is \cref{eqn:MR case upper bound}.
\end{proof}

To apply \cref{eqn:MR case upper bound} for convergence in gradient norm, the magnitude of $(\sMR)^2$ needs to be bounded from above and $\cos{\theta}$ must remain bounded away from zero. To that end, we introduce \cref{ass:MR local assumption}, which excludes cases where the MR scaling becomes unbounded and $\theta \approx \pi/2$, meaning $\HH\bgg$ and $\bgg$ become nearly orthogonal.

\begin{assumption}\label{ass:MR local assumption}
Given $\xx_0$, there exists constants $\nu_{0} > 0$ and $\mu_{0} > 0$, possibly depending on $\xx_0$, such that for all the iterates of the form $\xx_{k+1} = \xx_k - \sMR_{k} \bgg_k$, $k \geq 0$, with $\bgg_k \neq 0$ we have 
\begin{align*}
    \mu_0\vnorm{\bggk}^2 &\leq | \dotprod{ \bggk, \HHk \bggk} |, \tageq\label{eqn:MR curvature absolute lower bound} \\
    \nu_{0} &\leq \cos^2{\theta_{k}}. \tageq\label{eqn:cosine lower bound}
\end{align*} 
\end{assumption}
 
\begin{remark}
    The conditions in \cref{ass:MR local assumption} are general and somewhat unconventional. However, as we shall see they are a weakening of more typical conditions. 
    Indeed, any smooth and strongly convex function on the set containing the iterates satisfies \cref{ass:MR local assumption}. 
    Furthermore, as we shall see in the proof of \cref{thm:local convergence MR second-order sufficient}, the second order sufficient conditions subsume \cref{ass:MR local assumption} for $\xx_0$ close enough to a minima. 
    The gradient curvature condition, \cref{eqn:MR curvature absolute lower bound}, is a weakening of strong convexity, as the curvature along gradient direction need only be uniformly bounded away from zero (instead of uniformly positive). 
    Applying the Cauchy-Schwarz inequality to \cref{eqn:MR curvature absolute lower bound} we see that $\vnorm{\bgg_k} > 0$ implies $\vnorm{\HH_k \bgg_k} > 0$, which ensures that $\cos{\theta_k}$ and $\sMR_k$ remain well defined along the path of the iterates.
    On the other hand, the cosine condition \cref{eqn:cosine lower bound} is implied by \cref{eqn:MR curvature absolute lower bound} along with a bound of the form $\vnorm{\HH_k \bgg_k} \leq L \vnorm{\bgg_k}$.
    However, \cref{eqn:cosine lower bound} is also more generally applicable. For example, when $d=1$, \cref{eqn:cosine lower bound} holds with $\cos^2{\theta}=1$, so long as $\HH \neq 0$. Another example, is $f(\xx) = \vnorm{\xx}^3/6$ where, again, one can show $\cos^2{\theta}=1$. 
\end{remark}

An immediate consequence of \cref{eqn:MR curvature absolute lower bound} and $\vnorm{\HH_k\bgg_k} \geq \abs{\dotprod{\bgg_k, \HH_k \bgg_k}}/\vnorm{\bgg_k}$ is that
\begin{align*}
    (\sMR_k)^2 = \frac{(\dotprod{\bgg_k, \HH_k \bgg_k})^2 }{ \vnorm{\HH_k \bgg_k}^4} \leq \frac{\vnorm{\bgg_k}^4}{(\dotprod{\bgg_k, \HH_k \bgg_k})^2} \leq \frac{1}{\mu_0^2}.
\end{align*}
Therefore by applying \cref{ass:MR local assumption,eqn:MR case upper bound}, we get  
\begin{align*}
    \vnorm{\bgg(\xx + \pp)} \leq \left(\frac{L_2}{2\mu_{0}^2}\vnorm{\bgg} + \sqrt{ 1 - \nu_{0}} \right)\vnorm{\bgg}. \tageq\label{eqn:MR case gradient recursion}
\end{align*}
The recursion in the gradient norm, \cref{eqn:MR case gradient recursion}, is linear-quadratic in a manner similar to the results obtained for Newton-MR method applied to invex problems; see \cite{roostaNewtonMRInexactNewton2022}. 
It follows directly from \cref{eqn:MR case gradient recursion} that, when the gradient is small enough, i.e., 
\begin{align*}
    \vnorm{\bgg} < G_{0} \defeq 2 \mu_{0}^2\left(1 - \sqrt{1 - \nu_{0}} \right)/L_2, \tageq\label{eqn:G_0}
\end{align*}
the gradient norm decreases monotonically. Naturally, once the gradient hits this monotonic phase, every subsequent iterate will also satisfy $\vnorm{\bgg}\leq G_{0}$. These properties are all formalized to a rate in \cref{prop:gradient norm recursion}.

\begin{proposition}[MR Scaling Convergence Gradient Norm] \label{prop:gradient norm recursion}
    Consider \cref{ass:Hessian smoothness condition}, and iterations of the form $\xx_{k+1} = \xx_k - \sMR_{k} \bgg_k$ starting from $\xx_0$ for which $\vnorm{\bgg(\xx_0)} < G_{0}$ and \cref{ass:MR local assumption} holds. For all $k \geq 1$, we have  $\vnorm{\bgg_k} < \vnorm{\bgg_{k-1}}$ and 
    \begin{align*}
        \vnorm{\bgg_k} \leq \frac{G_{0} }{G_{0} - \vnorm{\bgg_0}} \left( 1 -  \frac{1 - \sqrt{1 - \nu_{0}}}{2 - \sqrt{1 - \nu_{0}}} \right)^k \vnorm{\bgg_0}.
    \end{align*}
\end{proposition}
\begin{proof}
Since \cref{ass:MR local assumption} holds with $\xx_{0}$ and $\vnorm{\bgg(\xx_0)}< G_{0}$, \cref{eqn:MR case gradient recursion} implies that $\vnorm{\bgg_k} \leq \vnorm{\bgg_{k-1}} < G_{0}$ for $k \geq 1$. To obtain a quantitative rate, we follow a similar line of reasoning as \citep[Theorem 1.2.4]{nesterovIntroductoryLecturesConvex2004}. Let $a_k = \vnorm{\bgg_k} L_2/(2\mu_{0}^2) $ and $q = 1 - \sqrt{1 - \nu_{0}}$. We have $a_k < q$ and
\begin{align*}
    a_{k+1} \leq a_k^2 + \sqrt{1- \nu_{0}} a_k = (a_k + \sqrt{1 - \nu_{0}}) a_k = (1 + a_k - q) a_k.
\end{align*}
Also, since $\alpha_k - q < 0$ by \cref{eqn:G_0},
\begin{align*}
    a_k(1 + a_k - q)  = \left(\frac{(1 - (a_k - q))(1+(a_k-q))}{(1 - (a_k -q))}\right) a_k  = \left(\frac{(1 - (a_k -q)^2)}{1 - (a_k - q)}\right) a_k  \leq \frac{a_k}{1+ q - a_k},
\end{align*}
where the inequality follows from $(1 - (a_k -q)^2)\leq 1$. Hence,  
\begin{align*}
    \frac{q}{a_{k+1}} - 1 \geq (1 + q) \left(\frac{q}{a_k} - 1 \right) .
\end{align*}
Applying this iteratively gives
\begin{align*}
    \frac{q}{a_k} - 1 \geq (1+q)^k \left(\frac{q}{a_0} - 1\right) = (1+q)^k \left(\frac{2\mu_{0}^2(1 - \sqrt{1-\nu_{0}})}{L_2\vnorm{\bgg_0}} - 1\right) = (1 + q)^k \left( \frac{G_{0}}{\vnorm{\bgg_0}} - 1 \right).
\end{align*}
Rearranging, we obtain 
\begin{align*}
    a_k &\leq \frac{q \vnorm{\bgg_0} }{(1 + q)^k(G_{0} - \vnorm{\bgg_0}) + \vnorm{\bgg_0}}.
\end{align*}
Finally, by substituting in the definition of $a_k$ and $q$, we obtain
\begin{align*}
    \vnorm{\bgg_k} &\leq  \frac{G_{0} \vnorm{\bgg_0}}{(2 - \sqrt{1 - \nu_{0}})^k(G_{0} - \vnorm{\bgg_0})} = \frac{G_{0} }{G_{0} - \vnorm{\bgg_0}} \left( 1 -  \frac{1 - \sqrt{1 - \nu_{0}}}{2 - \sqrt{1 - \nu_{0}}} \right)^k \vnorm{\bgg_0}.
\end{align*}
\end{proof}

Finally, we prove \cref{thm:local convergence MR second-order sufficient} by specializing to the case where the second order sufficient conditions hold.

\begin{proof}[Proof of \cref{thm:local convergence MR second-order sufficient}]
    It suffices to find a region where the conditions of \cref{prop:gradient norm recursion} are satisfied, that the iterates will not leave. Similar to the proof of \cref{thm:local convergence}, when the second order sufficient conditions hold, there exists a ball $\sB_r^\star$ such that \cref{eqn:second-order sufficient conditions} holds. This in turn implies \cref{eqn:gradient local bound}, as well as \cref{eqn:MR curvature absolute lower bound,eqn:cosine lower bound} with $\mu_{0} = \mu$ and $\nu_{0} = \mu^2/M^2$, respectively, for any  $\xx \in \sB_r^\star$. With these $\mu_{0}$ and $\nu_{0}$ in hand, consider $G_{0}$ in \cref{eqn:G_0} and define  
    \begin{align*}
        \sG = \left\{ \xx \mid 0 < \vnorm{\bgg(\xx)} \leq \min\left\{G_{0}, \frac{r \mu}{M \mu + 1} \right\} \right\}.
    \end{align*}
    Suppose $\xx_k \in \sB_r^\star \cap \sG$ (such an $\xx_k$ exists by the RHS of \cref{eqn:gradient local bound}). From \cref{eqn:MR case gradient recursion}, we get $\vnorm{\bgg_{k+1}} < \vnorm{\bgg_k}$, which in turn implies $\xx_{k+1} \in \sG$. Applying the upper bound in \cref{eqn:gradient local bound}, $\sMR_{k} \leq 1/\mu$ and $\xx_k\in \sG$, gives 
    \begin{align*}
        \vnorm{\xx_{k+1} - \xx^{\star}} &\leq \vnorm{\xx_k - \xx^{\star}} + \abs{\sMR_{k}}\vnorm{\bgg_k} 
        \leq \left(M + \frac{1}{\mu} \right) \vnorm{\bgg_k} 
        \leq r.
    \end{align*}
    That is, $\xx_{k+1} \in \sB_r^\star \cap \sG$. Hence, \cref{ass:MR local assumption} holds for all of the iterates if $\xx_0 \in \sB_r^\star \cap \sG$. \cref{prop:gradient norm recursion} can be applied to obtain the convergence rate. 
\end{proof}

Finally, we remark that the iteration $\xx_{k+1} = \xx_k - \sMR_k\bgg_k$, as applied in \cref{thm:local convergence MR second-order sufficient}, can occur as a special case of \cref{alg:scaled gradient}. In particular, if $\sigma$ is chosen $\sigma < \mu $ then the $\SPC$ scaling (in particular, the MR scaling) will always be selected. Moreover, if $\xx_0$ is chosen sufficiently close to $\xx^\star$, \cref{eqn:gradient line search termination} will hold; implying that $\alpha=1$ will be acceptable to the line search. Indeed, \cref{eqn:gradient line search termination} will be satisfied for all future iterations by monotonicity of the gradient norm.

\subsection{Inexact Hessian} \label{apx:inexact Hessian}

\subsubsection{Statement of Results}

We now consider the case where the Hessian can only be accessed through an inexact estimate $\HHt$. Specifically, we demonstrate that under certain conditions on the inexactness tolerance, we can attain similar unit step size and convergence guarantees as in the exact case. 
For clarity, we defer the proofs of our results to the end of this section. We begin with a condition on the allowable inexactness in the Hessian.
\begin{assumption}[Hessian Error Bound] \label{ass:Hessian inexactness bound}
    For a given $\Delta_H > 0$ and $\xx$, we can produce $\HHt$ such that
    \begin{align*}
        \abs{\dotprod{\bgg(\xx), (\HH(\xx) - \HHt(\xx))\bgg(\xx)}} \leq \Delta_H \vnorm{\bgg(\xx)}^2. \tageq\label{eqn:Hessian error bound}
    \end{align*}
\end{assumption}
This condition requires only that the inexactness is bounded along the gradient direction, a weaker requirement than a direct bound on the Hessian error, $\HH(\xx) - \HHt(\xx)$.
To ground \cref{ass:Hessian inexactness bound}, consider an objective which can be written as a finite-sum
\begin{align*}
    f(\xx) \defeq \frac{1}{n}\sum_{i=1}^n f_i(\xx).
\end{align*}
This is a formulation which arises often in machine learning as part of the empirical risk minimization framework \cite{murphyMachineLearningProbabilistic2012}. In the ``big data'' regime, where $n \gg 1$, optimization algorithms often use sub-sampling to estimate gradients or Hessians. Suppose we estimate the Hessian at each iteration via
\begin{align*}
    \HHt(\xx) = \frac{1}{|\sI_H|} \sum_{i \in \sI_H} \HH_i(\xx), \tageq\label{eqn:Hessian subsampling}
\end{align*}
where $\sI_H \subseteq \{1, \ldots, n\}$ is a minibatch of indices uniformly, with replacement. In this case, \cref{ass:Hessian inexactness bound} holds with high probability if the sample size is sufficiently large and the individual Hessians, $\HH_i$, are uniformly bounded along the gradient direction (see \cref{sec:subsampled hessian} for details). 

When the estimated Hessian $\HHt$ replaces $\HH$ in \cref{alg:scaling selection}, the curvature tests and resulting scalings (denoted by $\st$) depend on $\HHt$ rather than $\HH$. Fortunately, many of the key properties from the deterministic case still hold for inexact Hessian. Indeed, $\SPC$ steps satisfy $\sigma \vnorm{\bgg}^2 < \dotprod{\bgg, \HHt \bgg}$, $\LPC$ steps satisfy $0 \leq \dotprod{\bgg, \HHt \bgg} \leq \sigma \vnorm{\bgg}^2$ and $\NC$ steps satisfy $\dotprod{\bgg, \HHt \bgg} < 0$. The second order descent condition \cref{eqn:second-order descent} from \cref{prop:second-order descent} continues to hold with respect to $\HHt$, that is, for $\ppt = -\st \bgg$
\begin{align*}
    \dotprod{\bgg, \ppt} + \dotprod{\ppt, \HHt \ppt} \leq 0 \tageq\label{eqn:inexact second order descent}
\end{align*}
The upper bounds from \cref{prop:scaling upper bounds} also continue to hold, as they arise due to the curvature test. The search direction, $\ppt$, remains collinear to to the gradient, so that $\dotprod{\bgg, \ppt} = - \vnorm{\bgg} \vnorm{\ppt}$ and \cref{ass:Hessian smoothness condition} can be applied along $\ppt$ to ensure that \cref{eqn:function value upper bound} holds. The Armijo condition for the line search is given by 
\begin{align}
    \label{eqn:inexact line search}
    f(\xx + \alpha \ppt) \leq f(\xx) + \alpha \rho \dotprod{\bgg, \ppt},
\end{align}
for $\rho \in (0, 1/2)$. With these ideas in hand we can derive results reminiscent of \cref{sec:global} for inexact case. Beginning with a unit step size guarantee that applies when $\Delta_H$ is sufficiently small.

\begin{proposition} \label{prop:inexact unit step size acceptance}
    Consider \cref{ass:Hessian smoothness condition} and \cref{ass:Hessian inexactness bound} with $\Delta_H \leq 2 \min\left\{ \sigma\left( {1}/{2} - \rho\right), {(1-\rho)}/{\sNC_{\tmax}}\right\}$. 
    If 
    \begin{align*}
        \vnorm{\bgg} \leq  6 \min\left\{\frac{\sigma^2 \left(\frac{1}{2} - \rho - \frac{\Delta_H}{2\sigma} \right)}{L_2 },   \frac{\left( 1 - \rho - \frac{\sNC_{\tmax}\Delta_H}{2} \right) }{L_2 (\sNC_{\tmax})^2}\right\},
    \end{align*}
    then the Armijo condition \cref{eqn:inexact line search} is satisfied with $\alpha=1$.
\end{proposition}

Next we consider a global convergence guarantee for the inexact case. Much like the analysis in the exact case, this requires an additional smoothness condition, however this time the condition is based on the inexact Hessian.
\begin{assumption}[Inexact Hessian-gradient Smoothness] 
    \label{ass:inexact Hessian gradient smoothness condition}
    There exists $0 \leq \Lt_1< \infty$ such that for all $\xx \in \real^d$, we have $\|\HHt(\xx) \bgg(\xx)\| \leq \Lt_1 \vnorm{\bgg(\xx)}$.
\end{assumption}
Much like the exact Hessian case, \cref{ass:inexact Hessian gradient smoothness condition} relaxes the typical Lipschitz gradient assumption used for global convergence in gradient descent. Returning to the sub-sampled Hessian, \cref{eqn:Hessian subsampling}, \cref{ass:inexact Hessian gradient smoothness condition} clearly holds if the individual Hessians, $\HH_i$, are uniformly bounded along the gradient direction. 

Much like the the exact Hessian case (compare with \cref{lemma:scaling lower bounds}), under \cref{ass:inexact Hessian gradient smoothness condition} we have a lower bound on the scaling magnitudes 
\begin{align*}
    \frac{\sigma}{\Lt_1^2} \leq \sMR \leq \sGM \leq \sCG. \tageq\label{eqn:inexact step size lower bound}
\end{align*}
We now state the global convergence result.
\begin{proposition} \label{prop:inexact global convergence}
    Consider \cref{ass:Hessian smoothness condition}, \cref{ass:Hessian inexactness bound} for any $\Delta_H \geq 0$, and \cref{ass:inexact Hessian gradient smoothness condition}. Suppose $f$ is lower bounded. For any $0<\eg < 1$, after at most $K \in \bigO{\eg^{-2}}$  iterations of \cref{alg:scaled gradient}, with $\HH$ replaced with $\HHt$, we have $\vnorm{\bgg_k} \leq \varepsilon_{\bgg}$ for some $0\leq k\leq K$.
\end{proposition}

Much like \cref{prop:global convergence}, the rate obtained in \cref{prop:inexact global convergence} matches the worst case rate for gradient descent on a function with Lipschitz smooth gradients. Interestingly, unlike \cref{prop:inexact unit step size acceptance}, \cref{prop:inexact global convergence} imposes no specific requirement on $\Delta_H$. This is because small step sizes (i.e., $\alpha<1$) can mitigate Hessian noise. Imposing a stronger condition on $\Delta_H$ cannot offer an improvement in the overall worst case rate as the inexact rate already matches that of the exact case. However, in \cref{remark:remark on inexact global convergence}, we discuss how bounding $\Delta_H$ can offer an improvement in the per-iteration decrease for certain step types. 

Finally, we remark that when \cref{ass:inexact Hessian gradient smoothness condition} holds only with high probability, the conclusions of \cref{prop:inexact unit step size acceptance} and \cref{prop:inexact global convergence} may also hold only with high probability. A specific issue arises for \cref{prop:inexact global convergence}, as \cref{eqn:Hessian error bound} must hold at \emph{each iteration}. However, by ensuring the probability of \cref{eqn:Hessian error bound} failing is sufficiently small, sufficient decrease over any finite span of iterations can be guaranteed with arbitrarily high probability.

\subsubsection{Proof of Results} 

\begin{proof}[Proof of \cref{prop:inexact unit step size acceptance}]
We begin by adding and subtracting $\dotprod{\ppt, \HHt\ppt}$ to the upper bound in \cref{eqn:function value upper bound} and applying \cref{ass:Hessian inexactness bound}
\begin{align*}
    f(\xx + \alpha \pp) - f(\xx) 
    &\leq \alpha \dotprod{\bgg, \ppt} + \frac{\alpha^2}{2}\dotprod{\ppt, (\HH - \HHt)\ppt} + \frac{\alpha^2}{2}\dotprod{\ppt, \HHt \ppt} + \frac{L_2 \alpha^3}{6} \vnorm{\ppt}^3  \\
    &\leq \alpha \dotprod{\bgg, \ppt} + \frac{\alpha^2\Delta_H }{2} \vnorm{\ppt}^2 + \frac{\alpha^2}{2}\dotprod{\ppt, \HHt \ppt} + \frac{L_2 \alpha^3}{6} \vnorm{\ppt}^3. \tageq\label{eqn:inexact function value upper bound}
\end{align*}
The upper bound in \cref{eqn:inexact function value upper bound} now involves the curvature of the inexact Hessian, $\dotprod{\ppt, \HHt \ppt}$, which we can control with the curvature tests and second order descent condition.
Starting with non-negative curvature case (encompassing $\SPC$ and $\LPC$ scalings) with $\dotprod{\bgg, \HHt \bgg} \geq 0$. Taking \cref{eqn:inexact function value upper bound}, subtracting $\rho \alpha \dotprod{\bgg, \ppt}$ and applying $\alpha \leq 1$ and \cref{eqn:inexact second order descent} we obtain
\begin{align*}
     f(\xx + \alpha \pp) - f(\xx) - \rho \alpha \dotprod{\bgg, \ppt} 
     &\leq \alpha\left(\frac{1}{2} - \rho \right) \dotprod{\bgg, \ppt} + \frac{\alpha}{2} \left(\dotprod{\bgg, \ppt} + \dotprod{\ppt, \HHt \ppt}\right) \\
     &+ \frac{\alpha^2}{2} \Delta_H \vnorm{\ppt}^2  + \frac{L_2 \alpha^3}{6} \vnorm{\ppt}^3 \\
     &\leq -\alpha \left(\frac{1}{2} - \rho\right)\vnorm{\bgg} \vnorm{\ppt} + \frac{\alpha^2}{2} \Delta_H \vnorm{\ppt}^2  + \frac{L_2 \alpha^3}{6} \vnorm{\ppt}^3 \\
     &=\left(-\left(\frac{1}{2} - \rho\right) + \frac{\alpha \st \Delta_H}{2}   + \frac{L_2 \alpha^2 \st }{6} \vnorm{\ppt} \right) \alpha \vnorm{\ppt} \vnorm{\bgg} \\
     &\leq \left(-\left(\frac{1}{2} - \rho\right) + \frac{\alpha \Delta_H}{2\sigma}   + \frac{L_2 \alpha^2  }{6 \sigma } \vnorm{\ppt} \right) \alpha \vnorm{\ppt} \vnorm{\bgg}. \tageq\label{eqn:inexact case positive function upper bound}
\end{align*}
In the final line we apply $\st \leq 1/\sigma$, which holds for both the $\LPC$ and $\SPC$ scalings. The line search criteria \cref{eqn:inexact line search} is clearly satisfied if the upper bound in \cref{eqn:inexact case positive function upper bound} is negative. That is,
\begin{align*}
    \vnorm{\ppt} \leq \frac{6 \sigma}{L_2 \alpha^2}\left(\left(\frac{1}{2} - \rho \right) - \frac{\alpha \Delta_H}{2\sigma} \right).
\end{align*}
This bound is nontrivial if 
\begin{align*}
    \Delta_H \leq \frac{2\sigma}{\alpha}\left( \frac{1}{2} - \rho\right).
\end{align*}
Now we consider the negative curvature case, where$\dotprod{\bgg, \HHt \bgg} <0$. Applying the negative curvature condition to \cref{eqn:inexact function value upper bound} and subtracting $\alpha \rho \dotprod{\bgg, \ppt}$
\begin{align*}
     f(\xx + \alpha \pp) - f(\xx) -\alpha \rho \dotprod{\bgg, \ppt} &\leq \alpha(1-\rho)\dotprod{\bgg, \ppt} + \frac{\alpha^2 \Delta_H }{2}\vnorm{\ppt}^2 + \frac{\alpha^2}{2}\dotprod{\ppt, \HHt \ppt} + \frac{L_2 \alpha^3}{6} \vnorm{\ppt}^3 \\
     &\leq -\alpha(1-\rho)\vnorm{\bgg}\vnorm{\ppt} + \frac{\alpha^2 \Delta_H }{2}\vnorm{\ppt}^2 + \frac{L_2 \alpha^3}{6} \vnorm{\ppt}^3 \\
     &= \left(-(1-\rho)  + \frac{\alpha\sNCt}{2} \Delta_H  + \frac{L_2 \alpha^2\sNCt}{6} \vnorm{\ppt} \right) \alpha \vnorm{\bgg}\vnorm{\ppt} \\
     &\leq \left(-(1-\rho)  + \frac{\alpha\sNC_{\tmax}}{2} \Delta_H  + \frac{L_2 \alpha^2\sNC_{\tmax}}{6} \vnorm{\ppt} \right) \alpha \vnorm{\bgg}\vnorm{\ppt}, \tageq\label{eqn:inexact case negative function upper bound}
\end{align*}
where in the final line we we apply $\sNCt \leq \sNC_{\tmax}$. Again, if the upper bound in \cref{eqn:inexact case negative function upper bound} is negative \cref{eqn:inexact line search} is clearly satisfied. In particular, \cref{eqn:inexact case negative function upper bound} is negative if
\begin{align*}
    \vnorm{\ppt} \leq  \frac{6}{L_2 \sNC_{\tmax} \alpha^2}\left( 1 - \rho - \frac{\alpha \sNC_{\tmax}\Delta_H}{2} \right).
\end{align*}
This bound is nontrivial if 
\begin{align*}
    \Delta_H \leq \frac{2(1-\rho)}{\alpha \sNC_{\tmax}}.
\end{align*}
The result follows from combining the two tolerances on $\Delta_H$, setting $\alpha=1$ and using the fact that $\vnorm{\ppt} \leq \vnorm{\bgg}/\sigma$ in the non-negative curvature case and $\vnorm{\ppt} \leq \sNC_{\tmax} \vnorm{\bgg}$ in the negative curvature case.
\end{proof}

We now move on to a proof of \cref{prop:inexact global convergence}. First we require a technical lemma.
\begin{lemma} \label{lemma:technical lower bound}
    For $A \geq 0$, $x \geq0 $ and $B > 0$ we have 
    \begin{align*}
        -A + \sqrt{A^2 + Bx } \geq \frac{B}{A + \sqrt{A^2 + B}} \min\{x, 1\}.
    \end{align*}
\end{lemma}
\begin{proof}
The $x=0$ case is trivial, so take $x >0$. Multiplying the numerator and denominator by a conjugate
\begin{align*}
    -A + \sqrt{A^2 + Bx } &= \frac{(-A + \sqrt{A^2 + Bx})(A + \sqrt{A^2 + Bx})}{A + \sqrt{A^2 + Bx}} \\
    &= \frac{-A^2 + A^2 + Bx}{A + \sqrt{A^2 + Bx}} \\
    &= \frac{Bx}{A + \sqrt{A^2 + Bx}}.
\end{align*}
Now consider two cases. If $x < 1$ then $A + \sqrt{A^2 + Bx} < A + \sqrt{A^2 + B}$ so that 
\begin{align*}
     -A + \sqrt{A^2 + Bx} = \frac{Bx}{A + \sqrt{A^2 + Bx}} \geq \frac{B}{A + \sqrt{A^2 + B}}x.
\end{align*}
On the other hand, if $x \geq 1$, we rearrange to obtain
\begin{align*}
    -A + \sqrt{A^2 + Bx} = \frac{Bx}{A + \sqrt{A^2 + Bx}} = \frac{B}{A/x + \sqrt{A^2/x^2 + B/x}}.
\end{align*}
The denominator of this expression is decreasing in $x$, indeed, we clearly have $A/x + \sqrt{A^2/x^2 + B/x} \leq A + \sqrt{A^2 + B}$ so that
\begin{align*}
    -A + \sqrt{A^2 + B} \geq \frac{B}{A + \sqrt{A^2 + B}}.
\end{align*}
\end{proof}

\begin{proof}[Proof of \cref{prop:inexact global convergence}] 

We begin by re-examining \cref{eqn:inexact case positive function upper bound,eqn:inexact case negative function upper bound} from the proof of \cref{prop:inexact unit step size acceptance}, to demonstrate that there is a positive step size for which \cref{eqn:inexact line search} certainly holds. 
We assume that \cref{alg:scaled gradient} has not terminated, i.e., $\vnorm{\bgg} > \eg$.
Starting with the negative curvature case, consider the upper bound in \cref{eqn:inexact case negative function upper bound}. The line search condition \cref{eqn:inexact line search} is certainly satisfied if
\begin{align*}
    C_2 \alpha^2 + C_1 \alpha + C_0 \leq 0,
\end{align*}
where 
\begin{align*}
    C_2 \defeq \frac{L_2 \sNC_{\tmax} \vnorm{\ppt}}{6}, \quad C_1 \defeq \frac{\sNC_{\tmax} \Delta_H}{2}, \quad C_0 \defeq -(1-\rho).
\end{align*}
This quadratic has two real roots, one positive and one negative. The positive root is given by 
\begin{align*}
    \alpha^\star_1 &= -\frac{C_1}{2C_2} + \sqrt{\frac{C_1^2}{4C_2^2} - \frac{C_0}{C_2}} \\
    &= -\frac{3\Delta_H }{2L_2 \vnorm{\ppt}} + \sqrt{\frac{9 \Delta_H^2}{4 L_2^2 \vnorm{\ppt}^2} + \frac{6(1-\rho)}{L_2 \sNC_{\tmax} \vnorm{\ppt}}},
\end{align*}
which implies that the largest step size for which \cref{eqn:inexact line search} is satisfied also must satisfy $\alpha \geq \alpha^\star_1$. Therefore, when \cref{eqn:inexact line search} is satisfied, we have
\begin{align*}
    f(\xx + \alpha \ppt) - f(\xx) 
    &\leq -\alpha\rho\vnorm{\ppt}\vnorm{\bgg} \\
    &\leq -\rho \left(-\frac{3\Delta_H }{2L_2} + \sqrt{\frac{9 \Delta_H^2}{4 L_2^2} + \frac{6(1-\rho)\vnorm{\ppt}} {L_2 \sNC_{\tmax}}} \right) \vnorm{\bgg} \\
    &< -\rho \left(- \frac{3\Delta_H}{2L_2}  + \sqrt{ \frac{9 \Delta_H^2}{4L_2^2} + \frac{6(1-\rho)\sNC_{\tmin}\eg} {L_2\sNC_{\tmax}} } \right) \eg \\
    &\leq -\rho \left( \frac{ 6(1-\rho)\sNC_{\tmin}/(L_2\sNC_{\tmax}) } {\frac{3\Delta_H}{2L_2}  + \sqrt{\frac{9 \Delta_H^2}{4L_2^2}  + \frac{6(1-\rho)\sNC_{\tmin}} {L_2\sNC_{\tmax}}} } \right) \eg^2 \\
    &= - \cNCt \eg^2.
\end{align*}
On the third line we apply $\vnorm{\ppt} \geq \sNC_{\tmin} \vnorm{\bgg}$ and $\vnorm{\bgg} > \eg$, whereas on the fourth line we apply \cref{lemma:technical lower bound} and $\eg < 1$. In the final line we define 
\begin{align*}
    \cNCt \defeq \rho\frac{12(1- \rho) \sNC_{\tmin}/\sNC_{\tmax}} {3\Delta_H + \sqrt{9 \Delta_H^2 + 24(1-\rho)L_2\sNC_{\tmin}/\sNC_{\tmax}}}.
\end{align*}

Next we come to the positive curvature case. Considering \cref{eqn:inexact case positive function upper bound} we see that the line search condition \cref{eqn:inexact line search} is satisfied if
\begin{align*}
    D_2 \alpha^2 + D_1 \alpha + D_0 \leq 0,
\end{align*}
where 
\begin{align*}
    D_2 \defeq \frac{L_2 \vnorm{\ppt}}{6\sigma}, \quad D_1 \defeq \frac{\Delta_H}{2\sigma}, \quad D_0 \defeq -\left(\frac{1}{2} - \rho\right).
\end{align*}
Again this quadratic has a positive and negative root. The positive root is given by 
\begin{align*}
    \alpha^\star_2 &\defeq - \frac{D_1}{2D_2} + \sqrt{\frac{D_1^2}{4D_2^2} - \frac{D_0}{D_2}} \\
    &= -\frac{3 \Delta_H }{2 L_2 \vnorm{\ppt}} + \sqrt{\frac{9\Delta_H^2}{4L_2^2 \vnorm{\ppt}^2} + \frac{6\sigma(1/2 - \rho)}{L_2 \vnorm{\ppt}}}.
\end{align*}
Therefore, the largest step size for which \cref{eqn:inexact line search} holds must satisfy $\alpha \geq \min\{ 1, \alpha^\star_2 \}$. We now isolate the $\LPC$ and $\SPC$ cases. For the $\LPC$ case, by combining the line search criteria \cref{eqn:inexact line search} and $ \alpha \geq \min\{1, \alpha^\star_2\} $ we obtain
\begin{align*}
    f(\xx + \alpha \ppt) - f(\xx) &\leq - \rho \alpha \vnorm{\bgg}\vnorm{\ppt} \\
    &\leq - \rho \min\left\{ \alpha^\star_2\vnorm{\ppt} , \vnorm{\ppt} \right\}\vnorm{\bgg} \\
    &= - \rho \min\left\{ -\frac{3 \Delta_H }{2 L_2} + \sqrt{\frac{9\Delta_H^2}{4L_2^2} + \frac{6\sigma(1/2 - \rho) \vnorm{\ppt} }{L_2} }, \vnorm{\ppt} \right\}\vnorm{\bgg} \\
    &< - \rho \min\left\{  -\frac{3 \Delta_H }{2 L_2} + \sqrt{\frac{9\Delta_H^2}{4L_2^2} + \frac{6\sigma(1/2 - \rho)\sLPC_{\tmin} \eg } {L_2} }, \sLPC_{\tmin}\eg \right\}\eg \\
    &\leq - \rho \min\left\{ \frac{6\sigma(1/2 - \rho)\sLPC_{\tmin}/L_2}{ \frac{3 \Delta_H }{2 L_2} + \sqrt{\frac{9\Delta_H^2}{4L_2^2} + \frac{6\sigma(1/2 - \rho)\sLPC_{\tmin}}{L_2}}}, \sLPC_{\tmin} \right\}\eg^2 \\
    &= -\cLPCt \eg^2.
\end{align*}
On the fourth line we applied $\vnorm{\ppt} \geq \sLPC_{\tmin} \vnorm{\bgg}$ and $\vnorm{\bgg}> \eg$, while in the fifth line we applied \cref{lemma:technical lower bound} and $\eg < 1$. In the final line we define 
\begin{align*}
    \cLPCt \defeq \rho \min \left\{ \frac{12\sigma(1/2 - \rho)\sLPC_{\tmin}} {3\Delta_H + \sqrt{9\Delta_H^2 + 24\sigma(1/2 - \rho)L_2\sLPC_{\tmin}}},  \sLPC_{\tmin} \right\}.
\end{align*}

Finally, for the SPC case the line search condition \cref{eqn:inexact line search} and $\alpha \geq \min\{ 1, \alpha^\star_2 \}$ yields
\begin{align*}
    f(\xx + \alpha \ppt) - f(\xx) &\leq - \rho \alpha \vnorm{\bgg}\vnorm{\ppt} \\
    &\leq - \rho \min\left\{ \alpha^\star_2\vnorm{\ppt} , \vnorm{\ppt} \right\}\vnorm{\bgg} \\
    &= - \rho \min\left\{ - \frac{3 \Delta_H}{2L_2} + \sqrt{ \frac{9\Delta_H^2}{4L_2^2} + \frac{6\sigma(1/2 - \rho)\vnorm{\ppt}}{L_2}  }, \vnorm{\ppt} \right\}\vnorm{\bgg} \\
    &< - \rho \min\left\{ - \frac{3 \Delta_H}{2L_2} + \sqrt{ \frac{9\Delta_H^2}{4L_2^2} + \frac{6\sigma^2(1/2 - \rho) \eg}{L_2 \Lt_1^2}  }, \frac{\sigma}{\Lt_1^2} \eg \right\}\eg \\
    &\leq - \rho \min\left\{ \frac{ 6\sigma^2(1/2 - \rho)/ (L_2 \Lt_1^2)} {\frac{3 \Delta_H} {2L_2} + \sqrt{\frac{9 \Delta_H^2} {4 L_2^2} + \frac{6\sigma^2(1/2 - \rho)} {L_2 \Lt_1^2}}}, \frac{\sigma}{\Lt_1^2} \right\}\eg^2 \\
    &= -\cSPCt \eg^2.
\end{align*}
We apply $\vnorm{\ppt} \geq (\sigma/\Lt_1^2) \vnorm{\bgg}$, which follows from \cref{eqn:inexact step size lower bound}, and $\vnorm{\bgg} > \eg$ in the fourth line, while in the fifth line we apply \cref{lemma:technical lower bound}. In the final line we define
\begin{align*}
    \cSPCt \defeq \rho\min\left\{ \frac{ 12\sigma^2(1/2 - \rho)/\Lt_1^2} { 3 \Delta_H + \sqrt{ 9 \Delta_H^2 + 24\sigma^2(1/2 - \rho)L_2/\Lt_1^2}}, \frac{\sigma}{\Lt_1^2} \right\} .
\end{align*}
The proof now proceeds similarly to the deterministic case. Let 
\begin{align*}
    K \defeq \left\lceil\frac{f(\xx_0) - f^\star}{\min\left\{ \cNCt, \cLPCt, \cSPCt \right\}\eg^2} \right\rceil.
\end{align*}

Suppose that $\vnorm{\bgg_k} \leq \eg$ fails to hold for $k=0,\ldots,K-1$. We divide the iterations $\sK = \{0, \ldots,K-1\}$ into a disjoint union $\sK = \sK_{\text{NC}} \cup \sK_{\text{LPC}} \cup \sK_{\text{SPC}} $ where 
\begin{align*} 
    \sK_{\text{NC}} &= \left\{ k \in \sK \mid \dotprod{\bgg_k, \HHt_k \bgg_k} < 0 \right\} \\
    \sK_{\text{LPC}} &= \left\{ k \in \sK \mid 0 \leq \dotprod{\bgg_k, \HHt_k \bgg_k} \leq \sigma \vnorm{\bgg_k}^2 \right\} \\
    \sK_{\text{SPC}} &= \left\{ k \in \sK \mid \dotprod{\bgg_k, \HHt_k \bgg_k} > \sigma \vnorm{\bgg_k}^2 \right\}.
\end{align*}
Applying a telescoping sum and the lower bounds on function decrease in each case, we have
\begin{align*}
    f(\xx_0) - f(\xx_K) &= \sum_{i=0}^{K-1} f(\xx_k) - f(\xx_{k+1}) \\
    &= \sum_{k\in\sK_{\text{NC}}} f(\xx_k) - f(\xx_{k+1}) + \sum_{k\in\sK_{\text{LPC}}} f(\xx_k) - f(\xx_{k+1}) + \sum_{k\in\sK_{\text{SPC}}} f(\xx_k) - f(\xx_{k+1}) \\
    &> |\sK_{\text{NC}} | \cNCt \eg^2 + |\sK_{\text{LPC}}|\cLPCt \eg^2 + |\sK_{\text{SPC}}| \cSPCt \eg^2 \\
    &\geq (|\sK_{\text{NC}} | + |\sK_{\text{LPC}} | + |\sK_{\text{SPC}}|)\min\left\{ \cNCt, \cLPCt, \cSPCt  \right\}\eg^2 \\
    &= K\min\left\{ \cNCt, \cLPCt, \cSPCt \right\}\eg^2 \\
    &\geq f(\xx_0) - f^{\star},
\end{align*}
which implies $f(\xx_K) < f^\star$, a contradiction.
\end{proof}

\begin{remark}[Remark on \cref{prop:inexact global convergence}]\label{remark:remark on inexact global convergence}

Notably, \cref{prop:inexact global convergence} does not require an explicit bound on the Hessian error, $\Delta_H$. However, as we now demonstrate, the per iteration dependence on $\eg$ can be improved (relative to the worst case) for $\SPC$ steps if we impose a bound on $\Delta_H$. In particular, we stipulate that, for $\theta \in [0,1)$, the Hessian error tolerance satisfies
\begin{align*}
    \Delta_H \leq 2\theta\sigma(1/2 - \rho). \tageq\label{eqn:inexact tolerance better per iteration}
\end{align*}
Consider then a ``large $\ppt$'' setting defined by
\begin{align*}
    \vnorm{\ppt} \geq \frac{6\sigma(1/2 - \rho) - 3 \Delta_H}{L_2} = \frac{6\sigma(1/2 - \rho)}{L_2} - \frac{3\Delta_H}{L_2}.
\end{align*}
Note that the lower bound on $\ppt$ is nontrivial because of the bound in \cref{eqn:inexact tolerance better per iteration}. Rearranging and subbing to the step size acceptance tolerance for the $\SPC$ case, $\alpha^\star_2$, we have 
\begin{align*}
    \alpha_2^\star &= -\frac{3 \Delta_H }{2 L_2 \vnorm{\ppt}} + \sqrt{\frac{9\Delta_H^2}{4L_2^2 \vnorm{\ppt}^2} + \frac{6\sigma(1/2 - \rho)}{L_2 \vnorm{\ppt}}} \\ 
    &\leq -\frac{3 \Delta_H }{2 L_2 \vnorm{\ppt}} + \sqrt{\frac{9\Delta_H^2}{4L_2^2 \vnorm{\ppt}^2} + \frac{\vnorm{\ppt} + \frac{3\Delta_H}{L_2}}{\vnorm{\ppt}}} \\ 
    &= -\frac{3 \Delta_H }{2 L_2 \vnorm{\ppt}} + \sqrt{1 + \frac{9\Delta_H^2}{4L_2^2 \vnorm{\ppt}^2} + \frac{3\Delta_H}{L_2 \vnorm{\ppt}}} \\ 
    &= -\frac{3 \Delta_H }{2 L_2 \vnorm{\ppt}} + \sqrt{\left( 1 + \frac{3\Delta_H}{2L_2 \vnorm{\ppt}} \right)^2} \\ 
    &= 1.
\end{align*}
Therefore, for large $\ppt$ the largest step size that satisfies \cref{eqn:inexact line search}, must also satisfy $\alpha \geq \alpha_2^\star$. Additionally, applying \cref{eqn:inexact tolerance better per iteration} to the lower bound on $\vnorm{\ppt}$ we obtain
\begin{align*}
    \vnorm{\ppt} \geq \frac{6\sigma(1/2 - \rho) - 3 \Delta_H }{L_2} \geq \frac{6\sigma(1 - \theta)(1/2 - \rho)}{L_2}.
\end{align*}
By combining the line search condition \cref{eqn:inexact line search}, $\alpha \geq \alpha_2^\star$, \cref{eqn:inexact second order descent}, the $\SPC$ curvature test and the lower bound on $\vnorm{\ppt}$ we have
\begin{align*}
    f(\xx + \alpha \ppt) - f(\xx) &\leq  \alpha \rho \dotprod{\ppt, \bgg}  \\ 
    &\leq - \alpha^\star_2 \rho \dotprod{\ppt, \HHt \ppt} \\
    &\leq - \rho \left( -\frac{3 \Delta_H }{2 L_2 \vnorm{\ppt}} + \sqrt{\frac{9\Delta_H^2}{4L_2^2 \vnorm{\ppt}^2} + \frac{6\sigma(1/2 - \rho)}{L_2 \vnorm{\ppt}}} \right) \sigma \vnorm{\ppt}^2 \\ 
    &\leq - \sigma \rho \left( -\frac{3 \Delta_H }{2 L_2}  + \sqrt{\frac{9\Delta_H^2}{4L_2^2} + \frac{6\sigma(1/2 - \rho)\vnorm{\ppt}}{L_2}} \right) \vnorm{\ppt} \\ 
    &\leq - \sigma \rho \left( -\frac{3 \Delta_H }{2 L_2}  + \sqrt{\frac{9\Delta_H^2}{4L_2^2} + \frac{36\sigma^2(1/2 - \rho)^2(1-\theta)}{L_2^2}} \right)\left(\frac{6\sigma(1 - \theta)(1/2 - \rho)}{L_2} \right) \\
    &= - \sigma \rho \left( \frac{ 72\sigma^2(1/2 - \rho)^2(1-\theta)/L_2} { 3 \Delta_H + \sqrt{ 9\Delta_H^2 + 144\sigma^2(1/2 - \rho)^2(1-\theta)}} \right)\left(\frac{6\sigma(1 - \theta)(1/2 - \rho)}{L_2} \right) \\
    &= - \frac{\sigma \rho}{L_2^2} \left( \frac{ 432\sigma^3(1/2 - \rho)^3(1-\theta)^2} { 3 \Delta_H + \sqrt{ 9\Delta_H^2 + 144\sigma^2(1/2 - \rho)^2(1-\theta)}} \right). \tageq\label{eqn:inexact SPC larger p function decrease}
\end{align*}
This bound suggests that, analogously to the ``large $\pp$'' case for the exact Hessian, a better dependence on $\eg$ is obtained when $\ppt$ is large and $\Delta_H$ is controlled. In fact, when $\theta=0$ we have $\Delta_H=0$ and \cref{eqn:inexact SPC larger p function decrease} matches the ``large $\pp$'' exact case in \cref{lemma:per iteration decrease}. In conclusion, despite the worst case headline rate, in some cases a better dependence on $\eg$ may be obtained for certain steps if $\Delta_H$ is controlled. 
\end{remark}

\subsection{Sub-sampled Hessian Error Bound} \label{sec:subsampled hessian}

\begin{lemma}[Hessian Sub-sampling] \label{lemma:hessian subsampling}

Consider the finite-sum objective 
\begin{align*}
    f(\xx) = \frac{1}{n} \sum_{i=1}^n f_i(\xx).
\end{align*}
Suppose that there exists $0 \leq \Lmax_1 < \infty$ such that for all $\xx \in \real^d$ we have $\max_{i=1,\ldots,n} \vnorm{\HH_i \bgg} \leq \Lmax_1 \vnorm{\bgg}$. Let $\delta \in (0,1)$ and $\HHt$ be as in \cref{eqn:Hessian subsampling}. Then for any $\Delta_H \geq 0$ and $\xx \in \real^d$ we have 
\begin{align*}
    \Pr\left(\vnorm{(\HHt - \HH)\bgg} \leq \Delta_H \vnorm{\bgg} \right) \geq 1-\delta,
\end{align*}
if 
\begin{align*}
    |\sI_H| \geq \frac{ (\Lmax_1)^2 \left(1 + \sqrt{8\log(1/\delta)}\right)^2} {\Delta_H^2 }.
\end{align*}
    
\end{lemma}

\begin{proof}
    The result proceeds similarly to \citet[Lemma 3]{roosta2019subsampledNewton}. First we write $\HH\bgg = \AA \BB$ where
    \begin{align*}
        \AA = [\HH_1\bgg, \ldots, \HH_n \bgg] \in \real^{d \times n}, \quad \BB = [1/n, \ldots, 1/n]^\transpose \in \real^n.
    \end{align*}
    We would like to relate this matrix product to the sub-sampled alternative. To do this, take the minibatch, $\sI_H$, and form $\tilde{\AA} \in \real^{d \times |\sI_H| }$, using the columns of $\AA$ corresponding to $\sI_H$, rescaled by $\sqrt{n/|\sI_H|}$. Similarly, form $\tilde{\BB} \in \real^{|\sI_H|}$, using the rows of $\BB$ corresponding to $\sI_H$, rescaled by $\sqrt{n/|\sI_H|}$. Clearly,
    \begin{align*}
        \tilde{\AA} \tilde{\BB}= \frac{n}{|\sI_H|} \sum_{i \in \sI_H} \frac{1}{n} \HH_i \bgg = \HHt \bgg.
    \end{align*}
    Applying \citet[Lemma 11]{drineas2006FastMonteCarlo} allows us to relate $\AA\BB$ to $\tilde{\AA}\tilde{\BB}$. In particular, with probability at least $1- \delta$, we have
    \begin{align*}
        \vnorm{\HH\bgg - \HHt\bgg} = \vnorm{\AA \BB - \tilde{\AA}\tilde{\BB}} &\leq \sqrt{\frac{n}{|\sI_H|} \sum_{i=1}^n \vnorm{\AA_i}^2 |\BB_i|^2} + \frac{n}{\sqrt{|\sI_H|}} \sqrt{8 \log(1/\delta)} \max_{i=1, \ldots, n}\vnorm{\AA_i}|\BB_i| \\
        &\leq \sqrt{\frac{1}{n|\sI_H|} \sum_{i=1}^n \vnorm{\HH_i \bgg}^2  } + \frac{n}{\sqrt{|\sI_H|}} \sqrt{8 \log(1/\delta)}\max_{i=1, \ldots, n} \vnorm{\HH_i\bgg}/n \\
        &\leq \sqrt{\frac{1}{|\sI_H|} (\Lmax_1)^2\vnorm{\bgg}^2 } + \sqrt{\frac{8 \log(1/\delta)}{|\sI_H|}} \Lmax_1 \vnorm{\bgg} \\
        &\leq \frac{1}{\sqrt{|\sI_H|}}\left( 1 + \sqrt{8 \log(1/\delta)} \right) \Lmax_1\vnorm{\bgg}.
    \end{align*}
    Finally, we see that 
    \begin{align*}
        |\sI_H| \geq \frac{ (\Lmax_1)^2 \left(1 + \sqrt{8\log(1/\delta)}\right)^2} {\Delta_H^2 } \implies \frac{1}{\sqrt{|\sI_H|}}\left( 1 + \sqrt{8 \log(1/\delta)} \right)\Lmax_1 \leq \Delta_H.
    \end{align*}
\end{proof}

As a result of \cref{lemma:hessian subsampling} it is clear that \cref{ass:Hessian inexactness bound} is satisfied with high probability as, by the Cauchy-Schwarz inequality, we have 
\begin{align*}
    \abs{\dotprod{\bgg, (\HH - \HHt)\bgg}} \leq \vnorm{\bgg} \vnorm{(\HH - \HHt)\bgg}.
\end{align*}

\section{Miscellaneous Scaling Properties }

\paragraph{MR Scaling Tighter Lower Bound} Recall from \cref{lemma:scaling lower bounds} that if $\vnorm{\HH\bgg} \leq L_1 \vnorm{\bgg}$ for some $L_1$ (see \cref{ass:Hessian gradient smoothness condition}) then then CG, MR and GM scalings satisfy a lower bound in terms of the gradient norm. Notably, the MR scaling lower bound is weaker by a factor of $\sigma/L_1$, in comparison with the CG and GM. In the following lemma we show that this lower bound can be tightened to match the CG and GM case if the Hessian spectrum is positive and bounded.

\begin{lemma}\label{lemma:tighter MR lower bound}
    If the Hessian satisfies $0 \preceq \HH \preceq M \eye$ and $\vnorm{\HH \bgg} >0$ then 
    \begin{align*}
        \sMR \geq 1/M.
    \end{align*}
\end{lemma}
\begin{proof}
    We utilise the fact that $\sMR$ is a Rayleigh quotient of $\HH^\dagger$. Suppose that $\HH$ has $\psi_+$ nonzero eigenvalues ($\vnorm{\HH \bgg} > 0$ implies $\psi_+ > 0$) given by $0 < \lambda_1 < \lambda_2 \leq \ldots \leq \lambda_{\psi_+} \leq M$. The corresponding nonzero eigenvalues of $\HH^\dagger$ are given by $1/\lambda_{\psi_+} \leq \ldots \leq 1/\lambda_1$. Since $\HH\bgg \in \range{(\HH)}$ and is therefore orthogonal to $\Null{(\HH)}$ we have
    \begin{align*}
        \sMR = \frac{\langle \bgg, \HH \bgg \rangle}{\vnorm{\HH \bgg}^2}  = \frac{\langle \HH \bgg, \HH^\dagger \HH \bgg \rangle}{\| \HH \bgg \|^2} \geq \frac{1}{\lambda_{\psi^+}} \frac{\vnorm{\HH \bgg}^2}{\vnorm{\HH \bgg}^2} \geq \frac{1}{M}.
    \end{align*}
\end{proof}

The requirement that $\vnorm{\HH \bgg} > 0$ is not stringent since the positive curvature test $\dotprod{\bgg, \HH \bgg} > 0$ implies $\HH\bgg \neq 0$.

\section{Additional Numerical Results} \label{apx:additional-numerical-results}

The code used to run our experiments can be found \href{https://anonymous.4open.science/r/FirstishOrderMethods-B25E/README.md}{here}.

\subsection{Oracle Calls as Complexity Measure} \label{apx:numerical-results-oracle-calls}

Following the typical convection in the optimization literature, in all our experiments, we plot the objective value against the total number of oracle calls for function, gradient, and Hessian-vector product evaluations, which allows for a fair comparison between methods with a differing per-iteration computational costs.
We adopt this approach because the measurement of ``wall-clock'' time can be heavily dependent on specific implementation details and computational platform. 
In contrast, counting the number of equivalent function evaluations, as an implementation and system independent unit of complexity is more appropriate and fair. 
More specifically, upon evaluating the function, computing its gradient is equivalent to one additional function evaluation, and computing a Hessian-vector product requires two additional function evaluations compared to a gradient evaluation \cite{pearlmutterFastExactMultiplication1994}. 
For example, in neural networks, for given data at the input layer, evaluation of network's output, i.e., function evaluation, involves one forward propagation. The corresponding gradient is computed by performing one additional backward propagation. After computing the gradient, an additional forward followed by a backward propagation give the corresponding Hessian-vector product \cite{goodfellow2016deep,blondel2024elementsdifferentiableprogramming}. 

\subsection{Competitor Algorithms} \label{apx:competitor algorithms}

In this section we give an outline of the algorithms we are comparing against.

\paragraph{Line Search}
For line search gradient descent, we apply a backtracking procedure (\cref{alg:back tracking line search}) based on the Amijo condition \cref{eqn:armijo condition} with $\pp= -\bgg$. 
Inspired by \citep[Algorithm 2]{vaswaniPainlessStochasticGradient2021}, we consider three ``reset schemes'' for initializing back tracking line search at each iteration, i.e., setting $\alpha^0$ in \cref{alg:back tracking line search}. 
In the first procedure, which we term \textit{no reset}, we set $\alpha^0 = \alpha_{k-1}$, where $\alpha_{k-1}$ is the final step size from the previous iteration. Because the final step size is recycled at each iteration, no reset initialization will require significantly less backtracking. However, by the same token, once the step size shrinks it cannot be increased in future iterations, which limits step sizes and adaptivity.
In contrast, for \textit{full reset} initialization we fix some $\alpha^\text{init}$ and set $\alpha^0 = \alpha^\text{init}$ at each iteration. This setting will lead to additional backtracks as the line search must begin over at each iteration, however it also allows for a greater degree of adaptivity when compared with no resetting. 
Finally, in \textit{limited reset} we set $\alpha^0 = \alpha_{k-1}/\gamma$ where $\gamma \in (0, 1)$ is a growth parameter and $\alpha_{k-1}$ is the final step size from the previous iteration. This setting is an intermediate between full and no resets, as it allows for some growth in the step length from iteration to iteration. For our experiments we set $\gamma = \theta$, $\alpha^{\text{init}}=1$ and $\alpha_{-1}=1$. 
We report the resetting technique utilized in our final results in the additional results for each of our examples (\cref{apx:numerical-results-logistic,apx:numerical-results-mlp,apx:numerical-results-resnet}).

\paragraph{PoNo Line Search}
In addition to the above line search ``reset'' schemes, we also consider a deterministic version of the Polyak non-monotone line search from \citep{galliDontBeMonotone2023}. Specifically, we simply replace the minibatch function and gradient with the full function and gradient values. We set the hyperparameters to the ``standard'' values outlined in \cite[Appendix C]{galliDontBeMonotone2023} and use $f^*=0$ as our lower bound.

\paragraph{Fixed Step Size and Accelerated Methods}
Fixed step size GD consists of applying the update 
\begin{align*}
    \xx_{k+1} = \xx_k - \alpha \bgg_k, \tageq\label{eqn:fixed step size GD}
\end{align*}
at each iteration, where $\alpha$ is a step size parameter which must be set. The momentum methods utilize updates of the form 
\begin{align*}
    \begin{cases}
        \vv_{k+1} = \beta \vv_k -\alpha \bgg_{k} \\
        \xx_{k+1} = \begin{cases}
        \xx_k + \vv_{k+1} &\quad \text{if Heavy ball.} \\
        \xx_k -\alpha \bgg_k + \beta \vv_{k+1} &\quad \text{if Nesterov.}
    \end{cases}
    \end{cases} \tageq\label{eqn:momentum methods}
\end{align*}
where $\alpha$ is a step size parameter and $\beta$ is a momentum parameter. If no theoretical value is available we set $\beta=0.9$.
For Adam \citep{kingma2014adam} we use standard values of $\beta_1 = 0.9$, $\beta_2=0.999$ and $\epsilon =10^{-8}$. Each of the Adam, fixed step size and momentum methods require a step size parameter to be set. We do this by using either theoretical values (where available) or via a tuning procedure; see \cref{apx:numerical-results-logistic,apx:numerical-results-mlp,apx:numerical-results-resnet} for specifics. 

\subsection{Multi-class Logistic Regression} \label{apx:numerical-results-logistic}

Consider a set of data items $\sD = \{\aa_i, b_i \}_{i=1}^n \subset \real^d \times \{1, \ldots C\}$. Denote the weights of each class as $\xx_1, \ldots, \xx_{C}$ and define $\xx = [ \xx_1, \ldots, \xx_{C-1}]$. We are free to take $\xx_C = \zero$ as class $C$ is identifiable from the weights of the other classes. The objective, $f$, is given by 
\begin{align*}
    f(\xx) = \frac{1}{n} \sum_{i=1}^n  \sum_{c=1}^{C-1} -\one(b_i = c) \log{(\text{softmax}(\xx_c, \aa_i))} + \frac{\lambda}{2} \vnorm{\xx}^2, \tageq\label{eqn:multinomial regression objective}
\end{align*}
where $\one(\cdot)$ is the indicator function and 
\begin{align*}
    \text{softmax}(\xx_c, \aa_i) = \frac{\exp{(\langle \xx_c, \aa_i \rangle)}}{\sum_{c=1}^C \exp{(\langle \xx_c, \aa_i \rangle)}}.
\end{align*}
In our experiments a bias parameter is included in the weights for each class. We set the regularization parameter to $\lambda=10^{-3}$ and initialize by setting $\xx_0 = 0$. We run all methods until a maximum of $10^5$ oracle calls or until $\vnorm{\bgg_k} \leq 10^{-4}$.

Note that the spectrum of the Hessian of \cref{eqn:multinomial regression objective} can be bounded as 
\begin{align*}
    \lambda \eye \preceq \grad^2 f(\xx) \preceq \left(\frac{C-1}{4n} \vnorm{\AA}^2 + \lambda \right) \eye, \tageq\label{eqn:multinomial regression Hessian bound}
\end{align*}
which allows us to estimate the strong convexity parameter and Lipschitz constant of \cref{eqn:multinomial regression objective} as 
\begin{align*}
    \mu_{\text{approx}} \defeq \lambda, \quad L_{\text{approx}} \defeq \frac{C-1}{4n} \vnorm{\AA}^2 + \lambda.
\end{align*}

\paragraph{Parameter Settings}

For the scaled gradient methods, we set $\sigma=0$ as \cref{eqn:multinomial regression objective} is $\mu$-strongly convex (i.e., curvature along the gradient is already lower bounded). For fixed step size gradient descent \cref{eqn:fixed step size GD} we use a step size of $\alpha=1/L_\text{approx}$. For both the scaled GD and vanilla GD line search we utilized the the standard settings $\theta=0.5$ and $\rho=10^{-4}$. The limited reset initialization scheme achieved the best performance for vanilla GD with line search. 

For $\mu$-strongly convex, $\Lg$-smooth problems the parameters for the Heavy ball momentum \citep{polyakIntroductionOptimization1987} can be set as
\begin{align*}
    \alpha = \frac{4}{(\sqrt{\Lg} + \sqrt{\mu})^2}, \quad \beta = \left(\frac{\sqrt{\Lg} - \sqrt{\mu}}{\sqrt{\Lg} + \sqrt{\mu}} \right)^2.
\end{align*}
While for Nesterov acceleration \citep{nesterovIntroductoryLecturesConvex2004} the step size and momentum parameter can be set as
\begin{align*}
    \alpha = \frac{1}{\Lg}, \quad \beta = \frac{\sqrt{\Lg/\mu}- 1}{\sqrt{\Lg/\mu} + 1}.
\end{align*}
For the momentum methods \cref{eqn:momentum methods} in our experiments use these settings with $\mu_{\text{approx}}$ and $L_{\text{approx}}$ taking the place of $\mu$ and $\Lg$, respectively. The step size for Adam was tuned over the grid $\{10^{-1}, 10^{-2}, 10^{-3}, 10^{-4}, 10^{-5}\}$ . 

\paragraph{Additional Results}

We now present some additional numerical results for the multi-class logistic regression problem. In \cref{fig:logistic_regression_cifar10_ponos} we give the results from \cref{fig:logistic_regression_cifar10} with PoNo line search included. We see that, despite enforcing monotonicity, at each iteration our method outperforms the best iterates of the highly non-monotone PoNo method. Meanwhile, the best iterates of PoNo outperform vanilla GD with line search.
\begin{figure}[ht]
    \centering
    \includegraphics[width=\linewidth]{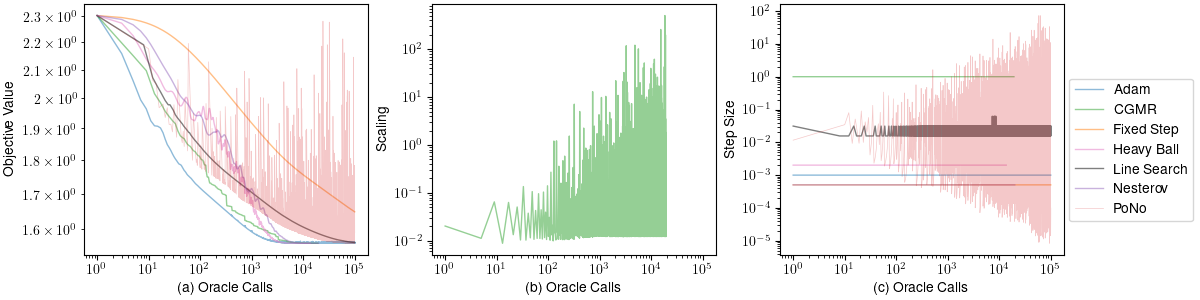}
    \caption{ Multi-class logistic regression on CIFAR10 (\cref{fig:logistic_regression_cifar10} with PoNo line search included). (a) Objective value. (b) Scaling utilized by the CGMR method. (c) Step size. }
    \label{fig:logistic_regression_cifar10_ponos}
\end{figure}

In \cref{fig:logistic_regression_cifar10_scalings} we evaluate the performance of each of the scaled gradient methods. We see that CGMR and MRCG (which differ by whether they start on the CG or MR scaling, respectively) perform quite similarly and handily outperform the other methods. 

\begin{figure}[!ht]
    \centering
    \includegraphics[width=0.4\linewidth]{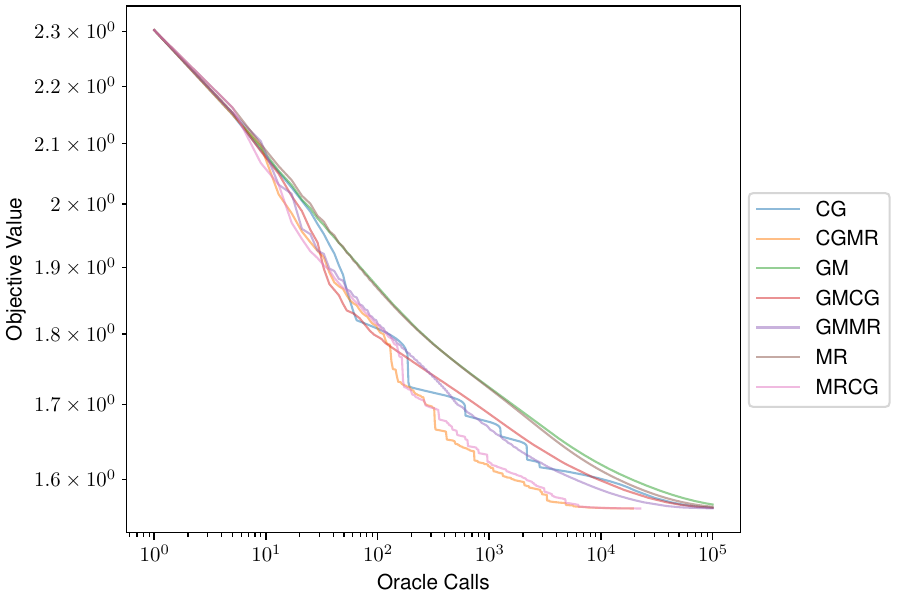}
    \caption{Comparison between scaling method performance for multiclass logistic regression on CIFAR10.}
    \label{fig:logistic_regression_cifar10_scalings}
\end{figure}

In \cref{fig:logistic_regression_cifar10_CG} and \cref{fig:logistic_regression_cifar10_MR} we see a breakdown of the performance of the ``CG'' and ``MR'' scalings, respectively. Firstly, in panel (c) for each method, we see that the unit step size is accepted by the line search at each iteration, similarly to MRCG.
For the MR scaling (\cref{fig:logistic_regression_cifar10_MR}) we plot the gradient norm in the (a) panel, from which we see that the MR method produces a monotonic decrease in the gradient norm at each iteration with a unit step size; indicating the results from \cref{thm:local convergence MR second-order sufficient} can hold over a wide portion of the optimization landscape. This is despite the line search targeting the function value, rather than the gradient norm. 
In panel (b) of \cref{fig:logistic_regression_cifar10_CG} and \cref{fig:logistic_regression_cifar10_MR} we see that the oscillating behavior is observed in the scaling values. 
However, comparing with \cref{fig:logistic_regression_cifar10}, it is clear that this oscillation is over a smaller range of values (particularly for MR scaling) than the alternating MRCG scaling. This indicates that the alternation between MR and CG scalings is key to obtaining large scaling values and, in light of \cref{thm:local convergence}, corresponding rapid convergence.

\begin{figure}[!ht]
    \centering
    \includegraphics[width=0.8\linewidth]{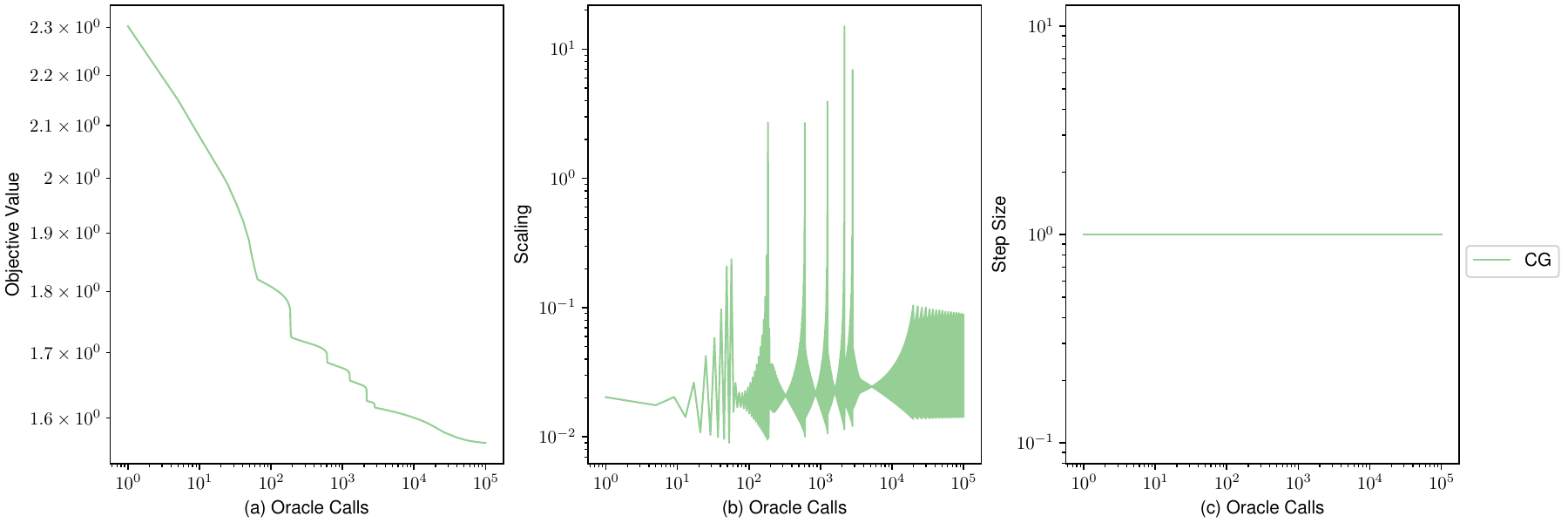}
    \caption{Performance of the CG scaling on the multi-class logistic regression problem on CIFAR10. (a) The objective function. (b) The scaling selected by the CG method. (c) The step size selected by line search.}
    \label{fig:logistic_regression_cifar10_CG}
\end{figure}

\begin{figure}[!ht]
    \centering
    \includegraphics[width=0.8\linewidth]{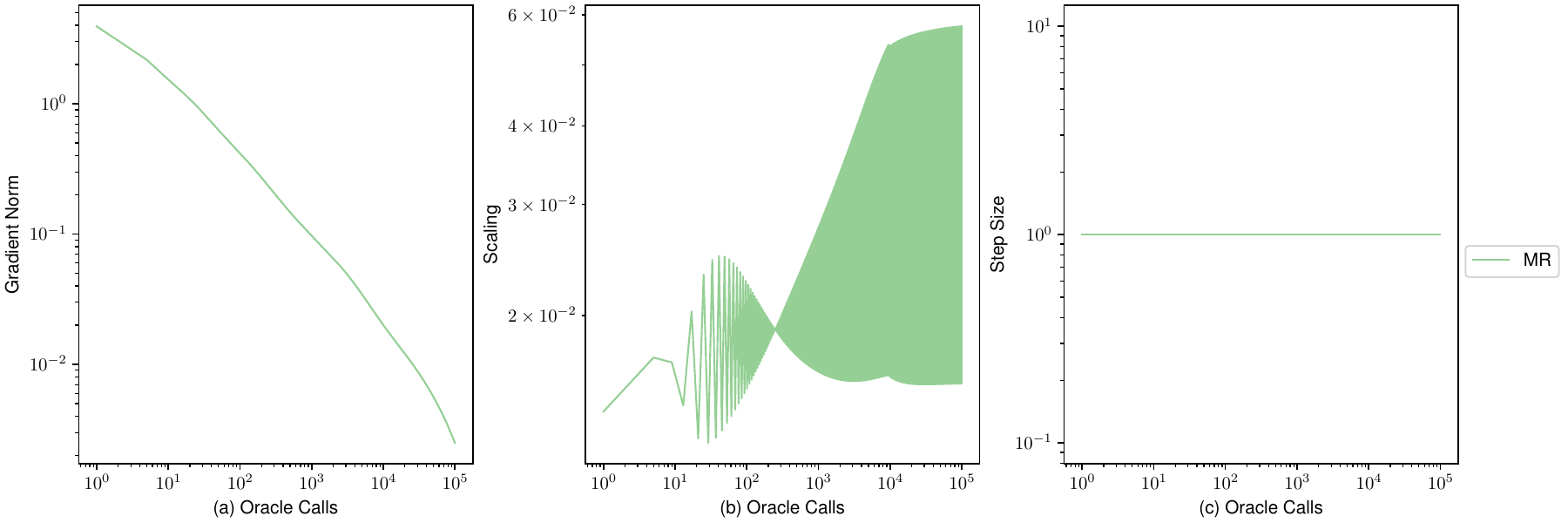}
    \caption{Performance of the MR scaling on the multi-class logistic regression problem on CIFAR10. (a) The gradient norm. (b) The scaling selected by the MR method. (c) The step size selected by line search.}
    \label{fig:logistic_regression_cifar10_MR}
\end{figure}

Finally, in \cref{fig:logistic_regression_cifar10_time} we plot the objective value against wall-clock time we see that results largely conform with oracle calls. 

\begin{figure} [!ht]
    \centering
    \includegraphics[width=0.4\linewidth]{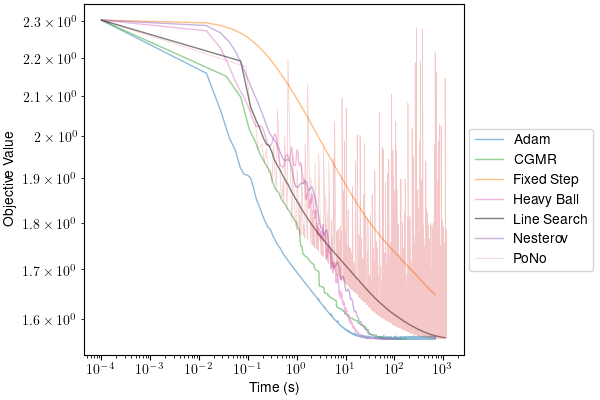}
    \caption{Wall-clock time for multi-class logistic regression on the CIFAR10 dataset. $10^{-4}$ has been added to all times for plotting purposes.}
    \label{fig:logistic_regression_cifar10_time}
\end{figure}

\subsection{MLP} \label{apx:numerical-results-mlp}

Let $\hh(\xx; \cdot)$ denote a two layer MLP with 100 hidden units per layer with GeLU \citep{hendrycksGaussianErrorLinear2023} activations and $C$ output layers, with parameters $\xx$. The objective for our experiment is the function
\begin{align*}
    f(\xx) = \frac{1}{n} \sum_{i=1}\text{CrossEntropy}(\hh(\xx; \aa_i), b_i) + \frac{\lambda}{2} \vnorm{\xx}^2, \tageq\label{eqn:mlp_problem} 
\end{align*}
where $\text{CrossEntropy}(\cdot, \cdot)$ is the cross-entropy loss between the predictions of the MLP and the true labels and $\lambda$ is a $\ell_2$ regularization (or weight decay) parameter. In our experiments the regularization parameter was set to $\lambda=10^{-3}$ and the parameters were initialised with the PyTorch default, that is, the weights for each layer are drawn from $U(-\sqrt{k}, \sqrt{k})$ where $k = 1/(\#\text{num inputs to the layer})$. We run all methods until a maximum of $10^5$ oracle calls or until $\vnorm{\bgg_k} \leq 10^{-4}$.

\paragraph{Parameter Settings}

For the scaled methods we set $\sigma=10^{-6}$ and utilize fixed scalings for the $\LPC$ and $\NC$ case given by $\sLPC = \sNC = 1$. For the scaled and vanilla GD line search methods we utilize the standard values $\theta=0.5$ and $\rho=10^{-4}$. For vanilla gradient descent with line search, limited reset initialization achieved the best performance. 

Since \cref{eqn:mlp_problem} is nonconvex and there is no closed expression for the gradient Lipschitz constant, we manually tuned the step size parameters for fixed step, momentum methods and Adam. The tuning grid and resulting learning rate parameters are summarized in \cref{tab:mlp_fashionMNIST_learning_rates}.

\begin{table}[!ht]
    \centering
    \begin{tabular}{|c|c|c|}
        \hline
        Method & Grid & Learning Rate \\ \hline
        Fixed & $\{ 10^{0}, 10^{-1}, 10^{-2}, 10^{-3} \}$ & $10^{-2}$  \\
        Adam &  $\{10^{-2}, 10^{-3}, 10^{-4}, 10^{-5} \}$ & $10^{-5}$ \\
        HB & $\{10^{-1}, 10^{-2}, 10^{-3} \}$ & $10^{-3}$ \\
        NES & $\{10^{-1}, 10^{-2}, 10^{-3} \}$ & $10^{-3}$ \\
        \hline
    \end{tabular}
    \caption{Tuned learning rates for MLP model on FashionMNIST}
    \label{tab:mlp_fashionMNIST_learning_rates}
\end{table}

\paragraph{Additional Results} We report some additional results for the MLP model on FashionMNIST dataset. In \cref{fig:mlp_fashionMNIST_pono} we plot the main body results (\cref{fig:mlp_fashionMNIST}) with PoNo line search included. Again we see that our method outperforms the best iterates of the non-monotone PoNo method. Meanwhile, the best iterates of the PoNo method significantly outperform vanilla line search.
\begin{figure}[!ht]
    \centering
    \includegraphics[width=\linewidth]{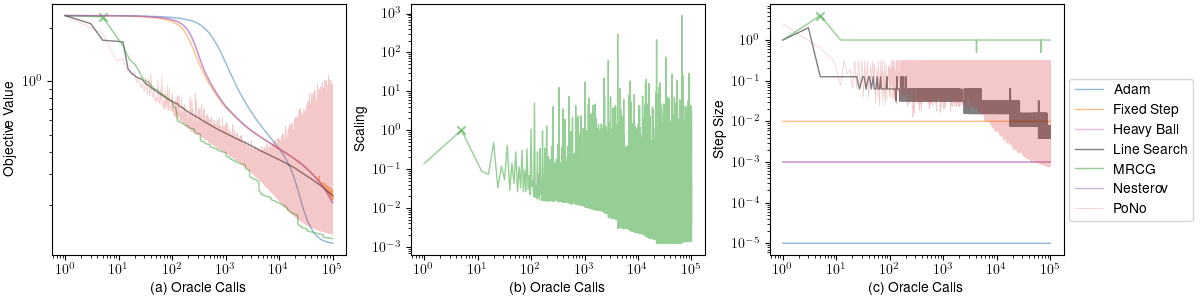}
    \caption{MLP on the FashionMNIST (\cref{fig:mlp_fashionMNIST} with PoNo line search included). (a) Objective value (b) Scaling utilized by the MRCG method (c) Step size. Crosses indicate iterations where negative curvature is detected.}
    \label{fig:mlp_fashionMNIST_pono}
\end{figure}

In \cref{fig:mlp_fashionMNIST_scalings} we compare the performance of each of the scalings on the problem. Again, we see that the alternating scalings ``CGMR'' and ``MRCG'' perform the best.

\begin{figure}[!ht]
    \centering
    \includegraphics[width=0.4\linewidth]{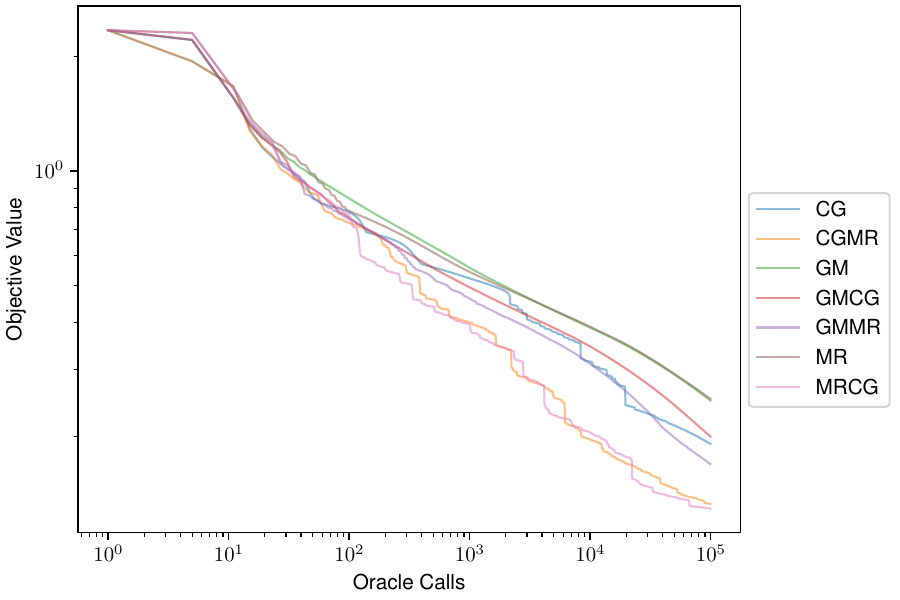}
    \caption{Comparison between scaling methods for MLP on FashionMNIST dataset.}
    \label{fig:mlp_fashionMNIST_scalings}
\end{figure}

In \cref{fig:mlp_fashionMNIST_cg} and \cref{fig:mlp_fashionMNIST_mr} we consider the performance of the CG and MR methods, respectively. For the CG and MR scaling we see that the unit step length is accepted at each iteration except for the iteration where $\NC$ is detected, consistent with \cref{thm:unit step size acceptance}. When $\NC$ is detected forward tracking line search kicks in and a larger step size is selected. Similarly, to the logistic regression case, in panel (a) of \cref{fig:mlp_fashionMNIST_mr} we see that the gradient norm is monotonic, except for iteration where $\NC$ is detected, reinforcing the results in \cref{thm:local convergence MR second-order sufficient}. No $\LPC$ directions are detected for either the MR or CG scaling.

\begin{figure}[!ht]
    \centering
    \includegraphics[width=0.8\linewidth]{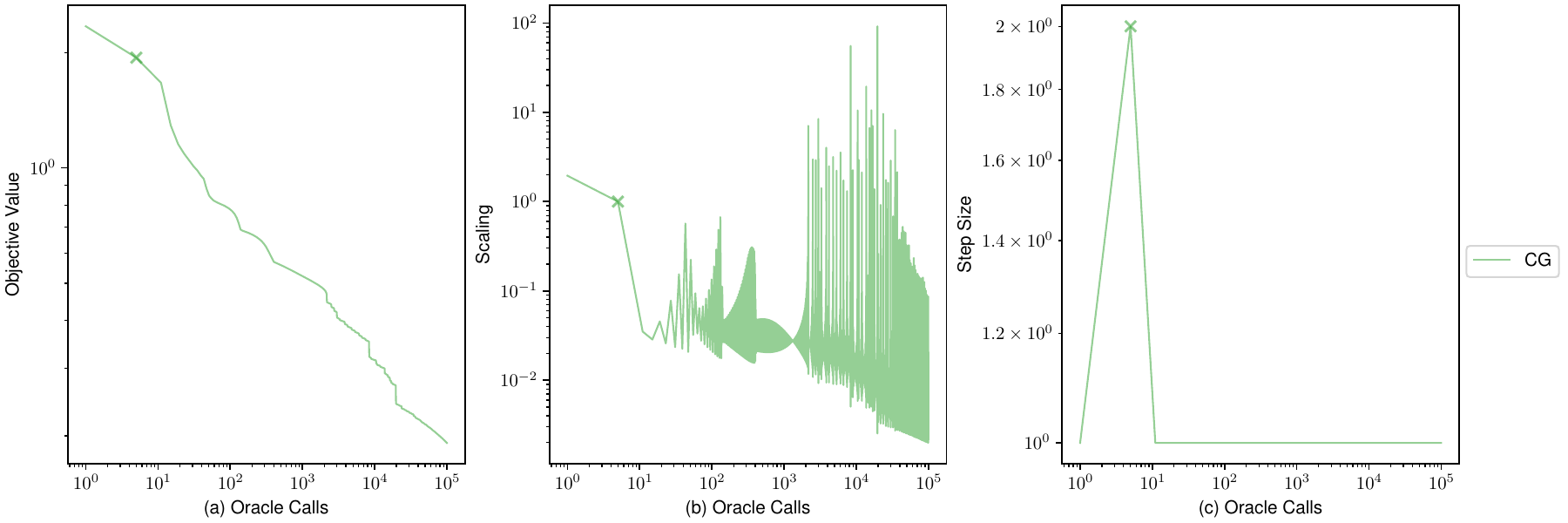}
    \caption{Performance of the CG scaling for the MLP on the FashionMNIST dataset. (a) The objective function. (b) The scaling selected by the CG method. (c) The step size selected by line search. Crosses indicate iterations where negative curvature is detected.}
    \label{fig:mlp_fashionMNIST_cg}
\end{figure}

\begin{figure}[!ht]
    \centering
    \includegraphics[width=0.8\linewidth]{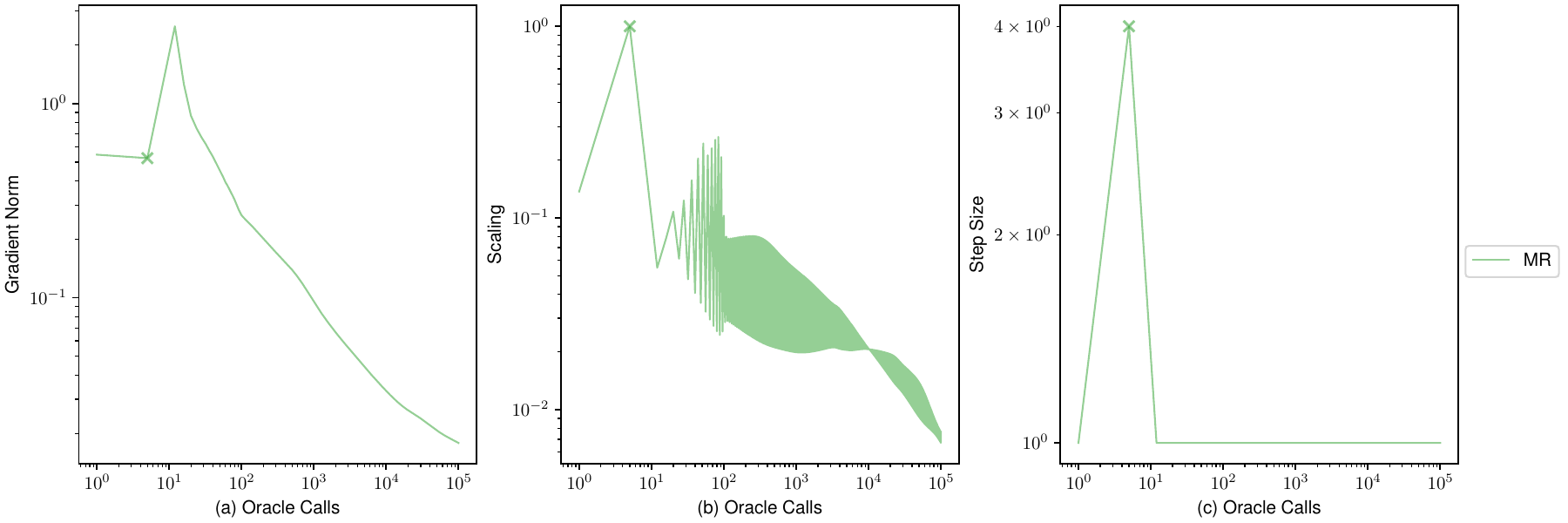}
    \caption{Performance of the MR scaling for the MLP on the FashionMNIST dataset. (a) The gradient norm. (b) The scaling selected by the MR method. (c) The step size selected by line search. Crosses indicate iterations where negative curvature is detected.}
    \label{fig:mlp_fashionMNIST_mr}
\end{figure}

In \cref{fig:mlp_fashionMNIST_time} we report our results from \cref{fig:mlp_fashionMNIST} in terms of wall-clock time. We see that the results are largely unchanged.

\begin{figure} [!ht]
    \centering
    \includegraphics[width=0.4\linewidth]{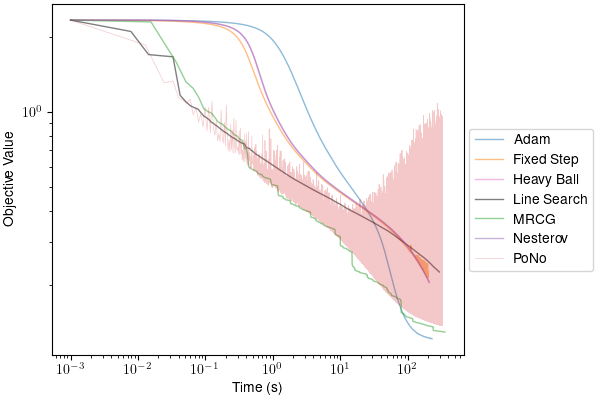}
    \caption{Wall-clock time results for MLP on the FashionMNIST dataset \citep{xiao2017fashionmnist}. $10^{-3}$ has been added to all times for plotting purposes.}
    \label{fig:mlp_fashionMNIST_time}
\end{figure}

\subsection{Resnet} \label{apx:numerical-results-resnet}

Let $\hh(\xx, \cdot)$ denote a depth 18 ResNet \citep{Kaiming2016ResNet} with parameters $\xx$. The objective for our experiment is the function
\begin{align*}
    f(\xx) = \frac{1}{n} \sum_{i=1}\text{CrossEntropy}(\hh(\xx; \aa_i), b_i) + \frac{\lambda}{2} \vnorm{\xx}^2, \tageq\label{eqn:resnet_problem} 
\end{align*}
where $\text{CrossEntropy}(\cdot, \cdot)$ is the cross-entropy loss between the predictions of the network and the true labels and $\lambda$ is an $\ell_2$ regularization parameter. We specifically consider off-the-shelf ResNet18 implementation from PyTorch\footnote{Available \href{https://pytorch.org/vision/main/models/resnet.html}{here}.} with two modifications. Firstly, we replace all $\text{ReLU}$ activations with $\text{GeLU}$ \citep{hendrycksGaussianErrorLinear2023} to ensure twice differentiability. The second modification is replacing the $\text{BatchNorm}$ \citep{Ioffe2015BatchNorm} layers with $\text{LayerNorm}$ \citep{ba2016layernormalization}. This modification has previously been considered in \citep{wu2018groupnormalization}. We utilize the ``160px'' version of the Imagenette dataset \citep{Howard2019Imagenette}, which is available from PyTorch \citep{pytorch}. The Imagenette is a 10 class, subset of the full Imagenet dataset \citep{deng2009imagenet}. We process the data by normalizing per channel followed by resizing the images to $32 \times 32$.

In our experiments we set $\lambda = 10^{-5}$ and use the default initialization for the PyTorch ResNet model. All experiments are run until a maximum of $10^4$ oracle calls or $\vnorm{\bgg_k} \leq 10^{-4}$.

\paragraph{Parameter Settings}

For our scaled gradient method we set $\sigma=10^{-6}$ and $\sNC = \sLPC = 1$. Our scaled and vanilla GD line searches utilized $\theta=0.5$, and $\rho=0$. This choice of $\rho$ was made for numerical stability purposes, as, in contrast to the previous two experiments, this experiment utilized single precision due to memory constraints. Limited reset initialization yielded the best results for initializing the vanilla GD line search. We set the learning rates for all other methods by tuning over a grid. We give the search grid as well as the selected learning rate in \cref{tab:resnet_Imagenette_learning_rates}.

\begin{table}[!ht]
    \centering
    \begin{tabular}{|c|c|c|}
        \hline
        Method & Grid & Learning Rate \\ \hline
        Fixed & $\{10^0, 10^{-1}, 10^{-2}, 10^{-3} \}$ & $10^{-1}$ \\
        Adam & $\{ 10^{-2}, 10^{-3}, 10^{-4}, 10^{-5} \}$ & $10^{-3}$ \\
        HB &  $\{ 10^{-1}, 10^{-2}, 10^{-3}\}$ & $10^{-2}$ \\
        NES &  $\{ 10^{-1}, 10^{-2}, 10^{-3}\}$ &  $10^{-2}$ \\
        \hline
    \end{tabular}
    \caption{Learning rate tuning for ResNet18 model on Imagenette.}
    \label{tab:resnet_Imagenette_learning_rates}
\end{table}

\paragraph{Additional Results}

We now collect the additional results from our ResNet18 experiments. In \cref{fig:resnet_Imagenette_ponos} we give the results from \cref{fig:resnet_Imagenette} with PoNo included. Again, our method outperforms the best iterates of the PoNo line search method.
\begin{figure}[!ht]
    \centering
    \includegraphics[width=1\linewidth]{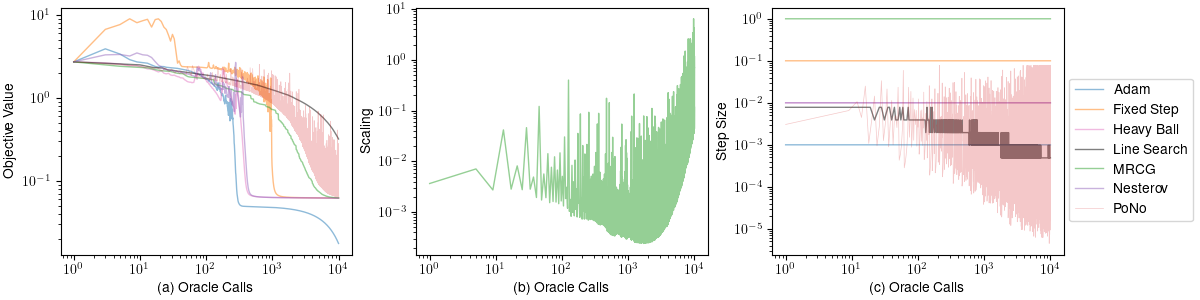}
    \caption{ResNet18 on Imagenette (\cref{fig:resnet_Imagenette} with PoNo line search included.) (a) Objective value. (b) Scaling utilized by the MRCG method. (c) Step size. Crosses indicate iterations where negative curvature is detected.}
    \label{fig:resnet_Imagenette_ponos}
\end{figure}
In \cref{fig:resnet_Imagenette_scaled} we compare between scalings and see that, similar to the previous two experiments, alternating MRCG scaling performs significantly better than the others.

\begin{figure}[!ht]
    \centering
    \includegraphics[width=0.4\linewidth]{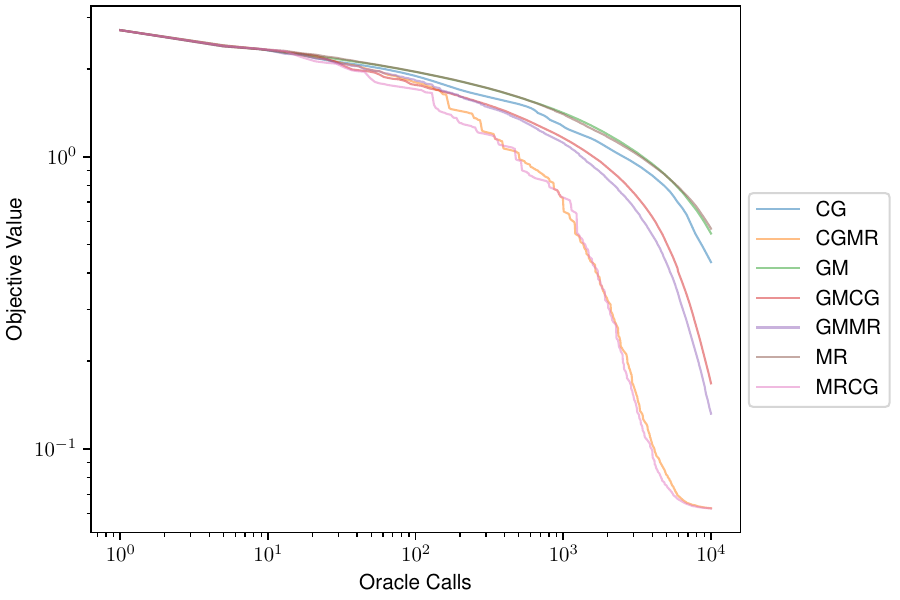}
    \caption{Comparison of scaling method performance on for the ResNet18 model on the Imagenette dataset}
    \label{fig:resnet_Imagenette_scaled}
\end{figure}

In \cref{fig:resnet_Imagenette_cg} and \cref{fig:resnet_Imagenette_mr_gnorm} we consider the CG and MR scalings in particular. Inspecting panel (b) of both figures and comparing with \cref{fig:resnet_Imagenette}, we see that the alternating MRCG scaling is capable of producing significantly larger scaling values, which could explain the faster convergence of the alternating scaling. In panel (a) of \cref{fig:resnet_Imagenette_mr_gnorm} we plot the gradient norm of the objective, we see that, similar to the MLP and logistic regression example, the MR scaling produces monotonic decrease in the gradient norm with the unit step size. This result is consistent with \cref{thm:local convergence MR second-order sufficient}. 

\begin{figure}[!ht]
    \centering
    \includegraphics[width=0.8\linewidth]{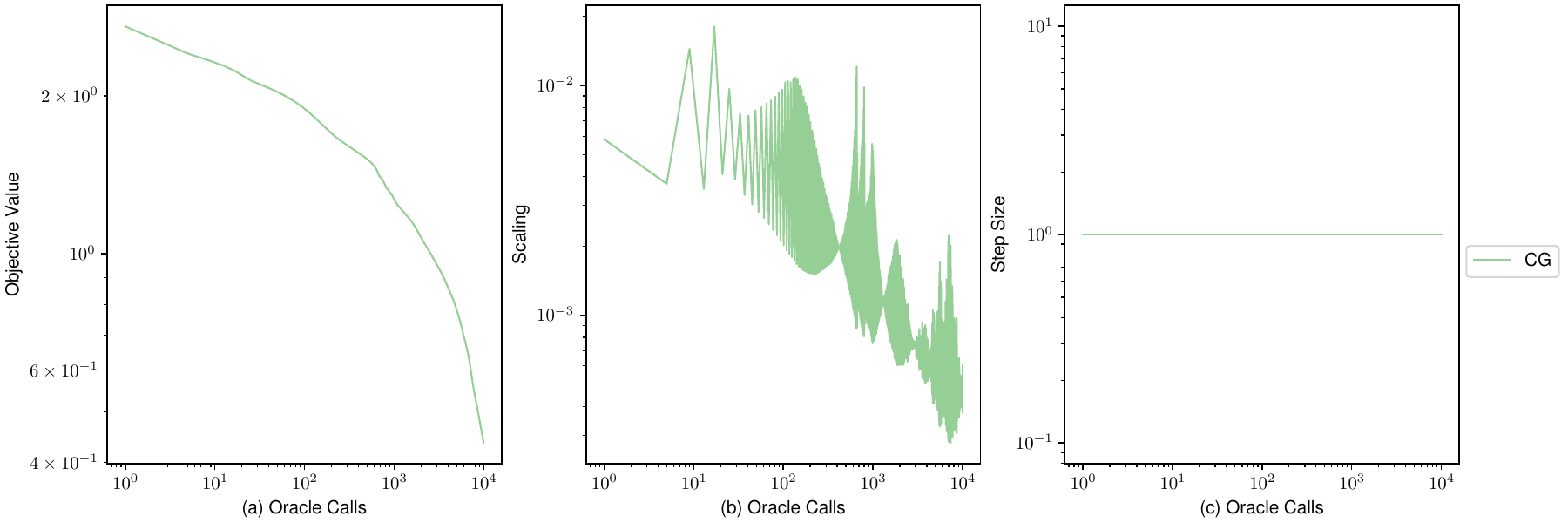}
        \caption{Performance of the CG scaling for the ResNet18 model on Imagenette dataset. (a) The objective function. (b) The scaling selected by the CG method.}
    \label{fig:resnet_Imagenette_cg}
\end{figure}

\begin{figure}[!ht]
    \centering
    \includegraphics[width=0.8\linewidth]{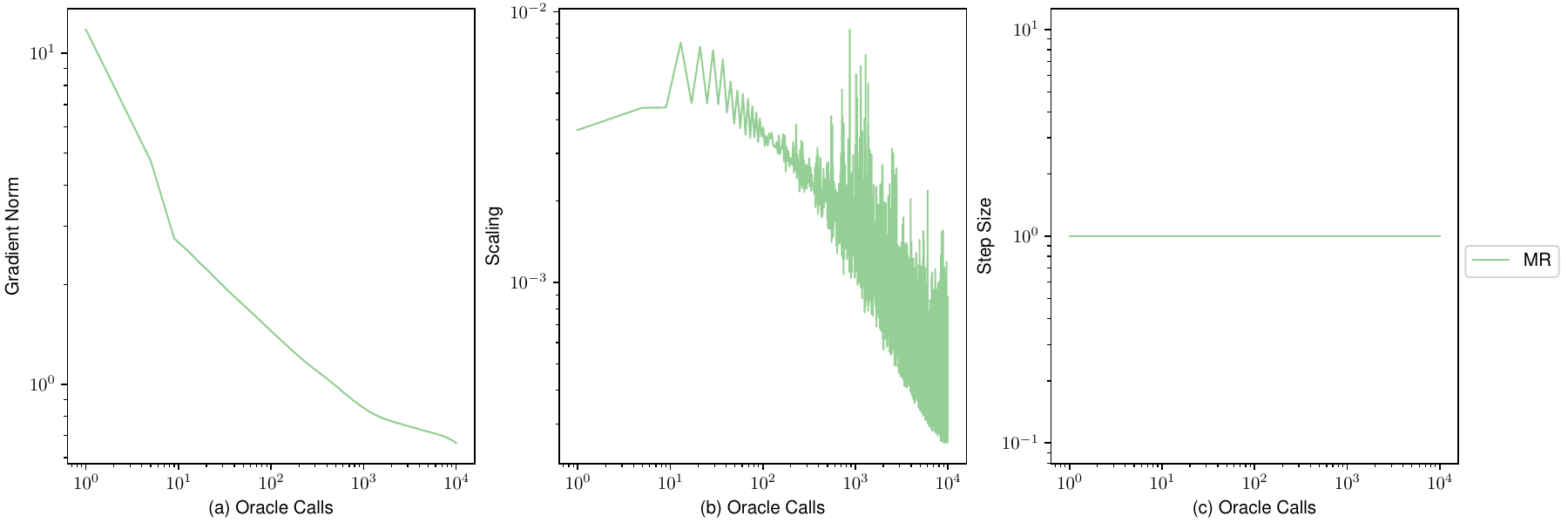}
    \caption{Performance of the MR scaling for the ResNet18 model on Imagenette dataset. (a) The objective function. (b) The scaling selected by the MR method.}
    \label{fig:resnet_Imagenette_mr_gnorm}
\end{figure}

In \cref{fig:resnet_Imagenette_time} we plot the results from \cref{fig:resnet_Imagenette} against wall-clock time. We see that the performance gap between wall clock and oracle calls is slightly worse relative to what was observed \cref{fig:logistic_regression_cifar10_time,fig:mlp_fashionMNIST_time}. This could be due to optimizations which are unavailable for second derivative information associated with more sophisticated convolutional architectures like ResNet.

\begin{figure}[!ht]
    \centering
    \includegraphics[width=0.4\linewidth]{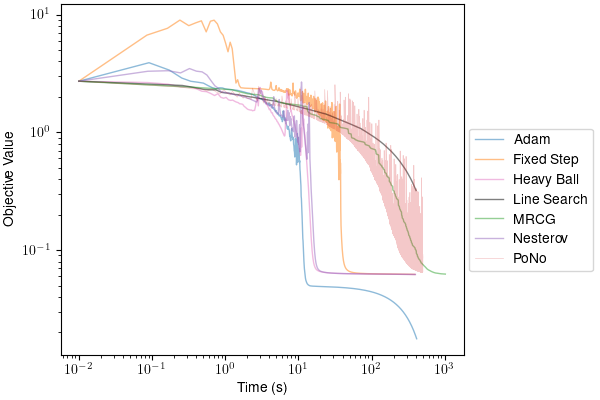}
    \caption{Wall-clock time results for the ResNet18 model on Imagenette dataset. $10^{-2}$ has been added to all times for plotting purposes.}
    \label{fig:resnet_Imagenette_time}
\end{figure}

\end{document}